\documentclass[12pt]{article}
\usepackage[utf8]{inputenc}
\usepackage{amsmath}
\usepackage{amssymb}
\usepackage[margin = 1.5 in, marginpar = 3 cm]{geometry}
\usepackage{graphicx}
\usepackage{amsthm}
\usepackage{mathtools}
\mathtoolsset{showonlyrefs}
\usepackage{hyperref}
\hypersetup{colorlinks,linkcolor={blue},urlcolor=cyan}

\usepackage{mleftright}

\usepackage{tikz}
\usepackage{theoremref}
\usepackage{wasysym}
\usepackage{harpoon}
\usepackage{float}
\usepackage{caption}
\usepackage{blindtext}
\usepackage{geometry}
\usepackage[title]{appendix}

\def\graphScale{0.29}
\def\largeGraphScale{0.5}
\def\specialVertexColor{white}

\title{Prime Graphs of Infinite Groups}
\author{Alexa Renner\footnote{arenner2@nd.edu; Department of Mathematics, Rose-Hulman Institute of Technology, 5500 Wabash Ave., Terre Haute, IN 47803, USA.}}
\DeclareMathOperator{\Aut}{Aut}

\DeclareMathOperator{\pgc}{\overline{\Gamma}}

\newtheorem{counter}{counter}[section]

\theoremstyle{definition}
\newtheorem{definition}[counter]{Definition}
\newtheorem*{definition*}{Definition}

\theoremstyle{plain}
\newtheorem{theorem}[counter]{Theorem}
\newtheorem{lemma}[counter]{Lemma}
\newtheorem{corollary}[counter]{Corollary}
\newtheorem{proposition}[counter]{Proposition}
\newtheorem{fact}[counter]{Fact}

\theoremstyle{remark}
\newtheorem{remark}[counter]{Remark}
\newcommand{\twoIsolatedpsl}{
\begin{tikzpicture}[scale=\graphScale]
    \draw (0, 0) -- (1, 0);
    \draw (0, 0) -- (0, 1);
    \draw (0, 1) -- (1, 0);
    \draw[fill=black] (0, 0) circle (0.125);
    \draw[fill=black] (0, 1) circle (0.125);
    \draw[fill=black] (1, 0) circle (0.125);
    \draw[fill=black] (1, 1) circle (0.125);
\end{tikzpicture}
}

\newcommand{\cricket}{
\begin{tikzpicture}[scale=\largeGraphScale]
    \draw (1, 0) -- (1, 1);
    \draw (1.5, 0.5) -- (1, 0);
    \draw (1.5, 0.5) -- (1, 1);
    \draw (1,0) -- (0,0);
    \draw (0,1) -- (1,0);
    \draw[fill=black] (0, 0) circle (0.125);
    \draw[fill=black] (0, 1) circle (0.125);
    \draw[fill=black] (1, 0) circle (0.125);
    \draw[fill=black] (1, 1) circle (0.125);
    \draw[fill=\specialVertexColor] (1.5, 0.5) circle (0.125);
\end{tikzpicture}
}
\newcommand{\fan}{
\begin{tikzpicture}[scale=\largeGraphScale]
    \draw (0, 0) -- (1, 0);
    \draw (0, 0) -- (0, 1);
    \draw (0, 1) -- (1, 1);
    \draw (0, 1) -- (1, 0);
    \draw (1,0) -- (1,1);
    \draw (1, 1) -- (1.5, 0.5);
    \draw (1, 0) -- (1.5, 0.5);
    \draw[fill=black] (0, 0) circle (0.125);
    \draw[fill=black] (0, 1) circle (0.125);
    \draw[fill=black] (1, 0) circle (0.125);
    \draw[fill=black] (1, 1) circle (0.125);
    \draw[fill=\specialVertexColor] (1.5, 0.5) circle (0.125);
\end{tikzpicture}
}
\newcommand{\fourVertexCompleteGraph}{
\begin{tikzpicture}[scale=\graphScale]
    \draw (0, 0) -- (1, 0);
    \draw (0, 0) -- (1, 1);
    \draw (0, 0) -- (0, 1);
    \draw (1, 0) -- (1, 1);
    \draw (1, 0) -- (0, 1);
    \draw (0, 1) -- (1, 1);
    \draw[fill=black] (0, 0) circle (0.125);
    \draw[fill=black] (0, 1) circle (0.125);
    \draw[fill=black] (1, 0) circle (0.125);
    \draw[fill=black] (1, 1) circle (0.125);
\end{tikzpicture}
}
\newcommand{\dart}{
\begin{tikzpicture}[scale=\largeGraphScale]
    \draw (0, 0) -- (1, 0);
    \draw (0, 0) -- (1, 1);
    \draw (1, 0) -- (1, 1);
    \draw (0, 1) -- (1, 1);
    \draw (1.5, 0.5) -- (1, 0);
    \draw (1.5, 0.5) -- (1, 1);
    \draw[fill=black] (0, 0) circle (0.125);
    \draw[fill=black] (0, 1) circle (0.125);
    \draw[fill=black] (1, 0) circle (0.125);
    \draw[fill=black] (1, 1) circle (0.125);
    \draw[fill=\specialVertexColor] (1.5, 0.5) circle (0.125);
\end{tikzpicture}
}
\newcommand{\northStar}{
\begin{tikzpicture}[scale=\largeGraphScale]
    \draw (0, 0) -- (1, 0);
    \draw (0, 0) -- (1, 1);
    \draw (1, 0) -- (1, 1);
    \draw (1.5, 0.5) -- (1, 0);
    \draw (1.5, 0.5) -- (1, 1);
    \draw[fill=black] (0, 0) circle (0.125);
    \draw[fill=black] (0, 1) circle (0.125);
    \draw[fill=black] (1, 0) circle (0.125);
    \draw[fill=black] (1, 1) circle (0.125);
    \draw[fill=\specialVertexColor] (1.5, 0.5) circle (0.125);
\end{tikzpicture}
}
\newcommand{\triangleWithShortTail}{
\begin{tikzpicture}[scale=\largeGraphScale]
    \draw (0, 1) -- (1, 1);
    \draw (1,0) -- (1,1);
    \draw (1, 1) -- (1.5, 0.5);
    \draw (1, 0) -- (1.5, 0.5);
    \draw[fill=black] (0, 0) circle (0.125);
    \draw[fill=black] (0, 1) circle (0.125);
    \draw[fill=black] (1, 0) circle (0.125);
    \draw[fill=black] (1, 1) circle (0.125);
    \draw[fill=\specialVertexColor] (1.5, 0.5) circle (0.125);
\end{tikzpicture}
}
\newcommand{\balloon}{
\begin{tikzpicture}[scale=\largeGraphScale]
    \draw (0, 0) -- (1, 1);
    \draw (0, 0) -- (0, 1);
    \draw (0, 1) -- (1, 1);
    \draw (1.5, 0.5) -- (1, 0);
    \draw (1.5, 0.5) -- (1, 1);
    \draw[fill=black] (0, 0) circle (0.125);
    \draw[fill=black] (0, 1) circle (0.125);
    \draw[fill=black] (1, 0) circle (0.125);
    \draw[fill=black] (1, 1) circle (0.125);
    \draw[fill=\specialVertexColor] (1.5, 0.5) circle (0.125);
\end{tikzpicture}
}
\newcommand{\happyFace}{
\begin{tikzpicture}[scale=\largeGraphScale]
    \draw (1, 0) -- (1, 1);
    \draw (1.5, 0.5) -- (1, 0);
    \draw (1.5, 0.5) -- (1, 1);
    \draw[fill=black] (0, 0) circle (0.125);
    \draw[fill=black] (0, 1) circle (0.125);
    \draw[fill=black] (1, 0) circle (0.125);
    \draw[fill=black] (1, 1) circle (0.125);
    \draw[fill=\specialVertexColor] (1.5, 0.5) circle (0.125);
\end{tikzpicture}
}
\newcommand{\codart}{
\begin{tikzpicture}[scale=\largeGraphScale]
    \draw (0, 0) -- (1, 1);
    \draw (0, 0) -- (0, 1);
    \draw (0, 1) -- (1, 1);
    \draw (1.5, 0.5) -- (1, 1);
    \draw[fill=black] (0, 0) circle (0.125);
    \draw[fill=black] (0, 1) circle (0.125);
    \draw[fill=black] (1, 0) circle (0.125);
    \draw[fill=black] (1, 1) circle (0.125);
    \draw[fill=\specialVertexColor] (1.5, 0.5) circle (0.125);
\end{tikzpicture}
}
\newcommand{\house}{
\begin{tikzpicture}[scale=\largeGraphScale]
    \draw (0, 0) -- (1, 0);
    \draw (0, 0) -- (1, 1);
    \draw (0, 0) -- (0, 1);
    \draw (1, 0) -- (1, 1);
    \draw (1, 0) -- (0, 1);
    \draw (0, 1) -- (1, 1);
    \draw (1.5, 0.5) -- (1, 0);
    \draw (1.5, 0.5) -- (1, 1);
    \draw[fill=black] (0, 0) circle (0.125);
    \draw[fill=black] (0, 1) circle (0.125);
    \draw[fill=black] (1, 0) circle (0.125);
    \draw[fill=black] (1, 1) circle (0.125);
    \draw[fill=\specialVertexColor] (1.5, 0.5) circle (0.125);
\end{tikzpicture}
}
\newcommand{\szthirtytwoTrickyGraph}{
\begin{tikzpicture}[scale=\graphScale]
    \draw (0, 0) -- (1, 0);
    \draw (0, 0) -- (0, 1);
    \draw (0, 1) -- (1, 0);
    \draw (1, 0) -- (1, 1);
    \draw[fill=black] (0, 0) circle (0.125);
    \draw[fill=black] (0, 1) circle (0.125);
    \draw[fill=black] (1, 0) circle (0.125);
    \draw[fill=black] (1, 1) circle (0.125);
\end{tikzpicture}
}
\newcommand{\SuzukiThirtyTwoElusiveGraph}{
\begin{tikzpicture}[scale=\graphScale]
    \draw (0, 0) -- (1, 0);
    \draw (0, 0) -- (0, 1);
    \draw (0, 0) -- (1, 1);
    \draw (1, 0) -- (1, 1);
    \draw[fill=black] (0, 0) circle (0.125);
    \draw[fill=black] (0, 1) circle (0.125);
    \draw[fill=black] (1, 0) circle (0.125);
    \draw[fill=black] (1, 1) circle (0.125);
\end{tikzpicture}
}
\newcommand{\trianglePlusIsolated}{
\begin{tikzpicture}[scale=\graphScale]
    \draw (0, 0) -- (1, 0);
    \draw (0, 0) -- (1, 1);
    \draw (1, 0) -- (1, 1);
    \draw[fill=black] (0, 0) circle (0.125);
    \draw[fill=black] (0, 1) circle (0.125);
    \draw[fill=black] (1, 0) circle (0.125);
    \draw[fill=black] (1, 1) circle (0.125);
\end{tikzpicture}
}
\newcommand{\notIsolatedVertexWhiteBad}{
\begin{tikzpicture}[scale=\graphScale]
    \draw (0, 0) -- (1, 0);
    \draw (0, 0) -- (0, 1);
    \draw (1, 0) -- (0, 1);
    \draw[fill=black] (0, 0) circle (0.125);
    \draw[fill=black] (0, 1) circle (0.125);
    \draw[fill=black] (1, 0) circle (0.125);
    \draw[fill=white] (1, 1) circle (0.125);
\end{tikzpicture}
}

\newcommand{\hamSandwich}{
\begin{tikzpicture}[scale=\graphScale]
    \draw (0, 0) -- (1, 0);
    \draw (0, 0) -- (0, 1);
    \draw (1, 0) -- (1, 1);
    \draw (1, 0) -- (0, 1);
    \draw (0, 1) -- (1, 1);
    \draw[fill=black] (0, 0) circle (0.125);
    \draw[fill=black] (0, 1) circle (0.125);
    \draw[fill=black] (1, 0) circle (0.125);
    \draw[fill=black] (1, 1) circle (0.125);
\end{tikzpicture}
}
\begin{document}
\maketitle
\begin{abstract}
    The prime graph of a finite group $G$ is the graph $\Gamma(G)$ with vertex set the set of prime divisors $\pi(G)$ of $|G|$ and an edge between vertices $p, q\in\pi(G)$ if and only if there exists an element $g\in G$ with order $o(g) = pq$. Given a finite nonabelian simple group $T$, a group $G$ is $T$-solvable if there exists a composition series of $G$ such that every composition factor is either abelian or isomorphic to $T$. In this paper, we introduce the prime graph of an infinite group, the graph $\Gamma(G)$ with vertex set $\pi(G) = \{o(g):g\in G\text{ and }o(g)\text{ is prime}\}$ and an edge between $p,q\in\pi(G)$ if and only if there exists an element $g\in G$ with order $pq$, and generalize several results on the prime graphs of finite solvable and $T$-solvable groups to results on the prime graphs of members of certain classes of infinite solvable and $T$-solvable groups. 
\end{abstract}
\hfill\break

\hfill\break
\section*{Introduction}
The prime graph of a finite group $G$, $\Gamma(G)$, is the graph with vertex set the prime divisors of $|G|$ such that given any two primes $p, q\in V(G)$, $p-q\in\pgc(G)$ if and only if $G$ contains an element of order $pq$. The prime graphs of finite groups have been well studied. The authors of \cite{2015REU} proved the following result:
\begin{theorem}\thlabel{threecolorold}
    Let $\Xi$ be a finite graph. $\Xi$ is isomorphic to the prime graph complement of a finite solvable group if and only if $\Xi$ is triangle free and 3-colorable.
\end{theorem}
This motivated the following definition:
\begin{definition}
    Given a collection $\mathcal{T}$ of nonabelian finite simple groups, a finite group $G$ is said to be $\mathcal{T}$\textbf{-solvable} if there exists a composition series of $G$ such that every composition factor is either isomorphic to an element of $\mathcal{T}$ or a cyclic group.
\end{definition}
As in \cite{2023REU}, if $\mathcal{T} = \{T\}$ for a single nonabelian finite simple group $T$ and $G$ is $\mathcal{T}$-solvable, we will call $G$ $T$-solvable. When studying $T$-solvable groups, the following definitions are helpful:
\begin{definition}
    A nonabelian finite simple group $T$ is called a $K_n$\textbf{-group} if $|T|$ has exactly $n$ distinct prime divisors.
\end{definition}
\begin{definition}
    Let $T$ be a nonabelian finite simple group. A finite group $G$ is \textbf{strictly} $T$\textbf{-solvable} if there exists a composition series of $G$ such that every composition factor is either isomorphic to an element of $T$ or a cyclic group and at least one such factor is isomorphic to $T$.
\end{definition}
The prime graphs of $T$-solvable groups have been studied extensively. In \cite{2022REU}, the prime graphs of $\mathcal{T}$-solvable groups where $\mathcal{T}$ contained only $K_3$-groups were completely classified. The authors of \cite{2023REU} extended this work to completely classify the prime graphs of $T$-solvable groups where $T$ is a $K_4$-group such that $T\not\cong \operatorname{Sz}(32), \operatorname{Sz}(8), \operatorname{PSL}(2, q)$ any $q$. In \cite{Suz}, the prime graphs of $\operatorname{Sz}(32), \operatorname{Sz}(8),$ and $\operatorname{PSL}(2, 2^5)$-solvable groups were classified, and in \cite{psl}, the authors classified the prime graphs of several $T$-solvable groups for $T = \operatorname{PSL}(2, q)$ for several small $q$, but notably not $q = 13$. Finally, the authors of \cite{alland} classified the prime graphs of $\operatorname{PSL}(2, 13)$-solvable groups and also proved a theorem on the realizability of certain prime graphs by groups with a certain structure.

In \cite{2015REU} and many of the papers that followed, prime graphs of certain classes are classified by their complements. We will continue to use this convention in this paper. In the following definitions, we use the notation $\mathbb{P}$ to denote the set of primes of $\mathbb{N}$.
\begin{definition}
    Let $G$ be a group, and let $\pi(G) = \{p\in\mathbb{P}:\text{ there exists a }g\in G\text{ such that }o(g) = p\}$. Let $\Gamma(G)$ be the graph with vertex set $\pi(G)$ and an edge between two vertices $p,q\in \pi(G)$ if and only if there exists an element of order $pq$ in $G$. We call $\Gamma(G)$ the \textbf{prime graph} of $G$.
\end{definition}
Notice that for a finite group $G$, this definition defines the same graph that the standard definition of a prime graph of a finite group does. Also notice that the trivial group is no longer recognizable by its prime graph, as the empty graph is also the prime graph of any torsion-free group. Given a subnormal series $1=N_0\unlhd\cdots\unlhd N_k = G$ of a group $G$, we call each successive quotient $N_{i+1}/N_i$ a factor. We modify the definition of $\mathcal{T}$-solvable to be more applicable to infinite groups.
\begin{definition}
    Given a collection $\mathcal{T}$ of nonabelian finite simple groups, a group $G$ is said to be $\mathcal{T}$\textbf{-solvable} if there exists a subnormal series of finite length of $G$ such that every factor is either isomorphic to an element of $\mathcal{T}$ or abelian.
\end{definition}
As in the finite case, if $\mathcal{T} = \{T\}$, we call a $\mathcal{T}$-solvable group $T$-solvable. We now generalize the definition of strictly $T$-solvable finite groups:
\begin{definition}
    Let $T$ be a nonabelian finite simple group. A group $G$ is strictly $T$-solvable if there exists a subnormal series of finite length of $G$ such that every factor is either isomorphic to an element of $T$ or is abelian and at least one such factor is isomorphic to $T$.
\end{definition}
In this paper, we classify the prime graph complements of several classes of groups which contain infinite members. To adequately describe these classes, we will need a few more definitions:
\begin{definition}
    A \textbf{locally torsion} solvable group is a solvable group $G$ such that for any finite $\pi\subseteq\pi(G)$, there exists a torsion $H\leq G$ with $\pgc(G)[\pi] = \pgc(H)[\pi]$.
\end{definition}
Locally torsion metabelian groups are defined similarly. We can extend this idea to $T$-solvable groups:
\begin{definition}
    We call a $T$-solvable group $G$ \textbf{locally} $T$-\textbf{torsion} if for every finite $\pi_0\supseteq\pi(T)$ with $\pi_0\subseteq\pi(G)$, there exists a torsion $T$-solvable group $H\leq G$ with $\pi_0\subseteq\pi(H)$ and $\pgc(H) = \pgc(G)[\pi(H)]$ such that if $G$ is strictly $T$-solvable, so is $H$. 
\end{definition}
We will often call a locally $T$-torsion group a locally torsion group if $T$ is clear from context. 
\begin{definition}
    We say a group $G$ has the \textbf{Finite Frobenius Property} if, for every prime $p\in\pi(G)$, there are only finitely many $q\in\pi(G)$ such that there is not an element of order $pq$ in $G$.
\end{definition}
Notice that a group $G$ has the finite Frobenius property if and only if $\pgc(G)$ has no vertices with infinite degree. 

In Section \ref{prelim}, we introduce the notation and conventions we will use throughout the paper. In Section \ref{lemmas}, we state and prove some lemmas that will be used extensively in the following sections.

In section \ref{solv} we extend \thref{threecolorold} to \thref{locallyTorsionSolvableThm1}:
\begin{theorem}
    Let $\Xi$ be a countable locally finite graph. $\Xi$ is triangle free and 3-colorable if and only if it is isomorphic to the prime graph complement of a (locally) torsion solvable group $G$ with the finite Frobenius property. 
\end{theorem}
We also extend the classification theorems for certain locally torsion and torsion $T$-solvable groups with the finite Frobenius property for certain $T$ in that section. Section \ref{subgraphlemmas} provides a few lemmas that will be used in the subsequent sections, and Section \ref{TTorsion} classifies the prime graph complements of many of the locally torsion and torsion $T$-solvable groups whose classification theorems could not be extended to infinite groups using only the tools in Section \ref{solv}. Finally, \ref{outlook} provides ideas for further work. Together, sections \ref{solv} and \ref{TTorsion} extend many of the classification theorems of the prime graph complements of finite $T$-solvable groups to locally torsion $T$-solvable groups.
\section{Preliminaries}\label{prelim}
\subsection{Notation and Conventions}
Throughout this paper we adopt the convention $0\in\mathbb{N}$.
\begin{itemize}
    \item For a graph $\Xi$, we write $V(\Xi)$ to denote the vertex set of $\Xi$. For two vertices $v,w\in V(\Xi)$ such that there exists an edge between them in $\Xi$, we write $w-v\in\Xi$.
    \item Given a graph $\Xi$ and $X\subseteq V(\Xi)$, we write $N[X]$ to denote the set of all $v\in V$ such that $v\in X$ or there exists $w\in X$ such that $v-w\in\Xi$ and call this the closed neighborhood of $X$. We denote the set $N[X]\setminus X$ by $N(X)$ and call this the open neighborhood of $X$.
    \item Given a group $G$ and an element $g\in G$, we write $o(g)$ to denote the order of $g$.
    \item Given a group $G$, we define $\pi(G)=\{p\in\mathbb{P}:\text{ there exists a }g\in G\text{ such that }o(g) =p\}$. Notice that this coincides with the standard definition of $\pi(G)$ when $G$ is finite.
    \item Given a sequence of sets $\{A_\alpha\}_{\alpha\in I},$ we denote the disjoint union of the sequence by $\bigsqcup_{\alpha\in I}A_\alpha$.
    \item Let $\{\Xi_\alpha\}_{\alpha\in I}$ be a sequence of labeled graphs. We define the \textbf{labeled graph product} $\sqcap_{\alpha\in I}\Xi_\alpha$ to have vertex set $\bigcup_{\alpha\in I}V(\Xi_\alpha)$ and edge set $E(K) \setminus(\bigcup_{\alpha\in I}E(\overline{\Xi_\alpha}))$ where $K$ is the complete graph on the set of vertices $X$. This definition extends the corresponding definition given in \cite{Suz}.
\end{itemize}
We will often prove results for torsion groups that also hold, using the same proof, for locally torsion groups. In such cases, we will write ``(locally) torsion" instead of ``torsion".

\section{Lemmas}\label{lemmas}
\begin{lemma}\thlabel{compact}
   [DeBrujin-Erdős Theorem]: Fix $k\geq 1$. A graph $\Xi$ is $k$-colorable if and only if every finite subgraph of $\Xi$ is $k$-colorable.
\end{lemma}

\begin{lemma}\thlabel{kOneColor}
    Fix $k\in\mathbb{N}^+$. Let $\Xi$ be a graph and let $X\subseteq \Xi$. Suppose every finite subgraph $\Delta\subseteq\Xi$ can be colored with $k$ colors such that $X\cap\Delta$ is monochromatic. Then, there exists a $k$-coloring of $\Xi$ in which $X$ is monochromatic.
\end{lemma}
\begin{proof}
    This follows from a straightforward compactness argument. The compactness theorem along with some background on mathematical logic can be found in \cite{marker}.
\end{proof}
\begin{remark}\thlabel{stolenAbelian}
    \cite{stolenStuff} A finitely generated torsion abelian group is finite.
\end{remark}
\begin{lemma}\thlabel{stolenLemma}
    \cite{stolenStuff}: Let $G$ be a finitely-generated torsion $\mathcal{T}-$solvable group such that $|\mathcal{T}|$ is finite. Then, $G$ is finite.
\end{lemma}
\begin{proof}
    The proof is analogous to the the proof for solvable groups. A version of the proof for solvable groups can be found in \cite{stolenStuff}.
\end{proof}
\begin{lemma}\thlabel{finitizer}
    Let $G$ be a torsion $\mathcal{T}$-solvable group such that $|\mathcal{T}|$ is finite and let $\Delta\subseteq\pgc(G)$ be a finite subgraph. There exists a finite $\mathcal{T}$-solvable subgroup $M\leq G$ such that $\pgc(M)[V(\Delta)] = \Delta$.
\end{lemma}
\begin{proof}
    Fix some subnormal series 
    \begin{align*}
        1=N_0\unlhd N_1\unlhd\cdots\unlhd N_k = G
    \end{align*}
    of $G$ such that for each $i\in\{0,\cdots, k-1\},$ $N_{i+1}/N_i\cong T$ for some $T\in\mathcal{T}$ or $N_{i+1}/N_i$ is abelian. Let $A = \bigcup_{T\in\mathcal{T}}\pi(T)\cap\pi(G)$. We define the following sets:

    Let $A_0 = \{i: N_{i+1}/N_i\text{ is abelian}\}$. For $T\in\mathcal{T}$, let $A_T = \{i:N_{i+1}/N_i\cong T\}$. Let $T\in\mathcal{T}$. For each $i\in A_T,$ let $H_{i,T}\subseteq N_{i+1}$ be a finite set of $|T|$ elements such that $\{hN_{i}:h\in H_{i,T}\}$ with the appropriate binary operation is isomorphic to $T$. Let $H = \bigcup_{T\in\mathcal{T}, i\in A_T} H_{i,T}$. For each edge $p-q\not\in\Delta$, let $g_{p, q}\in G$ be an element of order $pq$. Let $M = \langle h, g_{p, q}|h\in H, p-q\not\in\pgc(G)[\pi(T)]\rangle$. We claim that $M$ is $T$-solvable. Consider the subnormal series
    \begin{align*}
        1 = N_0\cap M \unlhd N_1\cap M\unlhd\cdots\unlhd N_k\cap M = M
    \end{align*}
    of $M$. We may write, for any $j$ with $1\leq j\leq k$:
    \begin{align*}
        (N_{j+1}\cap M)/(N_j\cap M) \cong (N_{j+1}\cap M)N_j/N_j\leq N_{j+1}/N_j.
    \end{align*}
   First suppose that $j\in A_T$ for some $T\in\mathcal{T}$. Then, $(N_{j+1}\cap M)/(N_j\cap M) \cong (N_{j+1}\cap M)N_j/N_j\leq N_{j+1}/N_{j}\cong T.$ For each $gN_{j}\in N_{j+1}/N_{j}$, there exists a $h\in H_{j, T}\subseteq N_{j}\cap M$ such that $hN_{j} = gN_{j}$, so because $H_{j, T}\subseteq N_{j+1}\cap M$, $(N_{j+1}\cap M)N_{j}/N_{j}\cong T$. Now consider the case where $j\not\in A_T$ for any $T\in\mathcal{T}$. This implies that $j\in A_0$, so $N_{j+1}/N_j$ is abelian. Then, $(N_{j+1}\cap M)/N_j\cap M$ is abelian. Thus, $M$ is $T$-solvable.

    We now show that $M$ has the desired prime graph complement. For any $p, q\in V(\Delta)$ such that $p-q\not\in\Delta$, $g_{p,q}\in M$, so $p-q\not\in\pgc(M)$. If $p-q\in\Delta$, there must exist no order $pq$ element of $G$, so because $M\leq G$, there is no order $pq$ element of $M$. Thus, $\Delta = \pgc(M)[V(\Delta)]$. Because $M$ is finitely generated, by \thref{stolenLemma}, $M$ is a finite $\mathcal{T}$-solvable group.
\end{proof}
\begin{lemma}\thlabel{hallAnalog}\thlabel{hallAnalogLocallyTorsion}
    Let $G$ be a locally torsion strictly $T$-solvable group, and let $\Delta$ be a finite subgraph of $\pgc(G)$. If $\pi(T)\subseteq V(\Delta)$, there exists a finite strictly $T$-solvable group $H\leq G$ such that $\pgc(H) = \pgc(G)[V(\Delta)]$.
\end{lemma}
\begin{proof}
    Let $M\leq G$ be a torsion $T$-solvable subgroup of $G$ with $V(\Delta)\subseteq\pi(M)$. By the proof of \thref{finitizer}, there exists a strictly $T-$solvable $K\leq G$ such that $\pgc(K)[V(\Delta)] = \pgc(G)[V(\Delta)]$. Let $H\leq K$ be a Hall $V(\Delta)$-subgroup guaranteed by \cite[Theorem 4.4]{Suz}. Then, $\pgc(H) = \pgc(G)[V(\Delta)] = \Delta$.
\end{proof}

\begin{lemma}\thlabel{kColorings}
   Fix $k\in\mathbb{N}_{\geq 3}$, and let $T$ be a finite nonabelian simple group. Suppose all finite $T$ solvable groups have $k$-colorable prime graph complements. If $G$ is a locally torsion $T$-solvable group, then $\pgc(G)$ is $k$-colorable.
\end{lemma}
\begin{proof}
    Let $\Delta\subseteq\pgc(G)$ be a finite subgraph of $\pgc(G)$. If $G$ is strictly $T$-solvable, additionally suppose that $\pi(T)\subseteq\Delta$. We consider two cases: The case where $G$ is strictly $T$-solvable and the case where $G$ is solvable.
    
    If $G$ is strictly $T$-solvable, then by \thref{hallAnalogLocallyTorsion}, there exists a strictly $T$-solvable group $H$ with $\pgc(H) = \Delta$. 
    
    If $G$ is solvable, then let $M\leq G$ be a torsion group such that $\pgc(M)[V(\Delta)] = \pgc(G)[V(\Delta)]$, and for each $p-q\not\in\Delta,$ let $g_{p,q}\in M$ be an element of order $pq$. For each $p\in V(\Delta)$, let $g_p\in M$ be an element of order $p$. Then, let $R= \langle g_{pq}:p-q\not\in\Delta\rangle$. Because $\Delta$ is finite, the generating set chosen for $R$ is finite, so by \thref{stolenLemma} and the fact that $R$ is a solvable torsion group, $R$ is finite. Let $H$ be a Hall $V(\Delta)$-subgroup of $R.$ Notice that $\pgc(H) =\pgc(R)[V(\Delta)] = \pgc(M)[V(\Delta)]$.

    By our assumptions on finite $T$-solvable groups, $\Delta$ is $k$-colorable. Thus, by \thref{compact}, $\pgc(G)$ is $k$-colorable.
\end{proof}
\begin{lemma}\thlabel{kColoringsMinusSubgrap}
   Fix $k\in\mathbb{N}^+$. Let $G$ be a locally torsion strictly $T$-solvable group such that, for any finite strictly $T$-solvable group $H$, $\pgc(H)\setminus\pgc(H)[\pi(T)]$ is $k$-colorable. Then, $\pgc(G)\setminus\pgc(G)[\pi(T)]$ is $k$-colorable.
\end{lemma}
\begin{proof}
    This follows from \thref{hallAnalog} and \thref{compact}.
\end{proof}
    
\begin{lemma}\thlabel{dll}
    Let $I$ be some directed set with ordering $\leq$. Let $\{G_i\}_{i\in I}$ be a directed system of finite groups such that $f_{i,j}$ is an isomorphism from $G_i$ to a subgroup of $G_j$. Then:
    \begin{align*}
        \pgc(\varinjlim G_i) = \sqcap_{i\in I}\pgc(G_i).
    \end{align*}
\end{lemma}
\begin{proof}
    Take the direct limit to get $G = \varinjlim G_i = \bigsqcup_{i\in I} G_i/\sim$ where $x_i\sim x_j$ for $x_i\in G_i, x_j\in G_j$ if $f_{i, k}(x_i) = f_{j,k}(x_j)$ for some $k\geq i,j$ and the group operation is given by, for $x_i\in G_i, x_j\in G_j$, $[x_i]\cdot[x_j] = [f_{i,k}(x_i)f_{j,k}(x_j)]$ where $k \geq i, j$. Notice that $[1]$ is the identity element of $\varinjlim G_i$.

     Suppose $p\in V\left(\sqcap_{i\in I}\pgc(G_i)\right)$. Then, there exists some $G_i$ for which there is an element $x_i\in G_i$ with order $p$. We may write:
    \[\
    [x_i]^p = [f_{i, k}(x)^p] = [1]
    \]
    for some $k\geq i$. Because $x_i\neq 1$ and the $f_{i, j}$ are isomorphisms, $f_{i,k}(1)\neq f_{j,k}(x)$ for any $j,i\leq k$, so $[x_i]\neq [1]$. Thus, $p\in V(\pgc(\varinjlim G_i))$. Now let $p\in  V(\pgc(\varinjlim G_i))$. Then, there exists an element $[x_i]\in G$ with order $p$, so $x_i\neq 1$. Notice that $f_{i, i}(x_i)^p \in[1]$, so $f_{i,i}(x_i)$ has order $p$. Because $f_{i,i}:G_i\to G_i$ is an isomorphism, $x_i$ has order $p$. Thus, $V\left(\sqcap_{i\in I}\pgc(G_i)\right) = V(G)$.

    Now let $p-q\not\in \sqcap_{i\in I}\pgc(G_i)$. Then, there exists an element $x_i$ of order $pq$ in $G_i$ for some $i$. We may write:
    \[
    [x_i]^{pq} = [f_{i, k}(x)^{pq}] = [1].
    \]
    Consider $f_{i,i}(x_i)^n$ for some $n<pq$. Notice that $f_{i,i}(x_i^n)=f_{i,i}(x_i)^n\neq 1$, as $f_{i,i}$ is an isomorphism. Then, $f_{i,k}(x_i^n) =f_{i, k}(x_i)^n\neq f_{j,k}(1)$ for any $k\geq i$ and $j\leq k$. Thus, $[x_i]^n = [f_{i,i}(x_i)]^n\neq [1]$ for all $n<pq$. Thus, $p-q\not\in\pgc(G)$. Now, suppose $p-q\not\in\pgc(G)$. There exists an element $[x_i]\in G$ with $x_i\in G_i$ for some $i$ such that the order of $[x_i]$ is $pq$. Then, $[x_i]^{pq} = [f_{i,i}(x_i)^{pq}] = [1]$, so $f_{i,i}(x)^{pq} = 1$. Notice that $[x_i]^n =[f_{i,i}(x_i)^n]\neq [1]$ for any $n<pq$, so $f_{i,i}(x_i)^n = f_{i,i}(x_i^n)\neq 1$. Because $f_{i,i}$ is an isomorphism, $x_i^n\neq 1$ for all such $n$, so the order of $x_i$ is $pq$. Thus, for all $j\geq i, p-q\not\in\pgc(G_i)$, so $p-q\not\in \sqcap_{i\in I} \pgc(G_i)$.
\end{proof}
\begin{lemma}\thlabel{dll2}
    Let $I$ be some directed set with ordering $\leq$. Let $\{G_i\}_{i\in I}$ be a directed system of finite groups along with embeddings $\{g_{i,k}\}_{i, j\in I,i\leq j}$ such that for each $i\in I$, there exists $N_i\unlhd G_i$ such that, for any $x\in G_i, k\geq i, g_{i, k}(x)\in N_k$ if and only if $x\in N_i$. Then, using the maps $g_{i,j}|_{N_i}, \varinjlim N_i\unlhd \varinjlim G_i$.
\end{lemma}
\begin{proof}
    First notice that because the $N_i$ are subgroups, $g_{i,j}|_{N_i}$ are well-defined maps. We first show that $\varinjlim N_i\leq \varinjlim G_i$. To do this, we first show that for any $[x]_N\in \varinjlim N_i$, $[x]_G\in \varinjlim G_i$. Notice that it is sufficient to show that $[x]_G\subseteq [x]_N$. Let $y\in [x]_G$, and notice that $y\in G_j, x\in N_i$ for some $i,j\in\mathbb{N}$. Then, $g_{i,k}(x) = g_{j, k}(y)$ for some $k\geq i,j$. By choice of $g_{j, k}, y\in N_j$. Thus, $g_{j,k}|_{N_j}(y) = g_{i,k}|_{N_i}(x)$, so $y\in [x]_N$. Thus, $[x]_G = [x]_N$.

    Notice that, by choice of $N_i$, the group operation of $\varinjlim G_i$ restricted to $\varinjlim N_i$ makes $\varinjlim N_i$ into a group.

    Now let $[x]\in\varinjlim N_i, [y]\in\varinjlim G_i$. Let $x\in N_i, y\in G_j$ for some $i,j\in I$, and let $k\geq i,j$. We may write:
    \begin{align*}
        [y]^{-1}[x][y] = [y^{-1}][x][y] = [g_{j,k}(y^{-1})g_{i,k}(x)g_{j,k}(y)] = [g_{j,k}(y)^{-1}g_{i,k}(x)g_{j,k}(y)].
    \end{align*}
    Because $g_{i,k}(x)\in N_k, g_{j,k}(y)^{-1}g_{i,k}(x)g_{j,k}(y)\in N_k$, so $[g_{j,k}(y)^{-1}g_{i,k}(x)g_{j,k}(y)]\in \varinjlim N_i$. Thus, $\varinjlim N_i\unlhd\varinjlim G_i$.
\end{proof}
\begin{lemma}\thlabel{dll3}
     Let $I$ be some directed set with ordering $\leq$. Let $\{G_i\}_{i\in I}$ be a directed system of finite groups along with embeddings $\{g_{i,k}\}_{i, j\in I,i\leq j}$ such that for each $i\in I$, there exists $N_i\unlhd G_i$ such that, for any $x\in G_i, k\geq i, g_{i, k}(x)\in N_k$ if and only if $x\in N_i$. If $G_i/N_i$ is abelian for each $i$, $\varinjlim G_i/\varinjlim N_i$ is abelian. If $N_i$ is abelian for each $i, \varinjlim N_i$ is abelian.
\end{lemma}
\begin{proof}
    By \thref{dll2}, $\varinjlim N_i\unlhd\varinjlim G_i$. First consider the case where each $N_i$ is abelian. Let $[x],[y]\in \varinjlim N_i$. Let $x\in N_i, y\in N_j$ for some $i,j\in I$, and let $k$ be an upper bound of $i$ and $j$. We may write:
    \begin{align*}
        [x][y] = [g_{i, k}(x)g_{j,k}(y)] = [g_{j,k}(y)g_{i,k}(x)] = [y][x].
    \end{align*}
    Thus, $\varinjlim N_i$ is abelian.

    Now consider the case where each $G_i/N_i$ is abelian. Let $[x]\varinjlim N_i, [y]\varinjlim N_i\in \varinjlim G_i/\varinjlim N_i$. Let $x\in G_i, y\in G_j$ and let $k\geq i,j$. We may write:
    \begin{align*}
        [x]\varinjlim N_i[y]\varinjlim N_i = [x][y]\varinjlim N_i = [g_{i,k}(x)g_{j,k}(y)]\varinjlim N_i.
    \end{align*}
    Notice that:
    \begin{align*}
        g_{i,k}(x)g_{j,k}(y)N_k = g_{j,k}(y)g_{i,k}(x)N_k
    \end{align*}
    so there exists $a\in N_k$ such that $g_{i,k}(x)g_{j,k}(y) = g_{j,k}(y)g_{i,k}(x)a$. Then:
    \begin{align*}
        [g_{i,k}(x)g_{j,k}(y)]\varinjlim N_i = [g_{j,k}(y)g_{i,k}(x)a]\varinjlim N_i=[y]\varinjlim N_i[x]\varinjlim N_i[a]\varinjlim N_i.
    \end{align*}
    Because $a\in N_k, [a]\varinjlim N_i = \varinjlim N_i$. Thus:
    \begin{align*}
        [y]\varinjlim N_i[x]\varinjlim N_i[a]\varinjlim N_i = [y]\varinjlim N_i[x]\varinjlim N_i.
    \end{align*}
    Thus, $\varinjlim G_i/\varinjlim N_i$ is abelian.
\end{proof}
\begin{lemma}\thlabel{dll4}
     Let $I$ be some directed set with ordering $\leq$ and let $T$ be some nonabelian finite simple group. Let $\{G_i\}_{i\in I}$ be a directed system of finite groups along with embeddings $\{g_{i,k}\}_{i, j\in I,i\leq j}$ such that for each $i\in I$, there exists $N_i\unlhd G_i$ such that, for any $x\in G_i, k\geq i, g_{i, k}(x)\in N_k$ if and only if $x\in N_i$. If $G_i/N_i\cong T$ for each $i$, $\varinjlim G_i/\varinjlim N_i\cong T$.
\end{lemma}
\begin{proof}
    Fix $i\in I$. We first show that, for each coset $[x]\varinjlim N_i\in \varinjlim G_i/\varinjlim N_i$, there is a $y\in G_i$ such that $[x]\varinjlim N_i = [y]\varinjlim N_i$. Let $x\in G_j$, and let $k\geq i,j$. 

    \hfill\break
    \textbf{Claim:} $\{g_{j, k}(s)N_k:s\in G_j\} = \{g_{i,k}(t)N_k:t\in G_i\}$

    \hfill\break
    \textbf{Proof:} Let $s_1, s_2\in G_j$ such that $s_1N_j\neq s_2N_j$. If $g_{j, k}(s_1)N_k = g_{j, k}(s_2)N_k$, then $g_{j, k}(s_1) = g_{j, k}(s_2)a$ for some $a\in N_k$. Then, $g_{j, k}(s_2^{-1}s_1) = a\in N_k$, so $s_2^{-1}s_1\in N_j$. This is a contradiction, as $s_1N_j\neq s_2N_j$. Thus, $g_{j, k}(s_1)N_k \neq g_{j, k}(s_2)N_k$. This implies that $|\{g_{j, k}(s)N_k:s\in G_j\}| = |T|$. An almost identical argument argument shows that $|\{g_{i,k}(t)N_k:t\in G_i\}| = T$. Because $\{g_{i,k}(t)N_k:t\in G_i\}, \{g_{j, k}(s)N_k:s\in G_j\}\subseteq G_k/N_k$ and $|G_k/N_k| = |T|$, we must have that $\{g_{i,k}(t)N_k:t\in G_i\}= \{g_{j, k}(s)N_k:s\in G_j\}$.
    \begin{flushright}
$\dashv_{\text{Claim}}$
\end{flushright}
    By the claim, $g_{j, k}(x)N_k = g_{i,k}(y)N_k$ for some $y\in G_i$, so $g_{j, k}(x) = g_{i,k}(y)a$ for some $a\in N_k$. This means that $[x]\varinjlim N_i =[g_{j, k}(x)]\varinjlim N_i = [g_{i,k}(y)a]\varinjlim N_i = [g_{i,k}(y)]\varinjlim N_i[a]\varinjlim N_i = [y]\varinjlim N_i$ as $a\in N_k$. We now claim that for every $[x]\varinjlim N_i$ with $x\in G_i$, for all $[y]\varinjlim N_i = [x]\varinjlim N_i$ with $y\in G_i, yN_i = xN_i$. For contradiction, suppose otherwise. Then there exist $x,y\in G_i$ with $xN_i\neq yN_i$ and $[y]\varinjlim N_i = [x]\varinjlim N_i$. Thus, $[y][a] = [x]$ for some $a\in N_j$ and some $j\in I$. This implies that $[x]^{-1}[y][a] = [1]$, so there exists $k\geq i,j$, $[g_{i,k}(x)^{-1}g_{i,k}(y)g_{j,k}(a)]= [1]$. Then, $g_{i,k}(x)^{-1}g_{i,k}(y)g_{j,k}(a) = 1$, so $g_{j,k}(a) = g_{i,k}(y^{-1})g_{i,k}(x) = g_{i,k}(y^{-1}x)$. Notice $a\in N_j$, so $g_{j,k}(a)\in N_k$. This implies $y^{-1}x\in N_i$, so $xN_i = yN_i$, a contradiction.

    By the claim and the above proof, we have shown that, for each $[x]\varinjlim N_i\in \varinjlim G_i/\varinjlim N_i$, there is exactly one $x_iN_i\in G_i/N_i$ such that $[x_i]\varinjlim N_i = [x]\varinjlim N_i$. Let $\varphi$ be the isomorphism mapping $G_i/N_i$ to $T$.  
    Define $\overline{\varphi}([x]\varinjlim N_i)$ to be $\varphi(x_iN_i)$ for the $x_iN_i\in G_i/N_i$ such that $[x]\varinjlim N_i = [x_i]\varinjlim N_i$. Notice that, by the above, this function is well-defined. We claim that it is an isomorphism. Let $[x]\varinjlim N_i,[y]\varinjlim N_i\in \varinjlim G_i/\varinjlim N_i$ with representatives $x_iN_i, x_jN_i$ respectively. We may write:
    \begin{align*}
        \overline{\varphi}([x]\varinjlim N_i)\overline{\varphi}([y]\varinjlim N_i) = \varphi(x_iN_i)\varphi(x_jN_i) = \varphi(x_ix_jN_i).
    \end{align*}
    and:
    \begin{align*}
        [x]\varinjlim N_i[y]\varinjlim N_i = [x_i]\varinjlim N_i[x_j]\varinjlim N_i = [x_i][x_j]\varinjlim N_i = [g_{i,i}(x_i)g_{i,i}(x_j)]\varinjlim N_i\\
        =[g_{i,i}(x_ix_j)]\varinjlim N_i=[x_ix_j]\varinjlim N_i.
    \end{align*}
    Thus:
    \begin{align*}
        \varphi(x_ix_jN_i) = \overline{\varphi}([x_ix_j]\varinjlim N_i) = \overline{\varphi}([x]\varinjlim N_i[y]\varinjlim N_i).
    \end{align*}
    Thus, $\overline{\varphi}$ is a homomorphism. Now let $a\in T$. Then, there exists an $x_i\in G_i$ with $\varphi(x_iN_i) = a$, so $\overline{\varphi}(x_i\varinjlim N_i) = a$, and $\overline{\varphi}$ is surjective. For contradiction, suppose $\overline{\varphi}$ is not injective. Then, there exists a $[x]\varinjlim N_i\neq \varinjlim N_i$ such that $\overline{\varphi}([x]\varinjlim N_i) = 1$. Thus, the representative of $[x]\varinjlim N_i$ must be $bN_i$ for some $b\in N_i$, so $[x]\varinjlim N_i = [b]\varinjlim N_i = \varinjlim N_i$, a contradiction. Thus, $\varinjlim G_i/\varinjlim N_i\cong T$. 
\end{proof}

\section{Solvable Results}\label{solv}
\begin{proposition}\thlabel{backwardsDirectionTorsion}
    Let $K$ be a locally torsion solvable group. Then, $\pgc(K)$ is triangle-free and 3-colorable.
\end{proposition}
\begin{proof}
    Let $\Delta$ be a finite induced subgraph of $\pgc(K)$. By definition, there exists a solvable torsion group $G\leq K$ with $\pgc(G)[V(\Delta)] = \pgc(K)[V(\Delta)] = \Delta$. For each $p\in V(\Delta)$, let $g_p\in G$ be an element with order $p$ and consider $A = \langle g_p|p\in V(\Delta)\rangle$. Because $A\leq G$ and $G$ is solvable, $A$ is solvable. $A$ is finitely generated, so by \thref{stolenLemma}, $A$ is a finite solvable group. By \cite{2015REU}, $\pgc(A)$ is triangle-free and 3-colorable. Because $\Delta\subseteq\pgc(A)$, as $\pi(A)\supseteq V(\Delta)$, $\Delta$ is triangle-free and 3-colorable. By \thref{compact}, $\pgc(K)$ is 3-colorable. Because no finite subgraph of $\pgc(K)$ contains a triangle, $\pgc(K)$ is triangle-free.
\end{proof}
\begin{proposition}\thlabel{labelingsolvable}
    Let $\Xi$ be a locally finite graph such that every finite subgraph of $\Xi$ is triangle-free and 3-colorable. Then, there exists a labeling and orientation $\Xi'$ of $\Xi$ such that if $p\to q\in\Xi',$ $q\equiv 1\mod p$.
\end{proposition}
\begin{proof}
    As is proved in \cite[Theorem 2.8]{2015REU}, every finite triangle-free and 3-colorable graph has an acyclic orientation such that the orientation contains no directed 3-paths. Because every finite subgraph of $\Xi$ is triangle-free and 3-colorable, by \thref{compact}, $\Xi$ is triangle-free, 3-colorable, and has an acyclic orientation $\overrightarrow{\Xi}$ that contains no directed 3-paths. Notice that, in $\overrightarrow{\Xi}$, the coloring that colors every vertex with a zero in-degree and nonzero out degree $\mathcal{O}$, every vertex with a nonzero in-degree and a nonzero out-degree $\mathcal{D}$, and every vertex with a zero out-degree $\mathcal{I}$  is a valid 3-coloring (it is immediate from the definitions that no vertices both colored $\mathcal{O}$ or $\mathcal{I}$ are adjacent. If two vertices colored $\mathcal{D}$ are adjacent, this contradicts the fact that $\overrightarrow{\Xi}$ contains no directed 3-paths by the definition of $\mathcal{D}$). Enumerate the vertices of $\Lambda$ by $c_0, c_1, c_2,\dots$.
    
    We describe three algorithms. In each algorithm, $x$ is a vertex of $\Xi$ and $R$ is some finite (possibly empty) set of pairs $(r, p)$ for a vertex $r$ and a prime $p$.
    
\hfill\break
$\mathcal{O}$-\textbf{algorithm:} On input $\langle x, R\rangle$: 
\begin{itemize}
    \item[1.] Find the first prime such that $(v,)p\not\in R$ for any $v\in V(\Xi)$, return $(x, p)$.
\end{itemize}
\hfill\break
$\mathcal{D}$-\textbf{algorithm:} On input $\langle x, R\rangle$:
\begin{itemize}
    \item[1.] Check if there exists a prime $p$ such that $(x, p)\in R$. If so, halt and return $R$. Otherwise, move to step 2.
    \item[2.] Find the set $N_{\uparrow}^1(x)$ and enumerate it as $x_1,\dots, x_n$.
    \item[3.] Let $A = \varnothing$. For each $i, 1\leq i\leq n$:
    \begin{itemize}
        \item[a.] If $(x_i, p)\in R$ for some prime $p$, add $(x_i, p)$ to $A$.
        \item[b.] If there exists no prime $p$ such that $(x_i, p)\in R$, choose some prime $p$ such that $(y, p)\not\in R$ for any $y$, and add $(x_i, p)$ to $A$.
    \end{itemize}
    \item[4.] Find a prime $p$ such that $p\equiv1\mod p_i$ for each prime $p_i$ such that $(x_i, p_i)\in A_n$.
    \item[5.] Return $A_n\cup\{(x, p)\}$.
\end{itemize}

\hfill\break
$\mathcal{I}$-\textbf{algorithm:} On input $\langle x, R\rangle$:
\begin{itemize}
    \item[1.] Check if there exists a prime $p$ such that $(x, p)\in R$. If so, halt and return $R$. Otherwise, move to step 2.
    \item[2.] Find the set $N_{\uparrow}^1(x)\cup N_{\uparrow}^2(x)$ and enumerate it as $x_1,\dots, x_n$.
    \item[3.] Let $A_0 = \varnothing$. For each $i, 1\leq i\leq n$:
    \begin{itemize}
        \item[a.] If $(x_i, p)\in R$ for some prime $p$, let $A_i = A_{i-1}\cup\{(x_i, p)\}$.
        \item[b.] If there exists no $p$ such that $(x_i, p)\in A_{i-1}\cup R$ and $x_i$ is colored $\mathcal{O}$, choose some prime $p$ such that $(y, p)\not\in R\cup A_{i-1}$ for any $y$, and let $A_i = A_{i-1}\cup \{(x_i, p)\}$.
        \item[c.] If there exists no $p$ such that $(x_i, p)\in A_{i-1}$ and $x_i$ is colored $\mathcal{D}$, run the $\mathcal{D}$-algorithm on $\langle x_i, A_{i-1}\cup R\rangle$ and let $A_i$ be the set the algorithm returns.
    \end{itemize}
    \item[4.] Enumerate the primes in the second coordinates of elements of $A_n$ by $p_1,\dots, p_n$. Find a prime $p$ such that $p\equiv1\mod \prod_{i=1}^n p_i$.
    \item[5.] Return $A_n\cup\{(x, p)\}$.
\end{itemize}
We now define a nested sequence of sets $\{A_i\}_{i\in\mathbb{N}}$ recursively as follows:

\hfill\break
At step 0, if $c_0$ is colored $\mathcal{O},\mathcal{D},$ or $\mathcal{I}$, follow a path to a vertex in $\mathcal{I}$ and apply algorithm $\mathcal{I}$. Let $A_0$ be what the algorithm returns.

\hfill\break
Assume $A_i$ has been defined for some $i\in\mathbb{N}$. If $c_{i+1}$ is colored $\mathcal{O},\mathcal{D},$ or $\mathcal{I}$ follow a path to a vertex in $\mathcal{I}$ and apply the appropriate algorithm to $\langle c_{i+1}, A_i\rangle$ and let $A_{i+1}$ be what it returns.

Notice that, by the finite degree assumption and Dirichlet's theorem on arithmetic progressions, the algorithms always halt and each $A_i$ is finite. Let $A = \bigcup_{i\in\mathbb{N}} A_i$. Then, for every vertex $c_i, c_j$ of $\Xi$, there exist primes $p_i, p_j$ such that $(c_i,p_i),(c_j,p_j)\in A$ and $p_i = p_j$ if and only if $i = j$. Let $\Xi_L$ denote the corresponding labeled graph. Then, by construction, if $p\to q\in \Xi'$, $q\equiv1\mod p$.
\end{proof}

\begin{theorem}\thlabel{locallyTorsionSolvableThm1}
    Let $\Xi$ be a countable locally finite graph. $\Xi$ is triangle free and 3-colorable if and only if it is isomorphic to the prime graph complement of a (locally) torsion solvable group $G$ with the finite Frobenius property. 
\end{theorem}
\begin{proof}
    First assume that $\Xi$ is triangle free and 3-colorable. Fix the labeling guaranteed by \thref{labelingsolvable}, let $\mathcal{O}$ be the set of vertices with nonzero out degree and zero in degree, $\mathcal{I}$ the set of vertices with out-degree zero, and $\mathcal{D}$ the set of vertices with both in- and out-degree nonzero. Let the $A_i$ be as defined at the end of the proof of \thref{labelingsolvable}, and let $U_i$ be the set of primes $p$ such that there exists some vertex $v$ with $\langle v,p\rangle\in A_i$. For each $U_i$, notice that $\pgc(G)[U_i]$ is finite, triangle free and 3-colorable. Now, define $O_i$ to be the set of primes of $U_i$ that label vertices in $\mathcal{O}$, $D_i$ to be the set of primes of $U_i$ that label vertices in $\mathcal{D}$, and $I_i$ to be the set of primes of $U_i$ that label vertices in $\mathcal{I}$. We now define a sequence of groups:

    Let $n\in\mathbb{N}$. For each $p_i\in O_n$ and each $q_j\in D_n$ define $\gamma_{i, j}:C_{p_i}\to \operatorname{Aut}(C_{q_j})$ such that $\gamma_{i, j}(x)$ fixes no points for all non-identity elements $x\in C_{p_i}$ if $p_i-q_j\in\Xi_L[U_n]$ and such that $\gamma_{i, j}$ is the identity automorphism otherwise. This is possible by \thref{labelingsolvable}. Enumerate the primes of $O_n$ by $p_{n,1},\cdots, p_{n,u_n}$ Now, define $P_n = C_{p_{n,1}}\times\cdots\times C_{p_{n,u_n}}$. Enumerate the primes of $D_n$ by $q_{n,1},\cdots, q_{n, m_n}$ and define $Q_n = C_{q_{n,1}}\times\cdots\times C_{q_{n,m_n}}$. Define an action $\phi_n:P_n\to\operatorname{Aut}(Q_n)$ by 
    \begin{align*}
        \phi_n(x_1,\cdots, x_{u_n})(y_1,\cdots, y_{m_n}) = (\gamma_{1, 1}(x_1)\circ\cdots\circ\gamma_{u_n, 1}(x_n)(y_1),\cdots, \gamma_{1, m_n}(x_1)\circ\cdots\circ\gamma_{u_n, m_n}(x_n)(y_{m_n})).
    \end{align*}
    We claim that $\phi_n$ is actually a homomorphism. Let $(x_1,\cdots, x_{u_n}),(z_1,\cdots, z_{u_n})\in P_n$. Then, we may write:
    \begin{align*}
        \phi_n(x_1z_1,\cdots, x_{u_n}z_{u_n})(y_1,\cdots, y_{m_n})\\
        = ((\gamma_{1, 1}(x_1z_1)\circ\cdots\circ\gamma_{u_n, 1}(x_{u_n}z_{u_n}))(y_1),\cdots, (\gamma_{1, m_n}(x_1z_1)\circ\cdots\circ\gamma_{u_n, m_n}(x_{u_n}z_{u_n}))(y_{m_n}))\\
        = ((\gamma_{1, 1}(x_1)\circ\gamma_{1,1}(z_1)\circ\cdots\circ\gamma_{u_n, 1}(x_{u_n})\circ\gamma_{u_n,1}(z_{u_n}))(y_1),\cdots,\\
        (\gamma_{1, m_n}(x_1)\circ\gamma_{1,m_n}(z_1)\circ\cdots\circ\gamma_{u_n, m_n}(x_{u_n})\circ\gamma_{u_n,m_n}(z_{u_n}))(y_{m_n}))\\
        =(((\gamma_{1, 1}(x_1)\circ\cdots\circ\gamma_{u_n, 1}(x_{u_n}))\circ(\gamma_{1,1}(z_1)\circ\cdots\circ\gamma_{u_n,1}(z_{u_n})))(y_1),\cdots,\\
        ((\gamma_{1, m_n}(x_1)\circ\cdots\circ\gamma_{u_n, m_n}(x_{u_n}))\circ(\gamma_{1,m_n}(z_1)\circ\cdots\circ\gamma_{u_n,m_n}(z_{u_n})))(y_{m_n}))\\
        =(\phi_n(x_1,\cdots, x_{u_n})\circ\phi_n(z_1,\cdots, z_{u_n}))(y_1,\cdots, y_{m_n})
    \end{align*}
    as $\operatorname{Aut}(C_{q_{n,i}})$ is abelian for each $i$. We now claim that $\pgc(K^{(n)})$ realizes $\Xi_L[O_n\cup D_n]$. Notice immediately that $V(\pgc(K^{(n)})) = V(\Xi_L[O_n\cup D_n])$ by construction. Notice that if $p_i - q_j\not\in\Xi_L[O_n\cup D_n]$, then $p_i - q_j\not\in \pgc(K^{(n)})$ by construction. Now suppose $p_i - q_j\in\Xi_L[O_n\cup D_n]$. Because elements of order $p_i$ in $K^{(n)}/Q_n\cong P_n$ fix no elements of order $q_i$ in $K^{(n)}$ by construction of $K^{(n)}$, $p_i-q_i\in\pgc(K^{(n)})$ by \cite[Corollary 1.5]{Suz}. Now that the $K^{(n)}$ are defined, we may consider elements of $I_n$. We proceed recursively.

Notice that $I_0 = \varnothing$. Thus, define $H_0 = 1$.

    Suppose $H_n$ and the associated morphisms have been defined for some $n\geq 0$. Enumerate $I_{n+1}$ as $r_0,\cdots, r_{w_{n+1}}$ such that the $k_{n+1}$ elements in $\pi(H_n)$ are in $r_0,\cdots,r_{k_{n+1}}$ and such that $r_{d}$ in $I_{n+1}$ is $r_d$ in $I_n$ for $d<k_{n+1}$. For each $1\leq d\leq w_{n+1}$:

    Consider $K^{(n+1)}$. Let $W_{n+1,r_d} = N_1^\uparrow(r_d)\cup N_2^\uparrow(r_d)$. If $d<k_{n+1}$, notice that $W_{n+1,r_d} = W_{n,r_d}$ by the construction given in \thref{labelingsolvable}. Enumerate $O_{r_d,n+1}$ as $p_{r_d,n+1,1},\cdots,p_{r_d,n+1,u_{n+1}}$ where $p_{r_d,n+1,i}$ is $p_{r_d,n,i}$ if $p_{r_d,n+1,i}\in O_{r_d,n}$ and $D_{r_d,n+1}$ as $q_{r_d,n+1,1},\cdots,q_{r_d,n+1,m_{n+1}}$ where $q_{r_d,n+1,j}$ is $q_{r_d,n,j}$ if $q_{r_d,n+1,j}\in D_{r_d,n}$. For each $1\leq i\leq u_{n+1}$, define $Y_{r_d,n+1,i}$ as $C_{p_{r_d,n+1,i}}$ and $V_{r_d,n+1,i}$ as 1 if $p_{r_d,n+1,i}\in W_{r_d,n+1}$ and $Y_{r_d,n+1,i}$ as 1 and $V_{r_d,n+1,i}$ as $C_{p_{r_d,n+1,i}}$ otherwise. For each $1\leq j\leq m_{n+1}$, define $X_{r_d,n+1,j}$ as 
    $C_{q_{r_d,n+1,j}}$ and $Z_{r_d,n+1,j}$ as 1 if $q_{r_d,n+1,j}\in W_{r_d,n+1}$ and  $X_{r_d,n+1,j}$ as 
    1 and $Z_{r_d,n+1,j}$ as $C_{q_{r_d,n+1,j}}$ otherwise. Let $L_{r_{r_d}}^{(n+1)}$ be $(Y_{r_d,n+1,1}\times\cdots\times Y_{r_d,n+1, u_{n+1}})\ltimes(X_{r_d,n+1,1}\times\cdots\times X_{r_d,n+1, m_{n+1}})$ and let $M_{r_d}^{(n+1)}$ be $(V_{r_d,n+1,1}\times\cdots\times V_{r_d,n+1, u_{n+1}})\ltimes(Z_{r_d,n+1,1}\times\cdots\times Z_{r_d,n+1, m_{n+1}})$ where $\ltimes$ is defined by the appropriate restriction action. Notice that $L_{r_d}^{(n+1)}, M_{r_d}^{(n+1)}\leq K^{(n+1)}$. We claim that $M_{r_d}^{(n+1)}\unlhd K_{r_d}^{(n+1)}$. To see this, we first notice that the $C_{r_d,n+1, p_i}$ and $C_{r_d,n+1, q_j}$ can be relabeled by an isomorphism (call it $\eta_{n+1,r_d}$) if necessary so that 
    \begin{align*}
        L_{r_d}^{(n+1)} = (C_{r_d,n+1,p_1}\times\cdots\times C_{r_d,n+1,p_{s_{r_d,n+1}}}\times 1\times\cdots\times 1)\\
        \ltimes(C_{r_d,n+1,q_1}\times\cdots\times C_{0,q_{t_{r_d,n+1}}}\times 1\times\cdots\times 1)\\
        M_{r_d}^{(n+1)} = (1\times\cdots\times1\times C_{r_d,n+1,p_{s_{r_d,n+1}}}\times\cdots\times C_{r_d,n+1,p_{u_{n+1}}})\\
        \ltimes(1\times\cdots\times1\times C_{r_d,n+1,q_{t_{r_d,n+1}+1}}\times\cdots\times C_{r_d,n+1,q_{m_{n+1}}}).
    \end{align*}
    Now, notice that it suffices to reduce to the case of proving that $L_{r_d}^{(n+1)}\leq N_{K^{(n+1)}}(M_{r_d}^{(n+1)})$. To see this, notice that $M_{r_d}^{(n+1)}\leq N_{K^{(n+1)}}(M_{r_d}^{(n+1)})$. Let $((x_1,\cdots, x_{m_{n+1}}), (y_1,\cdots, y_{u_{n+1}}))\in K^{(n+1)}$. Notice that:
    \begin{align*}
        ((x_1,\cdots, x_{m_{n+1}}), (y_1,\cdots, y_{u_{n+1}})) \\
        = ((x_1,1,\cdots,1),1)\cdots((1,\cdots,1,x_{m_{n+1}}),1)(1,(y_0,\cdots,1))\cdots(1,(1,\cdots,1,y_{u_{n+1}}))\in N_{K^{(n+1)}}(M_{r_d}^{(n+1)}).
    \end{align*}
    From there, we claim that it suffices to reduce to the case where any Sylow $q_i$-subgroup of $M_{r_d}^{(n+1)}$ or a $p_j$-subgroup of the form $\langle (1, x)\rangle$ (call the Sylow subgroup $S$) has $L_{r_d}^{(n+1)}\leq N_{K^{(n+1)}}(S)$. To see this, let $((1,\cdots,1,x_{s_{r_d, n+1}+1},\cdots, x_{m_{n+1}}),(1,\cdots,1,y_{t_{r_d,n+1}+1},\cdots, y_{u_{n+1}}))\in M_{r_d}^{(n+1)}$ and let $g\in K^{(n+1)}$. Notice that:
    \begin{align*}
        ((1,\cdots,1,x_{s_{r_d, n+1}+1},\cdots, x_{m_{n+1}}),(1,\cdots,1,y_{t_{r_d,n+1}+1},\cdots, y_{u_{n+1}}))^g\\
        =((1,\cdots,1,x_{s_{r_d},n+1}+1\cdots,1),1)^g\cdots((1,\cdots,1,x_{m_{n+1}}),1)^g(1,(1,\cdots,1,y_{t_{r_d,n+1}+1},\cdots,1))^g\cdots\\
        \cdot(1,(1,\cdots,1,y_{u_{n+1}}))^g\in M_{r_d}^{(n+1)}
    \end{align*}
    as each $S\leq M_{r_d}^{(n+1)}$. Finally, we claim that we may reduce to the case of showing that for any Sylow subgroup $S$ of $M_{r_d}^{(n+1)}$ as described above, that it suffices to show that every Sylow $q_j$-subgroup of $L_{r_d}^{(n+1)}$ or a $p_i$-subgroup of $L_{r_d}^{(n+1)}$ of form $\langle (1, x)\rangle$ (call the given Sylow subgroup $T$) is in $N_{K^{(n+1)}}(S)$. To see this, let $((x_1,\cdots,x_{s_{r_d,n+1}},1,\cdots, 1),(y_1,\cdots, y_{t_{r_d,n+1}},1,\cdots, 1))\in L_{r_d}^{(n+1)}$. We may write:
    \begin{align*}
        ((x_1,\cdots,x_{s_{r_d,n+1}},1,\cdots, 1),(y_1,\cdots, y_{t_{r_d,n+1}},1,\cdots, 1))\\
        =((x_1,1,\cdots,1),1)((1,\cdots,1,x_{s_{r_d,n+1}},1,\cdots,1),1)(1,(y_1,1,\cdots,1))\\
        \cdot(1,(1,\cdots,1,y_{t_{r_d,n+1}},1,\cdots,1))\in N_{K^{(n+1)}}(S)
    \end{align*}
    so $L_{r_d}^{(n+1)}\leq N_{K^{(n+1)}}(S)$ for each $S$. Thus, it suffices to show that every such $T$ normalizes every such $S$.

    Let $S$ and $T$ be Sylow subgroups of the appropriate subgroups and form. We consider a few cases.

\hfill\break
\textbf{Case 1:} $S$ is a Sylow $q_i$-group. We show that all Sylow $q_i$-groups are normal in $K^{(n+1)}$. To see this, let $A$ be a Sylow $q_i$-group. Then, because $A$ is cylic, suppose $(x,y)$ generates $A$. Notice that the second component of $(x,y)^{q_i}$ is $y^{q_i}$, which must be 1, but the only element of $P_{n+1}$ that can be raised to $q_i$ to get 1 is 1. Thus, $(x,y) = (x,1)$. Then, $(x,1)^{q_i} = (x^{q_i},1)$. Because $Q_{n+1}$ is abelian, $x = (1,\cdots,y,\cdots,1)$ where the only nontrivial element of the tuple, $y$, is in position $i$.

\hfill\break
\textbf{Case 2:} $T$ is a Sylow $p_j$-group and $S$ is a Sylow $p_i$-group of the desired form. Notice that elements of $S$ are of the form $(1,(1,\cdots, x,\cdots, 1))$ and elements of $T$ are of the form $(1,(1,\cdots,y,\cdots,1))$. Then:
\begin{align*}
    (1,(1,\cdots, x,\cdots, 1))(1,(1,\cdots,y,\cdots,1)) = (1\phi_{n+1}(1,\cdots, x,\cdots, 1)(1),(1,\cdots,x,\cdots,y,\cdots,1))\\
    =(1,(1,\cdots,y,\cdots,1))(1,(1,\cdots,x,\cdots,1)).
\end{align*}
Thus, $T$ and $S$ centralize each other.

\hfill\break
\textbf{Case 3:} Suppose $S$ is a Sylow $p_i$-group of the specified form and $T$ is a Sylow $q_j$-group. Then, an element of $S$ if of the form $(1,(1,\cdots,y,\cdots,1))$ and an element of $T$ is of the form $((1,\cdots,x,\cdots,1),1)$. If $C_{r_d,n+1,p_i}$ acts fixed-point freely on $C_{r_d,n+1,q_j}$, then $p_i\in\pi(L_{r_d}^{(n+1)})$, a contradiction to the choice of $L_{r_d}^{(n+1)}$. Thus, the action of $C_{r_d,n+1,p_i}$ on $C_{r_d,n+1,q_j}$ must be trivial. We may write:
\begin{align*}
    (1,(1,\cdots, y,\cdots,1))((1,\cdots,x,\cdots,1),1) = (\phi_{n+1}(1,\cdots,y,\cdots,1)(1,\cdots,x,\cdots,1),(1,\cdots,y,\cdots,1))\\
    = ((1,\cdots,x,\cdots,1),(1,\cdots,y,\cdots,1)) = ((1,\cdots,x,\cdots,1),1)(1,(1,\cdots,y,\cdots,1)).
\end{align*}
Thus, in all cases, $T$ normalizes $S$. By the earlier reductions, $M^{(n+1)}_{r_d}\unlhd K_{r_d}^{(n+1)}$. As we saw earlier, $L_{r_d}^{(n+1)}M_{r_d}^{(n+1)} = K^{(n+1)}$. Given $((x_1,\cdots,x_{m_{n+1}}),(y_1,\cdots,y_{u_{n+1}}))\in M_{r_d}^{(n+1)}\cap L_{r_d}^{(n+1)}$, notice that by choice of $M_{r_d}^{(n+1)}, L_{r_d}^{(n+1)}$, $x_i = 1$ for each $1\leq i\leq m_{n+1}$ and $y_i = 1$ for each $1\leq i\leq u_{n+1}$. Thus, $K^{(n+1)}$ is isomorphic to $L_{r_{d}}^{(n+1)}\ltimes M_{r_d}^{(n+1)}$ where $L_{r_d}^{(n+1)}$ acts on $M_{r_d}^{(n+1)}$ by conjugation. We aim to explicitly construct such an isomorphism $\delta_{n+1,r_d}$, which we claim is defined by:
\begin{align*}
    \delta_{n+1,r_d}((x_1,\cdots,x_{m_{n+1}}),(y_1,\cdots,y_{u_{n+1}}))\\
    =(((1,\cdots,1,x_{s_{r_d,n+1}+1},\cdots,x_{m_{n+1}}),(1,\cdots,1,y_{t_{r_d,n+1}+1},\cdots,y_{u_{n+1}})),\\
    ((x_1,\cdots,x_{s_{r_d,n+1}},1,\cdots,1),(y_1,\cdots,y_{t_{r_d,n+1}},1,\cdots,1))).
\end{align*}
We first prove that $\delta_{n+1,r_d}$ is a homomorphism. Let $(\overline{x},\overline{y}),(\overline{a},\overline{b})\in K^{(n+1)}$. Then:
    \begin{align*}
        ((x_1,\cdots, x_{m_{n+1}}),(y_1,\cdots,y_{u_{n+1}}))((a_1,\cdots,a_{m_{n+1}}),(b_1,\cdots,b_{u_{n+1}}))\\
        =((x_1,\cdots, x_{m_{n+1}})\phi_{n+1}(y_1,\cdots, y_{u_{n+1}})(a_1,\cdots,a_{m_{n+1}}),(y_1b_1,\cdots, y_{u_{n+1}}b_{u_{n+1}}))\\
        = ((x_1\gamma_{1,1}(y_1)\circ\cdots\circ\gamma_{u_{n+1},1}(y_{u_{n+1}})(a_1),\cdots,x_{m_{n+1}}\gamma_{1,m_{n+1}}(y_1)\circ\cdots\circ\gamma_{u_{n+1},m_{n+1}}(y_{u_{n+1}})(a_{m_{n+1}})),\\
        (y_1b_1,\cdots, y_{u_{n+1}}b_{u_{n+1}}))\\
       =((x_1\gamma_{1,1}(y_1)\circ\cdots\circ\gamma_{t_{r_d,n+1},1}(y_{s_{r_d,n+1}})(a_1),\cdots,x_{s_{r_d,n+1}}\gamma_{1,s_{r_d,n+1}}(y_1)\circ\cdots\circ\gamma_{t_{r_d,n+1},s_{r_d,n+1}}(y_{t_{r_d,n+1}})(a_{s_{r_d,n+1}}),\\
       x_{s_{r_d,n+1}+1}\gamma_{1,s_{r_d,n+1}+1}(y_1)\circ\cdots\circ\gamma_{u_{n+1},s_{r_d,n+1}+1}(y_{u_{n+1}})(a_{s_{r_d,n+1}+1}),\cdots,x_{m_{n+1}}\gamma_{1,m_{n+1}}(y_1)\circ\cdots\\
       \circ\gamma_{u_{n+1},m_{n+1}}(y_{u_{n+1}})(a_{m_{n+1}})),(y_1b_1,\cdots, y_{u_{n+1}}b_{u_{n+1}}))
 \end{align*}
 as, by choice of the actions $\gamma_{i,j}$, for $t_{r_{d},n+1}+1\leq i\leq u_{n+1}$ and $1\leq j\leq s_{r_d,n+1}$, $\gamma_{i,j}$ must be the trivial action or else $p_{i}$ would be a prime of $L_{r_d}^{(n+1)}$, a contradiction. Thus:
 \begin{align*}
     \delta_{n+1,r_d}((\overline{x},\overline{y})(\overline{a},\overline{b}))\\
     =(((1,\cdots,1,x_{s_{r_d,n+1}+1}\gamma_{1,s_{r_d,n+1}+1}(y_1)\circ\cdots\\
     \circ\gamma_{u_{n+1},s_{r_d,n+1}+1}(y_{u_{n+1}})(a_{s_{r_d,n+1}+1}),\cdots,x_{m_{n+1}}\gamma_{1,m_{n+1}}(y_1)\circ\cdots\\
       \circ\gamma_{u_{n+1},m_{n+1}}(y_{u_{n+1}})(a_{m_{n+1}})),(1,\cdots,1,y_{t_{r_d,n+1}+1}b_{t_{r_d,n+1}+1},\cdots,y_{u_{n+1}}b_{u_{n+1}})),\\
       ((x_1\gamma_{1,1}(y_1)\circ\cdots\circ\gamma_{t_{r_d,n+1},1}(y_{s_{r_d,n+1}})(a_1),\cdots,\\
       x_{s_{r_d,n+1}}\gamma_{1,s_{r_d,n+1}}(y_1)\circ\cdots\circ\gamma_{t_{r_d,n+1},s_{r_d,n+1}}(y_{t_{r_d,n+1}})(a_{s_{r_d,n+1}}),1,\cdots,1),\\
       (y_1b_1,\cdots,y_{t_{r_d,n+1}}b_{t_{r_d,n+1}},1,\cdots,1))
 \end{align*}
 Now:
 \begin{align*}
     \delta_{n+1,r_d}(\overline{x},\overline{y})\delta_{n+1,r_d}(\overline{a},\overline{b})\\
     =(((1,\cdots,1,x_{s_{r_d,n+1}+1},\cdots,x_{m_{n+1}}),(1,\cdots,1,y_{t_{r_d,n+1}+1},\cdots,y_{u_{n+1}})),\\
     ((x_1,\cdots,x_{s_{r_d,n+1}},1,\cdots,1),(y_1,\cdots,y_{t_{r_d,n+1}},1,\cdots,1)))\\
     (((1,\cdots,1,a_{s_{r_d,n+1}+1},\cdots,a_{m_{n+1}}),(1,\cdots,1,b_{t_{r_d,n+1}+1},\cdots,b_{u_{n+1}})),\\
     ((a_1,\cdots,a_{s_{r_d,n+1}},1,\cdots,1),(b_1,\cdots,b_{t_{r_d,n+1}},1,\cdots,1)))\\
     =(((1,\cdots,1,x_{s_{r_d,n+1}+1},\cdots,x_{m_{n+1}}),(1,\cdots,1,y_{t_{r_d,n+1}+1},\cdots,y_{u_{n+1}}))\\
     \cdot((x_1,\cdots,x_{s_{r_d,n+1}},1,\cdots,1),(y_1,\cdots,y_{t_{r_d,n+1}},1,\cdots,1))\\
     \cdot((1,\cdots,1,a_{s_{r_d,n+1}+1},\cdots,a_{m_{n+1}}),(1,\cdots,1,b_{t_{r_d,n+1}+1},\cdots,b_{u_{n+1}}))\\
     \cdot((x_1,\cdots,x_{s_{r_d,n+1}},1,\cdots,1),(y_1,\cdots,y_{t_{r_d,n+1}},1,\cdots,1))^{-1},\\
     ((x_1,\cdots,x_{s_{r_d,n+1}},1,\cdots,1),(y_1,\cdots,y_{t_{r_d,n+1}},1,\cdots,1))\\
     \cdot((a_1,\cdots,a_{s_{r_d,n+1}},1,\cdots,1),(b_1,\cdots,b_{t_{r_d,n+1}},1,\cdots,1)))\\
 \end{align*}
 Looking at the second component, we may write:
 \begin{align*}
     ((x_1,\cdots,x_{s_{r_d,n+1}},1,\cdots,1),(y_1,\cdots,y_{t_{r_d,n+1}},1,\cdots,1))\\
     \cdot((a_1,\cdots,a_{s_{r_d,n+1}},1,\cdots,1),(b_1,\cdots,b_{t_{r_d,n+1}},1,\cdots,1))\\
     =((x_1,\cdots,x_{s_{r_d,n+1}},1,\cdots,1)\phi_{n+1}(y_1,\cdots,y_{t_{r_d,n+1}},1,\cdots,1)(a_1,\cdots,a_{s_{r_d,n+1}},1,\cdots,1),\\
     (y_1,\cdots,y_{t_{r_d,n+1}},1,\cdots,1)(b_1,\cdots,b_{t_{r_d,n+1}},1,\cdots,1))\\
     =(x_1\gamma_{1,1}(y_1)\circ\cdots\circ\gamma_{t_{r_d,n+1},1}(y_{t_{r_d,n+1}})(a_1),\cdots, \\
     x_{s_{r_d,n+1}}\gamma_{1,s_{r_d,n+1}}(y_1)\circ\cdots\circ\gamma_{t_{r_d,n+1},s_{r_d,n+1}}(y_{t_{r_d,n+1}})(a_{s_{r_d,n+1}}),1,\cdots,1),\\
     (y_1b_1,\cdots,y_{t_{r_d,n+1}}b_{t_{r_d,n+1}},1,\cdots,1))
 \end{align*}
Now, we may consider the first component. Note that $\phi_{n+1}(\overline{c})$ and $\phi_{n+1}(\overline{d})$ commute for any choice of $\overline{c},\overline{d}$ by definition of $\phi_{n+1}$. We may write:
\begin{align*}
    ((1,\cdots,1,a_{s_{r_d,n+1}+1},\cdots,a_{m_{n+1}}),(1,\cdots,1,b_{t_{r_d,n+1}+1},\cdots,b_{u_{n+1}}))\\
     \cdot((x_1,\cdots,x_{s_{r_d,n+1}},1,\cdots,1),(y_1,\cdots,y_{t_{r_d,n+1}},1,\cdots,1))^{-1}\\
     =((1,\cdots,1,a_{s_{r_d,n+1}+1},\cdots,a_{m_{n+1}}),(1,\cdots,1,b_{t_{r_d,n+1}+1},\cdots,b_{u_{n+1}}))\\
     \cdot(\phi_{n+1}(y_1^{-1},\cdots,y_{t_{r_d,n+1}}^{-1},1,\cdots,1)(x_1^{-1},\cdots,x_{s_{r_d,n+1}}^{-1},1,\cdots,1),(y_1^{-1},\cdots,y_{t_{r_d,n+1}}^{-1},1,\cdots,1))\\
     =((1,\cdots,1,a_{s_{r_d,n+1}+1},\cdots,a_{m_{n+1}})\phi_{n+1}(1,\cdots,1,b_{t_{r_d,n+1}+1},\cdots,b_{u_{n+1}})\\
     \circ\phi_{n+1}(y_1^{-1},\cdots,y_{t_{r_d,n+1}}^{-1},1,\cdots,1)(x_1^{-1},\cdots,x_{s_{r_d,n+1}}^{-1},1,\cdots,1),(1,\cdots,b_{t_{r_d,n+1}+1}y_{t_{r_d,n+1}+1}^{-1},\cdots,b_{u_{n+1}}y_{u_{n+1}}^{-1})).
\end{align*}
From there, we may write:
\begin{align*}
    ((x_1,\cdots,x_{s_{r_d,n+1}},1,\cdots,1),(y_1,\cdots,y_{t_{r_d,n+1}},1,\cdots,1))\\
    \cdot((1,\cdots,1,a_{s_{r_d,n+1}+1},\cdots,a_{m_{n+1}})\phi_{n+1}(1,\cdots,1,b_{t_{r_d,n+1}+1},\cdots,b_{u_{n+1}})\\
     \circ\phi_{n+1}(y_1^{-1},\cdots,y_{t_{r_d,n+1}}^{-1},1,\cdots,1)(x_1^{-1},\cdots,x_{s_{r_d,n+1}}^{-1},1,\cdots,1),(1,\cdots,b_{t_{r_d,n+1}+1}y_{t_{r_d,n+1}+1}^{-1},\cdots,b_{u_{n+1}}y_{u_{n+1}}^{-1}))\\
     =((x_1,\cdots,x_{s_{r_d,n+1}},1,\cdots,1)\phi_{n+1}(y_1,\cdots,y_{t_{r_d,n+1}},1,\cdots,1)(1,\cdots,1,a_{s_{r_d,n+1}+1},\\
     \cdots,a_{m_{n+1}})\phi_{n+1}(y_1,\cdots,y_{t_{r_d,n+1}},1,\cdots,1)\circ\phi_{n+1}(1,\cdots,1,b_{t_{r_d,n+1}+1},\cdots,b_{u_{n+1}})\\
     \circ\phi_{n+1}(y_1^{-1},\cdots,y_{t_{r_d,n+1}}^{-1},1,\cdots,1)\\
     (x_1^{-1},\cdots,x_{s_{r_d,n+1}}^{-1},1,\cdots,1),(1,\cdots,1,b_{t_{r_d,n+1}+1},\cdots,b_{u_{n+1}}))\\
     =((x_1,\cdots,x_{s_{r_d,n+1}},1,\cdots,1)\phi_{n+1}(y_1,\cdots,y_{t_{r_d,n+1}},1,\cdots,1)(1,\cdots,1,a_{s_{r_d,n+1}+1},
     \cdots,a_{m_{n+1}})\\
     \phi_{n+1}(1,\cdots,1,b_{t_{r_d,n+1}+1},\cdots,b_{u_{n+1}})(x_1^{-1},\cdots,x_{s_{r_d,n+1}}^{-1},1,\cdots,1),(1,\cdots,1,b_{t_{r_d,n+1}+1},\cdots,b_{u_{n+1}})).
\end{align*}
Notice that, for each $C_{p_{r_d,n+1,i}}$ with $i>t_{r_d,n+1}$ and each $C_{q_{r_d,n+1,j}}$ with $j\leq s_{r_d,n+1}$, $\gamma_{i,j}$ is trivial, or otherwise $p_{r_d,n+1,i}\in \pi(L_{r_d}^{(n+1)})$, a contradiction. Thus:
\begin{align*}
    ((x_1,\cdots,x_{s_{r_d,n+1}},1,\cdots,1)\phi_{n+1}(y_1,\cdots,y_{t_{r_d,n+1}},1,\cdots,1)(1,\cdots,1,a_{s_{r_d,n+1}+1},
     \cdots,a_{m_{n+1}})\\
     \phi_{n+1}(1,\cdots,1,b_{t_{r_d,n+1}+1},\cdots,b_{u_{n+1}})(x_1^{-1},\cdots,x_{s_{r_d,n+1}}^{-1},1,\cdots,1),(1,\cdots,1,b_{t_{r_d,n+1}+1},\cdots,b_{u_{n+1}}))\\
    = ((x_1,\cdots,x_{s_{r_d,n+1}},1,\cdots,1)\phi_{n+1}(y_1,\cdots,y_{t_{r_d,n+1}},1,\cdots,1)(1,\cdots,1,a_{s_{r_d,n+1}+1},
     \cdots,a_{m_{n+1}})\\
     (x_1^{-1},\cdots,x_{s_{r_d,n+1}}^{-1},1,\cdots,1),(1,\cdots,1,b_{t_{r_d,n+1}+1},\cdots,b_{u_{n+1}}))\\
     =(\phi_{n+1}(y_1,\cdots,y_{t_{r_d,n+1}},1,\cdots,1)(1,\cdots,1,a_{s_{r_d,n+1}+1},
     \cdots,a_{m_{n+1}}), (1,\cdots,1,b_{t_{r_d,n+1}+1},\cdots,b_{u_{n+1}})).
\end{align*}
From there, we may write:
\begin{align*}
    ((1,\cdots,1,x_{s_{r_d,n+1}+1},\cdots,x_{m_{n+1}}),(1,\cdots,1,y_{t_{r_d,n+1}+1},\cdots,y_{u_{n+1}}))\\
    (\phi_{n+1}(y_1,\cdots,y_{t_{r_d,n+1}},1,\cdots,1)(1,\cdots,1,a_{s_{r_d,n+1}+1},
     \cdots,a_{m_{n+1}}), (1,\cdots,1,b_{t_{r_d,n+1}+1},\cdots,b_{u_{n+1}}))\\
     =((1,\cdots,1,x_{s_{r_d,n+1}+1},\cdots,x_{m_{n+1}})\phi_{n+1}(1,\cdots,1,y_{t_{r_d,n+1}+1},\cdots,y_{u_{n+1}})\\
     \circ\phi_{n+1}(y_1,\cdots,y_{t_{r_d,n+1}},1,\cdots,1)(1,\cdots,1,a_{s_{r_d,n+1}+1},
     \cdots,a_{m_{n+1}}),\\
     (1,\cdots,1,y_{t_{r_d,n+1}+1}b_{t_{r_d,n+1}+1},\cdots,y_{u_{n+1}}b_{u_{n+1}}))\\
     =((1,\cdots,1,x_{s_{r_d,n+1}+1},\cdots,x_{m_{n+1}})\phi_{n+1}(y_1,\cdots,y_{u_{n+1}})(1,\cdots,1,a_{s_{r_d,n+1}+1},
     \cdots,a_{m_{n+1}}),\\
     (1,\cdots,1,y_{t_{r_d,n+1}+1}b_{t_{r_d,n+1}+1},\cdots,y_{u_{n+1}}b_{u_{n+1}}))\\
     =((1,\cdots,1,x_{s_{r_d,n+1}+1}\gamma_{1,s_{r_d,n+1}+1}(y_1)\circ\cdots\circ\gamma_{u_{n+1},s_{r_d,n+1}+1}(y_{u_n})(a_{s_{r_d,n+1}+1}),\cdots,\\
     x_{m_{n+1}}\gamma_{1,m_{n+1}}(y_1)\circ\cdots\circ\gamma_{u_{n+1},m_{n+1}}(y_{u_{n+1}})(a_{m_{n+1}})),(1,\cdots,1,y_{t_{r_d,n+1}+1}b_{t_{r_d,n+1}+1},\cdots,y_{u_{n+1}}b_{u_{n+1}}))
\end{align*}
Thus, $\phi_{n+1}$ is a homomorphism. Notice that it is an injection and thus an isomorphism (as $|K^{(n+1)}|=|L_{r_d}^{(n+1)}\ltimes M_{r_d}^{(n+1)}|$).

Define $\pi_{n+1,r_d}:L_{r_d}^{(n+1)}\ltimes M_{r_d}^{(n+1)}\to L_{r_d}^{(n+1)}$ by 
    \begin{align*}
        \pi_{n+1, r_d}(((g_1,\cdots, g_{m_{n+1}}),(h_1,\cdots, h_{u_{n+1}})),((a_1,\cdots, a_{m_{n+1}}),(b_1,\cdots, b_{u_{n+1}})))\\
        = ((a_1,\cdots, a_{m_{n+1}}),(b_1,\cdots, b_{u_{n+1}})).
    \end{align*}
    Notice that $\pi_{n+1,r_d}$ is a group homomorphism. Now, for $d\geq k_{n+1}$, let $R_{r_d}$ be the $r_d$-module defined in \cite[Lemma 1.8]{fittingSub} and let $\xi_{r_d}$ be the action of $L_{r_d}^{(n+1)}$ on $R_{r_d}$, given by \cite[Lemma 1.8]{fittingSub}. Now, we may define $\xi_{r_d}^{(n+1)}:K^{(n+1)}\to\operatorname{Aut}(R_{r_d})$ by $\xi_{r_d}\circ\pi_{n+1, r_d}\circ\delta_{n+1, r_d}\circ\eta_{n+1,r_d}$. Notice that $\xi_{r_d}^{(n+1)}$ is a homomorphism. 

    For $d<k_{n+1}$, let $\zeta_{n+1,r_d}:L_{r_d}^{(n+1)}\to L_{r_d}^{(n)}$ be a projection operator defined by 
    \begin{align*}
        \zeta_{n+1,r_d}((x_1,\cdots,x_{m_{n+1}}),(y_1,\cdots,y_{u_{n+1}})) = ((x_1,\cdots,x_{m_n}),(y_1,\cdots,y_{u_n})).
    \end{align*}
    We claim that $\zeta_{n+1,r_d}$ is a homomorphism. By the fact that $r_d\in I_n$, the construction of $L_{r_d}^{(n+1)}$ and $M_{r_d}^{(n+1)}$ and the construction in \thref{labelingsolvable}, $x_{m_n+1} = \cdots =x_{m_{n+1}} = 1$ and $y_{u_n+1}=\cdots=y_{u_{n+1}}=1$. Let \\
    $((x_1,\cdots,x_{m_{n}},1,\cdots,1),(y_1,\cdots,y_{u_{n}},1,\cdots,1))$ and $((a_1,\cdots,a_{m_{n}},1,\cdots,1),(b_1,\cdots,b_{u_{n}},1,\cdots,1))$ be in $ L_{r_d}^{(n+1)}$. Then:
    \begin{align*}
        \zeta_{n+1,r_d}(((x_1,\cdots,x_{m_{n}},1,\cdots,1),(y_1,\cdots,y_{u_{n}},1,\cdots,1))\\
        \cdot((a_1,\cdots,a_{m_{n}},1,\cdots,1),(b_1,\cdots,b_{u_{n}},1,\cdots,1)))\\
        =\zeta_{n+1,r_d}((x_1,\cdots,x_{m_{n}},1,\cdots,1)\phi_{n+1}(y_1,\cdots,y_{u_{n}},1,\cdots,1)(a_1,\cdots,a_{m_{n}},1,\cdots,1),\\
        (y_1b_1,\cdots,y_{u_n}b_{u_n},1,\cdots,1))\\
        =\zeta_{n+1,r_d}((x_1,\cdots,x_{m_n},1,\cdots,1)(\gamma_{1,1}(y_1)\circ\cdots\\
        \circ\gamma_{u_n,1}(y_{u_n})(a_1),\cdots,\gamma_{1,m_n}(y_1)\circ\cdots\circ\gamma_{u_n,m_m}(y_{u_n})(a_{m_n}),1,\cdots1),\\
        (y_1b_1,\cdots,y_{u_n}b_{u_n},1,\cdots,1))\\
        =\zeta_{n+1,r_d}((x_1\gamma_{1,1}(y_1)\circ\cdots\circ\gamma_{u_n,1}(y_{u_n})(a_1),\cdots,x_{m_n}\gamma_{1,m_n}(y_1)\circ\cdots\\
        \circ\gamma_{u_n,m_m}(y_{u_n})(a_{m_n}),1,\cdots1),(y_1b_1,\cdots,y_{u_n}b_{u_n},1,\cdots,1))\\
        = ((x_1\gamma_{1,1}(y_1)\circ\cdots\circ\gamma_{u_n,1}(y_{u_n})(a_1),\cdots,x_{m_n}\gamma_{1,m_n}(y_1)\circ\cdots\\
        \circ\gamma_{u_n,m_m}(y_{u_n})(a_{m_n})),(y_1b_1,\cdots,y_{u_n}b_{u_n}))\\
        =((x_1,\cdots,x_{m_n})\phi_n(y_1,\cdots,y_{u_n})(a_1,\cdots,a_{m_n}),(y_1,\cdots,y_{u_n})(b_1,\cdots,b_{u_n}))\\
        = ((x_1,\cdots,x_{m_n}),(y_1,\cdots,y_{u_n}))((a_1,\cdots,a_{m_n}),(b_1,\cdots,b_{u_n}).
    \end{align*}
    Define $\xi_{r_d}^{(n+1)}$ by $\xi_{r_d}\circ \zeta_{h_{r_d}+1,r_d}\circ\cdots\circ\zeta_{n+1,r_d}\circ\pi_{n+1,r_d}\circ\delta_{n+1,r_d}\circ\eta_{n+1,r_d}$ where $h_{r_d}$ is the first index $n$ for which $r_d\in I_n$. Then, because $\xi_{r_d}^{(n+1)}$ is a homomorphism from $K^{(n+1)}$ to $\operatorname{Aut}(R_{r_d})$, we may define $K^{(n+1)}\ltimes R_{r_d}$ where the action is $\xi_{r_d}^{(n+1)}$. From there, we may define:
    \begin{align*}
        K^{(n+1)}\ltimes_{\xi^{(n+1)}}R_{r_0}\times\cdots\times R_{r_{w_{n+1}}}
    \end{align*}
    where $\xi^{(n+1)} = \xi_{r_0}^{(n+1)}\times\cdots\times\xi_{r_{w_{n+1}}}^{(n+1)}$, that is, $\xi_{r_d}^{(n+1)}$ acts only on the coordinates corresponding to $R_{r_d}$. Because, for any $(\overline{x},\overline{y})\in K^{(n+1)}, \xi_{r_d}^{(n+1)}(\overline{x},\overline{y})$ is an automorphism of $R_{r_d}$, $\xi^{(n+1)} (\overline{x},\overline{y})$ is an automorphism of $R_{r_0}\times \cdots\times R_{r_{w_{n+1}}}$. We aim to show $\xi^{(n+1)}$ is a homomorphism. Let $(\overline{x},\overline{y}),(\overline{a},\overline{b})\in K^{(n+1)}$. Then:
    \begin{align*}
        \xi^{(n+1)}((\overline{x},\overline{y})(\overline{a},\overline{b})) = \xi_{r_0}^{(n+1)}((\overline{x},\overline{y})(\overline{a},\overline{b}))\times\cdots\times\xi_{r_{w_{n+1}}}^{(n+1)}((\overline{x},\overline{y})(\overline{a},\overline{b}))\\
        =\xi_{r_0}^{(n+1)}(\overline{x},\overline{y})\circ\xi_{r_0}^{(n+1)}(\overline{a},\overline{b})\times\cdots\times\xi_{r_{w_{n+1}}}^{(n+1)}(\overline{x},\overline{y})\circ \xi_{r_{w_{n+1}}}^{(n+1)}(\overline{a},\overline{b})\\
        =(\xi_{r_0}^{(n+1)}(\overline{x},\overline{y})\times\cdots\times\xi_{r_{w_{n+1}}}^{(n+1)}(\overline{x},\overline{y}))\circ(\xi_{r_0}^{(n+1)}(\overline{a},\overline{b})\times\cdots\times\xi_{r_{w_{n+1}}}^{(n+1)}(\overline{a},\overline{b}))=\xi^{(n+1)}(\overline{x},\overline{y})\circ\xi^{(n+1)}(\overline{a},\overline{b}).
    \end{align*}
    Finally, define $H_{n+1}$ to be the group $K^{(n+1)}\ltimes_{\xi^{(n+1)}}R_{r_0}\times\cdots\times R_{w_{n+1}}$. We claim $\pgc(H_{n+1})$ is $\Xi_L[O_{n+1}\cup D_{n+1}\cup I_{n+1}]$. Notice that for any $p\in O_{n+1},q\in D_{n+1}$, $p-q\in \pgc(H_{n+1})$ if and only if $p-q\in\Xi_L[O_{n+1}\cup D_{n+1}\cup I_{n+1}]$ by  \cite[Corollary 1.5]{Suz} and the analogous fact for $K^{(n+1)}$. Suppose $r_d-q\in \Xi_L[O_{n+1}\cup D_{n+1}\cup I_{n+1}]$ for some $q\in O_{n+1}\cup D_{n+1}$ and $0\leq d\leq w_{n+1}$. Then, $q\in \pi(\operatorname{Fit}(L_{r_d}^{(n+1)}))$, so because every element of $L_{r_d}^{(n+1)}\ltimes M_{r_d}^{(n+1)}$ with order $q$ must be of the form $(x,y)$ where $y\in L_{r_d}^{(n+1)}$ has order $q$, every element of order $q$ in $L_{r_d}^{(n+1)}\ltimes M_{r_d}^{(n+1)}$ acts fixed-point freely on $R_{r_d}$, so by  \cite[Corollary 1.5]{Suz}, $r_d-q\in\pgc(H_{n+1})$ (to see this for the case where $d<k_{n+1}$, notice that $\zeta_{h_{r_d}+1,r_d}\circ\cdots\circ\zeta_{n+1,r_d}(y)$ must have order $q$ by choice of labeling). If $r_d-q\in \pgc(H_{n+1})$, then any element of order $q$ in $L_{r_d}^{(n+1)}\ltimes M_{r_d}^{(n+1)}$ must be of the form $(\overline{x}, \overline{y})$ where $\overline{y}\in L_{r_d}^{(n+1)}$ has order $q$ (and if $d<k_{n+1}$, $\zeta_{h_{r_d}+1,r_d}\circ\cdots\circ\zeta_{n+1,r_d}(\overline{y})$ has order $q$), so because $\pgc(H_{n+1})$ is triangle-free, $q$ must correspond to a prime in $\operatorname{Fit}(L_{r_d}^{(n+1)})$ for $d\geq k_{n+1}$ (and if $d<k_{n+1}$, $q$ must correspond to a prime in  $\operatorname{Fit}(L_{h_{r_d}}^{(n+1)})$) by construction. Thus, $r_d-q\in\Xi_L[O_{n+1}\cup D_{n+1}\cup I_{n+1}]$, so $\pgc(H_{n+1}) = \Xi_L[O_{n+1}\cup D_{n+1}\cup I_{n+1}]$.

\hfill\break
    Finally, for each $i\in \mathbb{N}$, define $\varphi_i:H_i\to H_{i+1}$ by
    \begin{align*}
        \varphi_i(\overline{a}, (\overline{b},\overline{c})) = ((\overline{a},1,\cdots, 1),((\overline{b},1,\cdots,1),(\overline{c},1,\cdots,1))).
    \end{align*}
    We claim that $\varphi_i$ is an embedding. To see this, we first show that it is a homomorphism. Let $(\overline{a}, (\overline{b},\overline{c})),(\overline{x},\overline{y},\overline{z})\in H_i$. Then:
    \begin{align*}
        (\overline{a}, (\overline{b},\overline{c}))(\overline{x},(\overline{y},\overline{z})) = (\overline{a}\xi^{(i)}(\overline{b},\overline{c})(\overline{x}),(\overline{b},\overline{c})(\overline{y},\overline{z}))\\
        =(\overline{a}\xi^{(i)}(\overline{b},\overline{c})(\overline{x}),(\overline{b}\phi_i(\overline{c})(\overline{y}),\overline{c}\overline{z}))\\
        \varphi((\overline{a}, (\overline{b},\overline{c}))(\overline{x},(\overline{y},\overline{z}))) = ((\overline{a},1,\cdots,1)(\xi^{(i)}(\overline{b},\overline{c})(\overline{x}),1,\cdots,1),((\overline{b}\phi_i(\overline{c})(\overline{y}),1,\cdots,1),(\overline{c}\overline{z},1,\cdots,1))).
    \end{align*}
    From here, we may write:
    \begin{align*}
        \varphi_i(\overline{a}, (\overline{b},\overline{c}))\varphi(\overline{x},(\overline{y},\overline{z}))\\
        = ((\overline{a},1,\cdots,1),((\overline{b},1,\cdots,1),(\overline{c},1,\cdots,1)))((\overline{x},1,\cdots,1),((\overline{y},1,\cdots,1),(\overline{z},1,\cdots,1)))\\
        =((\overline{a},1,\cdots,1)\xi^{(i+1)}((\overline{b},1,\cdots,1),(\overline{c},1,\cdots,1))(\overline{x},1,\cdots,1),\\
        ((\overline{b},1,\cdots,1)\phi_{i+1}(\overline{c},1,\cdots,1)(\overline{y},1,\cdots,1),(\overline{c}\overline{z},1,\cdots,1)))\\
        =((\overline{a},1,\cdots,1)\xi^{(i+1)}((\overline{b},1,\cdots,1),(\overline{c},1,\cdots,1))(\overline{x},1,\cdots,1),\\
        ((\overline{b}\phi_{i}(\overline{c})(\overline{y}),1,\cdots,1),(\overline{c}\overline{z},1,\cdots,1)))\\
        =((\overline{a},1,\cdots,1)(\xi_{r_0}\circ\zeta_{h_{r_0}+1,r_0}\circ\cdots\circ\zeta_{i+1,r_0}\circ\pi_{r_0,i+1}\\
        \circ\delta_{r_0,i+1}\circ\eta_{i+1,r_0}(((\overline{b},1,\cdots,1),(\overline{c},1,\cdots,1)))\times\cdots\\\times\xi_{r_{w_{i+1}}}\circ\zeta_{h_{r_{w_{i+1}}}+1,r_{w_{i+1}}}\circ\cdots\circ\zeta_{i+1,r_{w_{i+1}}}\\\circ\pi_{r_{w_{i+1}},i+1}\circ\delta_{r_{w_{i+1}},i+1}\circ\eta_{i+1,r_{w_{i+1}}}(((\overline{b},1,\cdots,1),(\overline{c},1,\cdots,1))))(\overline{x},1,\cdots,1),\\
        ((\overline{b}\phi_{i}(\overline{c})(\overline{y}),1,\cdots,1),(\overline{c}\overline{z},1,\cdots,1)))\\
    =((\overline{a},1,\cdots,1)(\xi_{r_0}\circ\zeta_{h_{r_0}+1,r_0}\circ\cdots\circ\zeta_{i,r_0}((b_1^{(i,r_0)},\cdots,b_{s_{r_0,i}}^{(i,r_0)},1,\cdots,1),(c_1^{(i,r_0)},\cdots,c_{t_{r_0,i}}^{(i,r_0)},1,\cdots,1))\times\cdots\\\times\xi_{r_{w_{i}}}\circ\zeta_{h_{r_{w_{i}}}+1,r_{w_{i}}}\circ\cdots\\
    \circ\zeta_{i,r_{w_{i}}}((b_1^{(i,r_{w_{i}})},\cdots,b_{s_{r_0,i}}^{(i,r_{w_{i}})},1,\cdots,1),(c_1^{(i,r_{w_{i}})},\cdots,c_{t_{r_{w_{i}},i}}^{(i,r_{w_{i}})},1,\cdots,1))\times1\times\cdots\times1)(\overline{x},1,\cdots,1),\\
    ((\overline{b}\phi_{i}(\overline{c})(\overline{y}),1,\cdots,1),(\overline{c}\overline{z},1,\cdots,1)))\\
    =(((\overline{a},1,\cdots,1)(\zeta_{r_0}^{(i)}\times\cdots\times\zeta_{r_{w_{i}}}\times1\times\cdots\times1)(\overline{x},1,\cdots,1)),\\
    ((\overline{b}\phi_{i}(\overline{c})(\overline{y}),1,\cdots,1),(\overline{c}\overline{z},1,\cdots,1)))\\
    =(((\overline{a},1,\cdots,1)(\xi^{(i)}(\overline{x}),1,\cdots,1)),\\
    ((\overline{b}\phi_{i}(\overline{c})(\overline{y}),1,\cdots,1),(\overline{c}\overline{z},1,\cdots,1)))\\
    =((\overline{a}\xi^{(i)}(\overline{x}),1,\cdots,1),  ((\overline{b}\phi_{i}(\overline{c})(\overline{y}),1,\cdots,1),(\overline{c}\overline{z},1,\cdots,1))).
    \end{align*}


    
    Thus, $\varphi_i$ is an injective homomorphism, so it is an embedding. Then, for each $i,j\in\mathbb{N}$ with $i<j$, define:
    \begin{align*}
        f_{i,j} = \varphi_{j-1}\circ\cdots\circ\varphi_i\\
        f_{i,i} = \operatorname{id}_{H_i}.
    \end{align*}
    We now claim that $G = \varinjlim H_i$ is a torsion solvable group with prime graph complement $\Xi_L$. Notice that $G = \bigsqcup_{i\in\mathbb{N}} H_i/\sim$ where $x_i\sim x_j$ for $x_i\in H_i, x_j\in H_j$ if $f_{i, k}(x_i) = f_{j,k}(x_j)$ where $k = \max\{i, j\}$ and the group operation is given by, for $x_i\in H_i, x_j\in H_j$, $[x_i]\cdot[x_j] = [f_{i,k}(x_i)f_{j,k}(x_j)]$ where $k = \max\{i, j\}$. Notice that $G$ is a torsion group, as for any $[x_i]\in G$, $[x_i]^{o(x_i)} = [f_{i, i}(x_i)^{o(x_i)}] = [1]$, and notice that $[1] = \{1\},$ so $[1]$ is the identity. We now claim that the group is solvable. By \thref{dll2}, $\varinjlim N_{i,1}\unlhd\varinjlim H_i, \varinjlim N_{i,2}\unlhd \varinjlim H_i,$ and $\varinjlim N_{i,1}\unlhd\varinjlim N_{i, 2}$. By \thref{dll3}, $\varinjlim N_{i,1}$ is abelian and $\varinjlim H_i/\varinjlim N_{i,2}$ is abelian. Notice that the maps $g_i|_{N_{i, 2}}$ are well-defined and $g_i|_{N_{i,2}}(x)\in N_{i+1, 1}$ if and only if $x\in N_{i,1}$. Thus, $\varinjlim N_{i,2}/\varinjlim N_{i,1}$ is abelian by \thref{dll3}, and thus $\varinjlim H_i$ is solvable. Finally, by \thref{dll} and the construction of $\Xi_L$, $\pgc(G) = \Xi_L$. 
    
\hfill\break
    Let $G$ be a locally torsion solvable group with the finite Frobenius property. Let $p\in\pi(G)$. By the definition of the finite Frobenius property, the set of $q\in\pi(G)$ for which there is not an element of order $pq$ is finite, so every vertex of $\pgc(G)$ has finite degree. By \thref{backwardsDirectionTorsion}, $\pgc(G)$ is triangle-free and 3-colorable.
\end{proof}
\begin{corollary}\thlabel{solvableTheorem1}
        A countable locally finite unlabeled graph $\Xi$ is triangle free and 3-colorable if and only if it is isomorphic to the prime graph complement of a torsion solvable group $G$ with the finite Frobenius property. 
\end{corollary}
\begin{proof}
    Notice that this holds from the forwards direction in the proof of \thref{locallyTorsionSolvableThm1}.
\end{proof}
\subsection{Applications}
\begin{proposition}
    Let $T$ be a nonabelian finite simple group such that all finite $T$-solvable groups have triangle-free and 3-colorable prime graphs. Then, a countable locally finite graph $\Xi$ is realized by a $T$-solvable (locally) torsion  group $G$ with the finite Frobenius property if and only if $\Xi$ is triangle-free and 3-colorable.
\end{proposition}
\begin{proof}
    Let $G$ be a $T$-solvable (locally) torsion  group with the finite Frobenius property. Let $\Delta_1\subseteq\pgc(G)$ be a finite subgraph, and let $\Delta_2 = \pgc(G)[\pi(T)]$. Define $\Delta = \Delta_1\cup\Delta_2$. By \thref{hallAnalogLocallyTorsion}, there exists a finite $T$-solvable $H\leq G$ such that $\pgc(H) = \Delta$. Thus, $\pgc(H)$ is triangle-free and 3-colorable, so because $\Delta_1\subseteq\pgc(H),\Delta_1$ is triangle-free and 3-colorable. Because $\Delta_1$ is an arbitrary finite subgraph of $\pgc(G), \pgc(G)$ is triangle-free (otherwise we could choose $\Delta_1$ to contain a triangle of $\pgc(G)$). Thus, by \thref{compact}, $\pgc(G)$ is triangle-free and 3-colorable.

    Let $\Xi$ be a countable triangle-free, 3-colorable graph such that every vertex has finite degree. Then, by \thref{solvableTheorem1}, there exists a torsion solvable group $G$ with the finite Frobenius property such that $\pgc(G)\cong \Xi$.
\end{proof}
We now apply the proposition. Define the set $\mathcal{A}$ to be the set containing the following finite simple groups:
\begin{align*}
    A_{10}, A_{12}, J_2, O_5(4), O_5(5), O_5(7), O_5(9), O_8^+(2), U_3(9), ^3D_4(2), A_9, O_7(2), U_4(2), U_3(3),\operatorname{PSL}(3,3)
\end{align*}
\begin{corollary}
    Let $T\in\mathcal{A}$. Then, a countable locally finite graph $\Xi$ is realized by a $T$-solvable (locally) torsion  group $G$ with the finite Frobenius property if and only if $\Xi$ is triangle-free and 3-colorable.
\end{corollary}
\begin{proof}
    This follows from \cite{2022REU}, \cite[Theorem 3.2.2]{2023REU} and \cite{2022REU}, and \cite[Proposition 2.1.8]{alland}.
\end{proof}
\begin{theorem}\thlabel{disjointClass}
    Let $T$ be a finite nonabelian simple group, and let $C$ be the classification of the prime graph complements of finite $T$-solvable groups. Suppose $C$ can be put in the form:

    ``$\Lambda$ is the prime graph of a finite $T$-solvable group $G$ if and only if
    \begin{enumerate}
        \item $\Lambda$ is triangle-free and 3-colorable.
        \item There exists an $X\subseteq\Lambda$ such that $\Lambda$ is the disjoint union of $X$ and $\Lambda\setminus X$, there exists a $T$-solvable group $H$ with $\pi(H) = \pi(T)$ and $\pgc(H)\cong X$, and $\Lambda\setminus X$ is triangle-free and 3-colorable."
    \end{enumerate}
     Then, a countable locally finite graph is realized by a $T$-solvable group $G$ with the finite Frobenius property if and only if one of the following holds:
    \begin{enumerate}
        \item $\Xi$ is triangle-free and 3-colorable.
        \item If (1) fails, $\Xi$ is the disjoint union of two graphs $A$ and $B$ such that $|A|=|\pi(T)|$, $A$ is realizable by a finite $T$-solvable group, and $B$ is triangle-free and three-colorable.
    \end{enumerate}
\end{theorem}
\begin{proof}
    For the forward direction, first consider the case where $\pgc(G)$ is triangle free and 3-colorable. (1) is satisfied, so we may turn to the case where $\pgc(G)$ is not triangle free and 3-colorable. Note that $G$ must be strictly $T$-solvable by \thref{locallyTorsionSolvableThm1}. By \thref{hallAnalogLocallyTorsion}, there exists a finite $T$-solvable group $H\leq G$ such that $\pi(H) = \pi(T)$ and $\pgc(H) =\pgc(G)[\pi(T)]$, so let $X$ be $\pi(T)$. If $\pi(T) = \pi(G)$, the theorem follows from this as $\pgc(G)[\pi(T)] = \pgc(G)$. Thus, assume $\pi(T)\neq\pi(G)$. 
    
    We will now prove that $\pgc(G)$ is the disjoint union of $\pgc(G)[\pi(T)]$ and $\pgc(G)\setminus\pgc(G)[\pi(T)]$. For contradiction, suppose there exists a $r\in\pi(G)\setminus\pi(T)$ and a $p\in\pi(T)$ such that $r-p\in\pgc(G)$. Let $\Delta$ be a finite induced subgraph of $\pgc(G)$ containing $r-p$ and $\pi(T)$.  By \thref{hallAnalogLocallyTorsion}, there exists a finite $T$-solvable $H\leq G$ such that $\pgc(H) = \Delta$. By our assumption on $C$, $p-r\not\in\pgc(H)=\Delta\subseteq\pgc(G)$, a contradiction. Thus, $\pgc(G)[\pi(T)]$ is disjoint from $\pgc(G)[\pi(G)\setminus\pi(T)]$.
    
     We now aim to show that $\pgc(G)[\pi(G)\setminus\pi(T)]$ is triangle-free and 3-colorable. Let $\Delta_1$ be a finite induced subgraph of $\pgc(G)[\pi(G)\setminus\pi(T)]$, and let $\Delta_2 = \pgc(G)[\pi(T)]$, so take $\Delta = \Delta_1\cup\Delta_2$.   
     By \thref{hallAnalogLocallyTorsion}, there exists a finite strictly $T$-solvable $H\leq G$ such that $\pgc(H) = \Delta$. Notice that by $C$, in both cases, $\Delta_1 = \pgc(H)[V(\Delta_1)]\subseteq\pgc(G)$ is triangle-free and 3-colorable. Thus, $\pgc(G)\setminus\pgc(G)[\pi(T)]$ must be triangle-free (or else we could choose $\Delta$ to contain the triangle). By \thref{compact}, $\pgc(G)\setminus\pgc(G)[\pi(T)]$ is triangle-free and 3-colorable.
    
    Now, we turn to the backwards direction. If $\Xi$ is triangle-free and 3-colorable, by \thref{solvableTheorem1}, $\Xi\cong \pgc(G)$ for a torsion solvable group $G$ with the finite Frobenius property. If $\Xi$ satisfies (2), let $H$ be a finite $T$-solvable group realizing $A$. By \thref{solvableTheorem1}, there exists $N$ a torsion solvable group with the finite Frobenius property such that $\pgc(N)\cong B$ and $\pi(N)\cap\pi(T)=\varnothing$ (to see the last condition, notice that the proof of \thref{labelingsolvable} can be modified to ensure that some finite number of primes do not appear in the labeling). Thus, $\Xi\cong\pgc(H\times N)$, and $H\times N$ is a torsion solvable group with the finite Frobenius property.
\end{proof}
\begin{corollary}
    Let $T$ be $A_6$ or $\operatorname{PSL}(2, 8)$. Then, a countable locally finite graph $\Xi$ is realized by a $T$-solvable (locally) torsion  group $G$ with the finite Frobenius property if and only if $\Xi$ is 3-colorable and one of the following holds:
    \begin{enumerate}
        \item $\Xi$ is triangle-free.
        \item There exists a single triangle in $\Xi$ with vertices $X =\{x, y, z\}$ such that the triangle is isolated.
    \end{enumerate}
\end{corollary}
\begin{proof}
    This follows from \thref{disjointClass}, \cite[Theorem 4.7]{2022REU}, and \cite[Theorem 5.9]{2022REU}.
\end{proof}
\begin{corollary}
    Let $T$ be one of $\operatorname{PSL}(3, 7), U_3(5), ^2F_4(2)',\operatorname{PSL}(4, 3), A_8$, and $U_3(8)$ and let $\Xi$ be countable locally finite graph. $\Xi$ is realized by the prime graph complement of a locally torsion $T$-solvable group with the finite Frobenius property if and only if $\Xi$ is 3-colorable and one of the following holds:
    \begin{enumerate}
        \item $\Xi$ is triangle-free.
        \item $\Xi$ has exactly one isolated triangle $\{w,x, y\}$ and an isolated vertex $z$.
        \item $\Xi$ has exactly one triangle $\{w, x, y\}$, the vertices $w, x$ are not adjacent to any other vertices in $\Xi$, vertex $y$ is adjacent to exactly one other vertex $z$, and $z$ is not adjacent to vertices other than $y$. 
    \end{enumerate}
\end{corollary}
\begin{proof}
    This follows from \thref{disjointClass} and \cite[Theorem 3.3.3]{2023REU}
\end{proof}
\begin{corollary}
        Let $T$ be $G_2(3)$ or $U_4(3)$. Let $\Xi$ be countable locally finite graph. $\Xi$ is realizable by the prime graph complement of a locally torsion $T$-solvable group with the finite Frobenius property if and only if $\Xi$ is 3-colorable and one of the following holds:
    \begin{enumerate}
        \item $\Xi$ is triangle-free.
        \item There exists a subset $X =\{w, x, y, z\}\subseteq V(\Xi)$ such that $\Xi$ is the disjoint union of $\Xi[X]$ and $\Xi\setminus X$, $\Xi\setminus X$ is triangle-free, and $\Xi[X]$ is not complete.
    \end{enumerate}
\end{corollary}
\begin{proof}
    This follows from \thref{disjointClass}, \cite[Theorem 3.6.2]{2023REU}, and \cite[Theorem 3.6.5]{2023REU}.
\end{proof}
\begin{corollary}
        Let $T$ be $\operatorname{PSL}(3, 4)$ or $\operatorname{Sz}(32)$. Let $\Xi$ be a countable locally finite graph. $\Xi$ is realized by the prime graph complement of a locally torsion $T$-solvable group with the finite Frobenius property if and only if one of the following holds:
    \begin{enumerate}
        \item $\Xi$ is triangle-free.
        \item There exists a subset $X = \{w, x, y, z\}\subseteq V(\Xi)$ such that $\Xi$ is the disjoint union of $\Xi[X]$ and $\Xi\setminus X$ and $\Xi\setminus X$ is triangle-free and 3-colorable.
    \end{enumerate}
\end{corollary}
\begin{proof}
    This follows from \thref{disjointClass}, \cite[Theorem 3.6.4]{2023REU}, and \cite[Theorem 2.5]{Suz}.
\end{proof}

\begin{corollary}\thlabel{316cor}
Let $G$ be a torsion $\mathcal{K}_3$-solvable group  with the finite Frobenius property. Suppose $G$ has a subnormal series in which at least two of the factors are nonabelian simple $K_3$-groups. Then, $\pgc(G)$ is triangle-free and 3-colorable
\end{corollary}
\begin{proof}
    We first prove triangle-freeness. Suppose, for contradiction, that $\pgc(G)$ contained a triangle. Let $\Delta\subseteq \pgc(G)$ be a finite graph containing that triangle. By the proof of \thref{finitizer}, there exists a finite $\mathcal{K}_3$-solvable group $H$ with at least two factors isomorphic to a nonabelian simple $K_3$-group such that $\Delta$ is an induced subgraph of $\pgc(H)$. By \cite[Theorem 1.3]{2022REU}, $\pgc(H)$ is triangle-free, so this is a contradiction and thus $\pgc(G)$ is triangle-free. Now let $\Delta$ be a finite induced subgraph of $\pgc(G)$. By the proof of \thref{finitizer}, there exists a finite $\mathcal{K}_3$-solvable group $H$ with at least two factors isomorphic to a nonabelian simple $K_3$-group such that $\Delta\subseteq\pgc(H)$. Applying \cite[Theorem 1.3]{2022REU}, $\pgc(H)$ is 3-colorable so $\Delta\subseteq\pgc(H)$ is 3-colorable. Thus, by \thref{compact}, $\pgc(G)$ is 3-colorable.
\end{proof}
\section{Subgraph Lemmas}\label{subgraphlemmas}
\begin{lemma}\thlabel{disconnect}
    Let $G$ be a locally torsion $T$-solvable group such that $\pgc(G)$ is not triangle-free and 3-colorable. Let $X\subseteq\pi(T)$ be such that the following holds: If $H$ is a finite $T$-solvable group such that $\pgc(H)$ is not triangle-free and 3-colorable, then for all $p\in X$ and $q\in\pi(H)\setminus\pi(T)$, $p-q\not\in\pgc(H)$. Then, for any $p\in X$ and $q\in \pi(G)\setminus\pi(T)$, $p-q\not\in\pgc(G)$.
\end{lemma}
\begin{proof}
    If $X=\varnothing,$ this is immediate, so assume $X\neq\varnothing$. By \thref{compact}, there exists a finite $M\subseteq\pgc(G)$ such that $M$ is not triangle-free and 3-colorable, and by \thref{backwardsDirectionTorsion}, $G$ is strictly $T$-solvable. For contradiction, suppose there exists some $p\in X$ and $q\in\pi(G)\setminus\pi(T)$ such that $p-q\in\pgc(G)$. Then, let $\Delta=\pgc(G)[\pi(T)\cup\{q\}\cup V(M)]$. By \thref{hallAnalogLocallyTorsion}, there exists a finite strictly $T$-solvable group $H$ with $\pgc(H) = \Delta$. Because $M\subseteq\pgc(H)$ is not triangle free and 3-colorable, neither is $\pgc(H)$, so $p-q\not\in\pgc(H)=\pgc(G)[\pi(T)\cup\{q\}\cup V(M)]$, a contradiction.
\end{proof}
\begin{lemma}\thlabel{triangleJail}
    Suppose $T$ is such that, for any finite $T$-solvable group $H$, all triangles of $\pgc(H)$ are contained in $\pgc(H)[\pi(T)]$. Then, for any locally torsion $T$-solvable group $G$, all triangles of $\pgc(G)$ are contained in $\pgc(G)[\pi(T)]$.
\end{lemma}
\begin{proof}
    Suppose $\pgc(G)$ contains a triangle $A$. $G$ must be strictly $T$-solvable by \thref{backwardsDirectionTorsion}. Then, let $\Delta=\pgc(G)[\pi(T)\cup V(A)]$. By \thref{hallAnalogLocallyTorsion}, there exists a finite $T$-solvable group $H$ such that $\pgc(H) = \Delta$. Thus, $A$ is contained in $\pgc(H)[\pi(T)] = \pgc(G)[\pi(T)]$.
\end{proof}
\begin{lemma}\thlabel{oneColoringComplete}
    Fix $k\in\mathbb{N}^+$. Let $G$ be a locally torsion $T$-solvable group such that $\pgc(G)$ is not triangle-free and 3-colorable but is $k$-colorable. Suppose that for every finite $T$-solvable group $H$ such that $\pgc(H)$ is not triangle-free and $3$-colorable, $N(\pi(T))$ in $\pgc(H)$ is monochromatic in some $k$-coloring of $\pgc(H)$. Then, $N(\pi(T))$ in $\pgc(G)$ is monochromatic in some $k$-coloring of $\pgc(G)$.
\end{lemma}
\begin{proof}
    By \thref{solvableTheorem1}, $G$ is not solvable, so $G$ must be strictly $T$-solvable. Let $M$ a finite subgraph of $\pgc(G)$ such that $M$ is not triangle-free and 3-colorable, which exists by \thref{compact}. Let $\Delta\subseteq\pgc(G)$ be a finite subgraph of $\pgc(G)$. Define $\Delta'=\pgc(G)[\pi(T)\cup V(\Delta)\cup V(M)]$ By \thref{hallAnalogLocallyTorsion}, there exists a finite strictly $T$-solvable group $H$ such that $\pgc(H) = \Delta'$. Notice that $M\subseteq\pgc(H)$, so $\pgc(H)$ is not triangle-free and 3-colorable. Then, there exists a $k$-coloring of $\Delta'$ such that all vertices in $N(\pi(T))$ in $\Delta'$ is monochromatic in some $k$-coloring of $\Delta'$, so such a coloring exists for $\Delta\subseteq\Delta'$. Thus, by \thref{kOneColor}, $N(\pi(T))$ in $\pgc(G)$ is monochromatic in a $k$-coloring of $\pgc(G)$.
\end{proof}
\section{$T$-Solvable Torsion Results}\label{TTorsion}
\subsection{Generalized Lemmas}
We first generalize \cite[Lemma 2.3.6]{2023REU}:
\begin{lemma}\thlabel{generalized236}
    Let $T$ be a finite nonabelian simple group, $F$ a finite strictly $T$-solvable group, $\Xi$ a countable locally finite graph, and $X$ a set of $|\pi(F)|$ vertices of $\Xi$. Suppose $\Xi\setminus X$ is triangle-free and $\Xi$ has a 3-coloring $\mathcal{O}, \mathcal{D},\mathcal{I}$ such that the vertices in $N(X)$ are all colored $\mathcal{I}$. Furthermore, suppose that 
    \begin{enumerate}
        \item There exists a graph isomorphism $\varphi$ from $\Xi[X]$ to $\pgc(F)$.
        \item Given $v\in\Xi\setminus X$, there exists a complex irreducible representation $V$ of $F$ such that for all $x\in X, x-v\in\Xi$ if and only if elements of order $\varphi(x)$ of $F$ act without fixed points in $V$.
    \end{enumerate}
    Then there exists a torsion group $G$ with the finite Frobenius property that realizes $\Xi$, and $G$ is $T$-solvable.
\end{lemma}
\begin{proof}
    Direct the arrows with a vertex in $\mathcal{O}$ outward and the arrows with a vertex in $\mathcal{I}$ inward. 
    
     We describe three algorithms, which also appear in \thref{labelingsolvable}. In each algorithm, $x$ is a vertex of $\Xi$ and $R$ is some finite (possibly empty) set of pairs $(r, p)$ for a vertex $r$ and a prime $p$.
    
\hfill\break
$\mathcal{O}$-\textbf{algorithm:} On input $\langle x, R\rangle$: 
\begin{itemize}
    \item[1.] Find the first prime such that $(v,p)\not\in R$ for any $v\in V(\Xi)$, return $(x, p)$.
\end{itemize}
\hfill\break
$\mathcal{D}$-\textbf{algorithm:} On input $\langle x, R\rangle$:
\begin{itemize}
    \item[1.] Check if there exists a prime $p$ such that $(x, p)\in R$. If so, halt and return $R$. Otherwise, move to step 2.
    \item[2.] Find the set $N_{\uparrow}^1(x)$ and enumerate it as $x_1,\dots, x_n$.
    \item[3.] Let $A = \varnothing$. For each $i, 1\leq i\leq n$:
    \begin{itemize}
        \item[a.] If $(x_i, p)\in R$ for some prime $p$, add $(x_i, p)$ to $A$.
        \item[b.] If there exists no prime $p$ such that $(x_i, p)\in R$, choose some prime $p$ such that $(y, p)\not\in R$ for any $y$, and add $(x_i, p)$ to $A$.
    \end{itemize}
    \item[4.] Find a prime $p$ such that $p\equiv1\mod p_i$ for each prime $p_i$ such that $(x_i, p_i)\in A_n$.
    \item[5.] Return $A_n\cup\{(x, p)\}$.
\end{itemize}

\hfill\break
$\mathcal{I}$-\textbf{algorithm:} On input $\langle x, R\rangle$:
\begin{itemize}
    \item[1.] Check if there exists a prime $p$ such that $(x, p)\in R$. If so, halt and return $R$. Otherwise, move to step 2.
    \item[2.] Find the set $N_{\uparrow}^1(x)\cup N_{\uparrow}^2(x)$ and enumerate it as $x_1,\dots, x_n$.
    \item[3.] Let $A_0 = \varnothing$. For each $i, 1\leq i\leq n$:
    \begin{itemize}
        \item[a.] If $(x_i, p)\in R$ for some prime $p$, let $A_i = A_{i-1}\cup\{(x_i, p)\}$.
        \item[b.] If there exists no $p$ such that $(x_i, p)\in A_{i-1}\cup R$ and $x_i$ is colored $\mathcal{O}$, choose some prime $p$ such that $(y, p)\not\in R\cup A_{i-1}$ for any $y$, and let $A_i = A_{i-1}\cup \{(x_i, p)\}$.
        \item[c.] If there exists no $p$ such that $(x_i, p)\in A_{i-1}$ and $x_i$ is colored $\mathcal{D}$, run the $\mathcal{D}$-algorithm on $\langle x_i, A_{i-1}\cup R\rangle$ and let $A_i$ be the set the algorithm returns.
    \end{itemize}
    \item[4.] Enumerate the primes in the second coordinates of elements of $A_n$ by $p_1,\dots, p_n$. Find a prime $p$ such that $p\equiv1\mod \prod_{i=1}^n p_i$.
    \item[5.] Return $A_n\cup\{(x, p)\}$.
\end{itemize}
\hfill\break
We now recursively define a sequence of graphs $\Lambda_i$ such that $X\subseteq\Lambda_i$, $\Lambda_i\subset\Lambda_{i+1}$, and $\bigcup_{i\in\mathbb{N}}\Lambda_i = \Xi$.

\hfill\break
Let $\Lambda_0$ be $\Xi[X]$ labeled according to the isomorphism. Define 
\begin{align*}
    A_0 = \{\langle v, p\rangle:v\text{ is labeled with }p\text{ in }\Lambda_0\}\\
    B_0 = \{ v:\text{ there exists }p\text{ such that }\langle v, p\rangle\in A_0\}\\
    C_0 = \{p:\text{ there exists }v\text{ such that }\langle v,p\rangle\in A_0\}.
\end{align*}

\hfill\break
Suppose $\Lambda_i, A_i, B_i,$ and $C_i$ have been defined for some $i\in\mathbb{N}$. Let $c\in\Xi\setminus \Lambda_i$ if such a $c$ exists. If $c$ has outgoing edges, find a vertex $d$ in $\mathcal{I}$ such that $c$ has a path to $d$, and run the $\mathcal{I}$-algorithm on $\langle d, A_i\rangle$. Call the output $A_{i+1}$, and define:
\begin{align*}
    B_{i+1} = \{ v:\text{ there exists }p\text{ such that }\langle v, p\rangle\in A_{i+1}\}\\
    C_{i+1} = \{p:\text{ there exists }v\text{ such that }\langle v,p\rangle\in A_{i+1}\}.
\end{align*}
Let $\Lambda_{i+1}$ be the labeled graph obtained from applying the labeling given in $A_{i+1}$ to $\Xi[ B_{i+1}]$.

\hfill\break
Notice that, by the finite degree assumption, the algorithms always halt and each $A_i$ is finite. Notice that, by construction, $\bigcup_{i\in\mathbb{N}}\Lambda_i = \Xi_L$ where $\Xi_L$ is the coloring of $\Xi$ given by $\bigcup_{i\in\mathbb{N}}A_i$.

We abuse notation slightly and refer to $\mathcal{O}$ as the set primes $p$ such that there exists a $v\in V(\Xi)$ with $\langle v, p\rangle\in\bigcup_{i\in\mathbb{N}}A_i$ such that $v$ is colored $\mathcal{O}$ in the 3-coloring of $\Xi$. We will refer to $\mathcal{D}$ and $\mathcal{I}$ similarly. Define:
\begin{align*}
    \mathcal{O}_i = \mathcal{O}\cap C_i\\
    \mathcal{D}_i = \mathcal{D}\cap C_i\\
    \mathcal{I}_i = \mathcal{I}\cap C_i.
\end{align*}
We now recursively define a sequence of $T$-solvable groups $\{H_n\}_{n\in\mathbb{N}}$ that realize $\Xi_L[\pi(H_n)]$.
\hfill\break
Let $H_0$ be $F$. By construction, $\Xi_L[\pi(H_0)] = \pgc(H_0)$.

\hfill\break
Suppose $H_i$ has been defined for $0\leq i\leq n$. Let $K^{(n+1)}$ be the solvable group constructed in \thref{locallyTorsionSolvableThm1} realizing $\Xi_L[\mathcal{O}_{n+1}\cup\mathcal{D}_{n+1}]$. Enumerate $\mathcal{I}_{n+1}$ as $r_0,r_1,\cdots, r_{w_{n+1}}$ such that, for $1\leq k\leq w_n, r_k$ is the same as $r_k$ in the enumeration of $\mathcal{I}_{n}$. For each $1\leq d\leq w_{n+1}$, let $L_{r_d}^{(n+1)}, M_{r_d}^{(n+1)}$ be as in \thref{locallyTorsionSolvableThm1}, and let $\delta_{n, r_d}$ be the isomorphism of $K^{(n)}$ and $L_{r_d}^{(n)}\ltimes M_{r_d}^{(n)}$ with the group action conjugation as defined in \thref{locallyTorsionSolvableThm1}. We claim that $F\times K^{(n)}\cong (F\times L_{r_d}^{(n+1)})\ltimes M_{r_d}^{(n+1)}$ where the automorphism corresponding to $(a, b)\in F\times L_{r_d}^{(n+1)}$ is conjugation by $b$. Define $\alpha_{n+1,r_d}:F\times(L_{r_d}^{(n)}\ltimes M_{r_d}^{(n)})\to (F\times L_{r_d}^{(n+1)})\ltimes M_{r_d}^{(n+1)}$ by $\alpha_{n+1,r_d}(a, (b, c)) = ( b,(a, c))$. Notice immediately that this is injective. We claim $\alpha_{n+1,r_d}$ is a homomorphism. Let $(a, (b, c)), (x, (y, z))\in F\times(L_{r_d}^{(n+1)}\ltimes M_{r_d}^{(n+1)})$. We may write:
\begin{align*}
    \alpha_{n+1,r_d}((a, (b, c))(x, (y, z))) = \alpha_{n+1,r_d}(ax,(b,c)(y,z))\\
    = \alpha_{n+1,r_d}(ax,(bcyc^{-1},cz)) = (bcyc^{-1},(ax, cz))\\
    =(bcyc^{-1},(a,c)(x,z)) = (b,(a,c))(y,(x,z)) =  \alpha_{n+1,r_d}(b,(a,c)) \alpha_{n+1,r_d}(y,(x,z)).
\end{align*}
Define $\pi_{n+1,r_d}:(F\times L_{r_d}^{(n+1)})\ltimes M_{r_d}^{(n+1)}\to F\times L_{r_d}^{(n+1)}$ by 
    \begin{align*}
        \pi_{n+1, r_d}(a, (b,c)) = (b,c).
    \end{align*}
    Notice that $\pi_{n+1,r_d}$ is a group homomorphism. 

    Let $d\geq w_n$. We now consider two cases, as in \cite[Lemma 2.3.6]{2023REU}: The case where $r_d$ is adjacent to an element of $\pi(F)$ in $\Xi_L$ and the case where $r_d$ is not adjacent to an element of $\pi(F)$ in $\Xi_L$.

    \hfill\break
    \textbf{Case 1:} By \cite[Lemma 3.5]{a5paper}, there exists a modular representation of $F\times L_{r_d}^{(n+1)}$ over a field of characteristic $r_d$ such that $\operatorname{Fit}(L_{r_d}^{(n+1)})$ acts frobeniously on the module and $F$ acts trivially on the module. Call the module $R_{r_d}$ and the associated action $\xi_{r_d}$.

    \hfill\break
    \textbf{Case 2:} Let $v$ be the vertex corresponding to $r_d$. Let $V$ be the associated irreducible representation. By \cite[Lemma 3.5]{a5paper}, there exists a representation $R_{r_d}$ of $F\times L_{r_d}^{(n+1)}$ over a finite field of characteristic $r_d$ such that $\operatorname{Fit}(L_{r_d}^{(n+1)})$ acts frobeniuously on $R_{r_d}$ and elements of $F$ act fixed point freely on $R_{r_d}$ if and only if they act fixed point freely in $V$. Call the action of $F\times L_{r_d}^{(n+1)}$ on $R_{r_d}$ $\xi_{r_d}$.

 Notice that $\xi_{r_d}^{(n+1)}$ is a homomorphism. 

    For $d<k_{n+1}$, let $\zeta_{n+1,r_d}:F\times L_{r_d}^{(n+1)}\to F\times L_{r_d}^{(n)}$ be a projection operator defined by 
    \begin{align*}
        \zeta_{n+1,r_d}((a,(x_1,\cdots,x_{m_{n+1}}),(y_1,\cdots,y_{u_{n+1}}))) = ((a,(x_1,\cdots,x_{m_n}),(y_1,\cdots,y_{u_n}))).
    \end{align*}
    We claim that $\zeta_{n+1,r_d}$ is a homomorphism. This follows from then fact that $\zeta_{n+1,r_d}$ as defined in \thref{locallyTorsionSolvableThm1} is a homomorphism and the multiplication operations in $F\times L_{r_d}^{(n+1)}$ and $ F\times L_{r_d}^{(n)}$.
    Define $\xi_{r_d}^{(n+1)}$ by $\xi_{r_d}\circ \zeta_{h_{r_d}+1,r_d}\circ\cdots\circ\zeta_{n+1,r_d}\circ\pi_{n+1,r_d}\circ\delta_{n+1,r_d}$ where $h_{r_d}$ is the first index $n$ for which $r_d\in \mathcal{I}_n$. Then, because $\xi_{r_d}^{(n+1)}$ is a homomorphism from $F\times K^{(n+1)}$ to $\operatorname{Aut}(R_{r_d})$, we may define $(F\times K^{(n+1)})\ltimes R_{r_d}$ where the action is $\xi_{r_d}^{(n+1)}$. From there, we may define:
    \begin{align*}K^{(n+1)}\ltimes_{\xi^{(n+1)}}R_{r_0}\times\cdots\times R_{r_{w_{n+1}}}
    \end{align*}
    where $\xi^{(n+1)} = \xi_{r_0}^{(n+1)}\times\cdots\times\xi_{r_{w_{n+1}}}^{(n+1)}$, that is, $\xi_{r_d}^{(n+1)}$ acts only on the coordinates corresponding to $R_{r_d}$. Because, for any $(z,(\overline{x},\overline{y}))\in F\times K^{(n+1)}, \xi_{r_d}^{(n+1)}(\overline{x},\overline{y})$ is an automorphism of $R_{r_d}$, $\xi^{(n+1)} (z,(\overline{x},\overline{y}))$ is an automorphism of $R_{r_0}\times \cdots\times R_{r_{w_{n+1}}}$. We aim to show $\xi^{(n+1)}$ is a homomorphism. Let $(z,(\overline{x},\overline{y})),(c,(\overline{a},\overline{b}))\in F\times K^{(n+1)}$. Then:
    \begin{align*}
        \xi^{(n+1)}(((z,\overline{x},\overline{y}))((c,(\overline{a},\overline{b}))) = \xi_{r_0}^{(n+1)}((z,(\overline{x},\overline{y}))(c,(\overline{a},\overline{b})))\times\cdots\times\xi_{r_{w_{n+1}}}^{(n+1)}((z,(\overline{x},\overline{y}))(c,(\overline{a},\overline{b})))\\
        =\xi_{r_0}^{(n+1)}(z,(\overline{x},\overline{y}))\circ\xi_{r_0}^{(n+1)}(c,(\overline{a},\overline{b}))\times\cdots\times\xi_{r_{w_{n+1}}}^{(n+1)}(z,(\overline{x},\overline{y}))\circ \xi_{r_{w_{n+1}}}^{(n+1)}(c,(\overline{a},\overline{b}))\\
        =(\xi_{r_0}^{(n+1)}(z,(\overline{x},\overline{y}))\times\cdots\times\xi_{r_{w_{n+1}}}^{(n+1)}(z,(\overline{x},\overline{y})))\circ(\xi_{r_0}^{(n+1)}(c,(\overline{a},\overline{b}))\times\cdots\times\xi_{r_{w_{n+1}}}^{(n+1)}(c,(\overline{a},\overline{b})))\\
        =\xi^{(n+1)}(z,(\overline{x},\overline{y}))\circ\xi^{(n+1)}(c,(\overline{a},\overline{b})).
    \end{align*}
    Finally, define $H_{n+1}$ to be the group $(F\times K^{(n+1)})\ltimes_{\xi^{(n+1)}}R_{r_0}\times\cdots\times R_{w_{n+1}}$. We claim $\pgc(H_{n+1})$ is $\Xi_L[\mathcal{O}_{n+1}\cup \mathcal{D}_{n+1}\cup \mathcal{I}_{n+1}]$. Notice that for any $p\in \mathcal{O}_{n+1},q\in \mathcal{D}_{n+1}$, $p-q\in \pgc(H_{n+1})$ if and only if $p-q\in\Xi_L[\mathcal{O}_{n+1}\cup \mathcal{D}_{n+1}\cup \mathcal{I}_{n+1}]$ by  \cite[Corollary 1.5]{Suz} and the analogous fact for $F\times K^{(n+1)}$. Suppose $r_d-q\in \Xi_L[\mathcal{O}_{n+1}\cup \mathcal{D}_{n+1}\cup \mathcal{I}_{n+1}]$ for some $q\in \mathcal{O}_{n+1}\cup \mathcal{D}_{n+1}$ and $0\leq d\leq w_{n+1}$. Then, $q\in \pi(\operatorname{Fit}(L_{r_d}^{(n+1)}))$, so because every element of $L_{r_d}^{(n+1)}\ltimes M_{r_d}^{(n+1)}$ with order $q$ must be of the form $(x,y)$ where $y\in L_{r_d}^{(n+1)}$ has order $q$, every element of order $q$ in $L_{r_d}^{(n+1)}\ltimes M_{r_d}^{(n+1)}$ acts fixed-point freely on $R_{r_d}$, so by  \cite[Corollary 1.5]{Suz}, $r_d-q\in\pgc(H_{n+1})$ (to see this for the case where $d<k_{n+1}$, notice that $\zeta_{h_{r_d}+1,r_d}\circ\cdots\circ\zeta_{n+1,r_d}(y)$ must have order $q$ by choice of labeling). If $r_d-q\in \pgc(H_{n+1})$, then any element of order $q$ in $L_{r_d}^{(n+1)}\ltimes M_{r_d}^{(n+1)}$ must be of the form $(1,(\overline{x}, \overline{y}))$ where $\overline{y}\in L_{r_d}^{(n+1)}$ has order $q$ (and if $d<k_{n+1}$, $\zeta_{h_{r_d}+1,r_d}\circ\cdots\circ\zeta_{n+1,r_d}(1,\overline{y})$ has order $q$), so because $\pgc(H_{n+1})$ is triangle-free, $q$ must correspond to a prime in $\operatorname{Fit}(L_{r_d}^{(n+1)})$ for $d\geq k_{n+1}$ (and if $d<k_{n+1}$, $q$ must correspond to a prime in  $\operatorname{Fit}(L_{h_{r_d}}^{(n+1)})$) by construction. Now suppose $r_d-q\in \Xi_L[\mathcal{O}_{n+1}\cup \mathcal{D}_{n+1}\cup \mathcal{I}_{n+1}\cup \pi(F)]$ for $q\in \pi(F)$. Then, $F$ acts fixed-point freely on $R_{r_d}$, so because $q\nmid|K^{(n)}|$, the only elements of order $q$ are $(x, (1,1))$ where $x$ has order $q$, which is mapped to an automorphism of $R_{r_d}$ which fixes no non-identity points. Thus, no elements of order $qr_d$ exist in $H_{n+1}$ by \cite[Corollary 1.5]{Suz}. If $q-r_d\not\in\Xi_L[\mathcal{O}_{n+1}\cup \mathcal{D}_{n+1}\cup \mathcal{I}_{n+1}\cup \pi(F)]$ for a $q\in\pi(F)$, $F$ acts trivially on $R_{r_d}$, so there exists an element of order $q$ in $R_{r_d}$. Thus, $\pgc(H_{n+1}) = \Xi_L[\mathcal{O}_{n+1}\cup \mathcal{D}_{n+1}\cup \mathcal{I}_{n+1}\cup\pi(F)]$.

\hfill\break
    Finally, for each $i\in \mathbb{N}$, define $\varphi_i:H_i\to H_{i+1}$ by
    \begin{align*}
        \varphi_i(\overline{a}, (d,(\overline{b},\overline{c}))) = ((\overline{a},1,\cdots, 1),(d,(\overline{b},1,\cdots,1),(\overline{c},1,\cdots,1))).
    \end{align*}
    We claim that $\varphi_i$ is an embedding. To see this, we first show that it is a homomorphism. Let $(\overline{a}, (d,(\overline{b},\overline{c}))),(\overline{x},(w,(\overline{y},\overline{z})))\in H_i$. Then:
    \begin{align*}
        (\overline{a}, (d,(\overline{b},\overline{c})))(\overline{x},(w,(\overline{y},\overline{z}))) = (\overline{a}\xi^{(i)}(\overline{b},\overline{c})(\overline{x}),(dw,(\overline{b},\overline{c})(\overline{y},\overline{z})))\\
        =(\overline{a}\xi^{(i)}(\overline{b},\overline{c})(\overline{x}),(dw,(\overline{b}\phi_i(\overline{c})(\overline{y}),\overline{c}\overline{z})))\\
        \varphi((\overline{a}, (d,(\overline{b},\overline{c})))(\overline{x},(w,(\overline{y},\overline{z}))))\\
        = ((\overline{a},1,\cdots,1)(\xi^{(i)}(dw,(\overline{b},\overline{c}))(\overline{x}),1,\cdots,1),(dw,((\overline{b}\phi_i(\overline{c})(\overline{y}),1,\cdots,1),(\overline{c}\overline{z},1,\cdots,1)))).
    \end{align*}
    From here, we may write:
    \begin{align*}
        \varphi_i(\overline{a}, (d,(\overline{b},\overline{c})))\varphi(\overline{x},(w,(\overline{y},\overline{z})))\\
        = ((\overline{a},1,\cdots,1),(d,((\overline{b},1,\cdots,1),(\overline{c},1,\cdots,1))))((\overline{x},1,\cdots,1),(w,((\overline{y},1,\cdots,1),(\overline{z},1,\cdots,1))))\\
        =((\overline{a},1,\cdots,1)\xi^{(i+1)}(dw,((\overline{b},1,\cdots,1),(\overline{c},1,\cdots,1)))(\overline{x},1,\cdots,1),\\
        (dw,((\overline{b},1,\cdots,1)\phi_{i+1}(\overline{c},1,\cdots,1)(\overline{y},1,\cdots,1),(\overline{c}\overline{z},1,\cdots,1))))\\
        =((\overline{a},1,\cdots,1)\xi^{(i+1)}(dw,((\overline{b},1,\cdots,1),(\overline{c},1,\cdots,1)))(\overline{x},1,\cdots,1),\\
        (dw,((\overline{b}\phi_{i}(\overline{c})(\overline{y}),1,\cdots,1),(\overline{c}\overline{z},1,\cdots,1))))\\
        =((\overline{a},1,\cdots,1)(\xi_{r_0}\circ\zeta_{h_{r_0}+1,r_0}\circ\cdots\circ\zeta_{i+1,r_0}\circ\pi_{r_0,i+1}\\
        \circ\delta_{r_0,i+1}(dw,(((\overline{b},1,\cdots,1),(\overline{c},1,\cdots,1))))\times\cdots\\\times\xi_{r_{w_{i+1}}}\circ\zeta_{h_{r_{w_{i+1}}}+1,r_{w_{i+1}}}\circ\cdots\circ\zeta_{i+1,r_{w_{i+1}}}\\\circ\pi_{r_{w_{i+1}},i+1}\circ\delta_{r_{w_{i+1}},i+1}(dw,(((\overline{b},1,\cdots,1),(\overline{c},1,\cdots,1)))))(\overline{x},1,\cdots,1),\\
        (dw,((\overline{b}\phi_{i}(\overline{c})(\overline{y}),1,\cdots,1),(\overline{c}\overline{z},1,\cdots,1))))\\
         =(((\overline{a},1,\cdots,1)(\zeta_{r_0}^{(i)}\times\cdots\times\zeta_{r_{w_{i}}}\times1\times\cdots\times1)(\overline{x},1,\cdots,1)),\\
    (dw,((\overline{b}\phi_{i}(\overline{c})(\overline{y}),1,\cdots,1),(\overline{c}\overline{z},1,\cdots,1))))
    \end{align*}
    \begin{align*}
    =(((\overline{a},1,\cdots,1)(\xi^{(i)}(\overline{x}),1,\cdots,1)),\\
    (dw,((\overline{b}\phi_{i}(\overline{c})(\overline{y}),1,\cdots,1),(\overline{c}\overline{z},1,\cdots,1))))\\
    =((\overline{a}\xi^{(i)}(\overline{x}),1,\cdots,1),  (dw,((\overline{b}\phi_{i}(\overline{c})(\overline{y}),1,\cdots,1),(\overline{c}\overline{z},1,\cdots,1)))).
        \end{align*}

Fix a composition series $1 = M_0\unlhd\cdots\unlhd M_u = F$ of $F$. For each $i\in\mathbb{N}$, fix a subnormal series of $H_i$ such that
\begin{align*}
    1 = N_{0, i}\unlhd N_{1, i}\unlhd\cdots\unlhd N_{u+3, i} = H_i\text{ and }\\
    N_{1,i}/N_{0,1} \cong R_{r_{1,i}}\times\cdots\times R_{r_{I_i, i}}\\
    N_{2, i}/N_{1,i}\cong C_{q_{1, i}}\times\cdots\times C_{q_{D_i, i}}\\
    N_{3, i}/N_{2, i}\cong C_{p_{1,i}}\times\cdots\times C_{p_{O_i,i}}\\
    \text{ for }1\leq j\leq u, N_{j+3, i}/N_{j+2,i}\cong M_{j}/M_{j-1}.
\end{align*}
Notice that, by definition of $\{\varphi_i\}_{i\in\mathbb{N}}$,
\begin{align*}
    \varphi_i(x)\in N_{j, i+1}\text{ if and only if }x\in N_{j, i}.
\end{align*}
For any $i, j\in\mathbb{N}$ with $i\leq j$, define 
\begin{align*}
    f_{i, j} = \varphi_{j-1}\circ\cdots\circ \varphi_i.
\end{align*}
Notice that $f_{i,i} =\operatorname{id}_{H_i}$ by definition. Take $G = \varinjlim H_i$. By \thref{dll} and the construction of $\{H_i\}_{i\in\mathbb{N}}, \pgc(\varinjlim H_i) = \Xi_L$. Notice that this implies $\varinjlim H_i$ has the finite Frobenius property.

By restricting the $g_{i, k}$ as needed, we may apply \thref{dll2} to get
\begin{align*}
    1=\varinjlim N_{0, i}\unlhd \varinjlim N_{1, i}\unlhd\cdots\unlhd \varinjlim N_{u+3, i} = \varinjlim H_i.
\end{align*}
By \thref{dll3}, $\varinjlim N_{1, i}/\varinjlim N_{0, i}$, $\varinjlim N_{2, i}/\varinjlim N_{1, i}$, and $\varinjlim N_{3, i}/\varinjlim N_{2, i}$ are abelian. Suppose $j$ is such that $M_{j}/M_{j-1}$ is abelian. Then, $N_{j + 3,i}/N_{j + 2, i}$ is abelian for all $i\in\mathbb{N}$, so by \thref{dll3}, $\varinjlim N_{j + 3,i}/\varinjlim N_{j + 2, i}$ is abelian for all such $j$. Now suppose $j$ is such that $M_{j}/M_{j-1}\cong T$. Then, $N_{j + 3,i}/ N_{j + 2, i}\cong T$ for all $i\in\mathbb{N}$. By \thref{dll4}, $\varinjlim N_{j + 3,i}/\varinjlim N_{j + 2, i}\cong T$ for all such $j$. Thus, $\varinjlim H_i$ is a $T$-solvable group. Let $[x]\in \varinjlim H_i$. Then, $x\in H_j$ for some $j\in\mathbb{N}$, so it has order $n$ for some $n\in\mathbb{N}^+$. We may write:
\begin{align*}
    [x]^n = [f_{j,j}(x)^n] =[f_{j,j}(x^n)] = [1]
\end{align*}
so $o([x])\mid n$. 
Thus, $\varinjlim H_i$ is a solvable torsion group with the finite Frobenius property such that $\pgc(\varinjlim H_i) = \Xi_L$.
\end{proof}
We now introduce a definition that will be useful later: 
\begin{definition}
    Let $F$ be a finite group. The number $n(F)$ is the least number such that there exists a subnormal series $1 = N_0\unlhd N_1\unlhd\cdots\unlhd N_{n(F)} = F$ such that for each $0\leq i\leq n(F) - 1, N_{i+1}/N_i$ is abelian or isomorphic to a nonabelian simple group.
\end{definition}
\begin{corollary}\thlabel{cor236}
        Let $T$ be a finite nonabelian simple group, $F$ a finite strictly $T$-solvable group, $\Xi$ a countable locally finite graph, and $X$ a set of $|\pi(F)|$ vertices of $\Xi$. Suppose $\Xi\setminus X$ is triangle-free and $\Xi$ has a 3-coloring $\mathcal{O}, \mathcal{D},\mathcal{I}$ such that the vertices in $N(X)$ are all colored $\mathcal{I}$. Furthermore, suppose that 
    \begin{enumerate}
        \item There exists a graph isomorphism $\varphi$ from $\Xi[X]$ to $\pgc(F)$.
        \item Given $v\in\Xi\setminus X$, there exists a complex irreducible representation $V$ of $F$ such that for all $x\in X, x-v\in\Xi$ if and only if elements of order $\varphi(x)$ of $F$ act without fixed points in $V$.
    \end{enumerate}
    Then there exists a labeling of $\Xi, \Xi_L$ such that $\Xi_L[X] = \pgc(F)$. Additionally, $\Xi_L$ is enumerated by a sequence of finite induced subgraphs $\{\Lambda_i\}_{i\in\mathbb{N}}$ with $\bigcup_{i\in\mathbb{N}}\Lambda_i = \Xi_L$ such that for all $i\in\mathbb{N}$, $\Lambda_i\subseteq\Lambda_{i+1}$, $\Xi_L[X]\subseteq \Lambda_i$, and $\pgc(H_i) = \Lambda_i$ for a finite strictly $T$-solvable group $H_i$, and $n(H_i)\leq n(F) + 3$.
\end{corollary}
\begin{proof}
    This follows from the proof of \thref{generalized236}.
\end{proof}
\subsection{Applications}
In all the following theorems, the forward directions holds for any $A_5$-solvable (locally) torsion  group, not just those with the finite Frobenius property. Our arguments only need the property in the construction of a group from a given graph. 

The following theorem generalizes \cite[Theorem A.8]{2023REU}, although it is worth pointing out that the prime graphs of $A_5$-solvable groups were first classified in \cite{a5paper}. As such, we attempt to keep the wording as similar as possible to that used in the proof of \cite[Theorem A.8]{2023REU}. We include some of the tables given in \cite{2023REU} for convenience:
\begin{fact}\thlabel{fpA5}
    \begin{tabular}[h]{|c|c|}
    \hline
        Fixed Point Information &  \\
        \hline
        $A_5$ & $[[2,3,5],[2,3]]$ \\
        \hline
        $S_5$ & $[[2,3,5],[2,3]]$ \\
        \hline
        $\operatorname{SL}(2, 5)$ & $[[2,3,5],[2,3],[3,5],[3],[]]$ \\
        \hline
        $C_2.S_5$& $[[2,3,5],[2,3],[2]]$\\
        \hline
    \end{tabular}\\
\end{fact}
The lists in \thref{fpA5} correspond to the fixed point information of complex irreducible representations of the listed groups. Each sublist corresponds to a representation, and a prime $p\in\pi(A_5)$ appears in a sublist if and only if an element of order $p$ in the representation has nontrivial fixed points. All information can be computed in \cite{GAP} or found in \cite{2023REU}, and the data is organized as in \cite{psl}.

\begin{theorem}
    Let $\Xi$ be a countable locally finite graph. $\Xi$ is realized by an $A_5$-solvable (locally) torsion  solvable group $G$ with the finite Frobenius property if and only if one of the following holds:
    \begin{itemize}
        \item[(1)] $\pgc(G)$ is triangle-free and 3-colorable.
        \item[(2)] There exist vertices $a, b, c$ in $\pgc(G)$ and a 3-coloring for which all vertices adjacent to but not included in $\{a, b,c\}$ have the same color. Additionally, $\pgc(G)$ has at least one triangle and one of the following holds:
        \begin{itemize}
            \item[(2.1)] $\pgc(G)$ has exactly one triangle $\{a, b, c\}$, and vertices $a$ and $b$ are not adjacent to any other vertices.
            \item[(2.2)] All triangles in $\pgc(G)$ are about vertices $\{b, c, d\}$ for vertices $d\neq a$. Additionally, $a-c,a-b\not\in\pgc(G)$, and $N(b)\setminus\{c\}\subset N(a)\cap N(c)$.
            \item[(2.3)] All triangles in $\pgc(G)$ are about vertices $\{b, c, d\}$ for vertices $d\neq a$. Vertex $a$ is isolated and $N(b)\setminus\{c\}\subset N(c)$.
        \end{itemize}
    \end{itemize}
\end{theorem}
\begin{proof}
    We first prove the forwards direction. By \thref{kColorings} and \cite[Theorem A.8]{2023REU}, $\pgc(G)$ is 3-colorable. If $\pgc(G)$ is triangle-free, then (1) is satisfied, so assume $\pgc(G)$ contains at least one triangle. Thus, by \thref{solvableTheorem1}, $G$ is strictly $A_5$-solvable. By \thref{oneColoringComplete} and \cite[Theorem A.8]{2023REU}, $N(\pi(A_5))\setminus\pi(A_5)$ is monochromatic in a 3-coloring of $\pgc(G)$. Let $\Delta$ be a finite subgraph of $\pgc(G)$ containing a triangle and $\pi(A_5)$. Then, by \thref{hallAnalog}, there exists a finite strictly $A_5$-solvable group $H$ with $\pgc(H) = \Delta$. By \cite[Lemma 2.1.2]{2023REU}, there exists a $K\leq H, K\cong N.A_5$ such that $\pgc(K) = \pgc(H)$. Then, by \cite[Lemma A.2]{2023REU}, the edge $3-5$ must be in the triangle. We now consider cases based on what kind of triangles $\pgc(G)$ contains:

    \hfill\break
    \textbf{Case 1:} If $\pgc(G)$ contains a $\{2,3,5\}$-triangle, for any $p\in\pi(G)\setminus\pi(A_5)$, there exists a finite subgraph $\Delta\subseteq\pgc(G)$ containing $p$, the $\{2,3,5\}$-triangle, and $\pi(A_5)$, so there exists a finite strictly $A_5$-solvable group $H$ with $\pgc(H) = \Delta$ by \thref{hallAnalog}. Thus, by \cite[Lemmas A.4 and A.5]{2023REU} show that $2-p, 3-p\not\in\Delta\subseteq\pgc(G)$. $p$ was an arbitrary element of $\pi(G)\setminus\pi(A_5)$, so $2-p,3-p\not\in\pgc(G)$ for any $p\in\pi(G)\setminus\pi(A_5)$. Thus, (2.1) is satisfied.

    \hfill\break
    \textbf{Case 2:} If $\pgc(G)$ has a $\{3, 5, p\}$-triangle for some prime $p\neq 2$, let $\Delta$ be a finite subgraph containing that triangle and $\pi(A_5)$. Then, by \thref{hallAnalog}, there exists a finite strictly $A_5$-solvable group $H$ with $\pgc(H) = \Delta$. By \cite[Lemma A.5]{2023REU}, $2-3,2-5\not\in\Delta\subseteq\pgc(G)$. If $2-r\not\in\pgc(G)$ for all $r\in\pi(G)\setminus\pi(A_5)$, for contradiction, suppose that $N(3)\setminus\{5\}\not\subseteq N(5)$. Then there would exist $p\in\pi(G),p\neq 5$ with $3-p\in\pgc(G), 5-p\not\in\pgc(G)$. Then, let $\Delta$ be a finite subgraph of $\pgc(G)$ containing $3-p$ and $\pi(A_5)$. By \thref{hallAnalog}, there exist a finite, strictly $A_5$-solvable group $H$ with $\pgc(H) = \Delta$, a contradiction to \cite[Lemma A.3]{2023REU}. Thus, $N(3)\setminus\{5\}\not\subseteq N(5)$, and (2.3) is satisfied. The only other subcase we must consider is if $2-r\in\pgc(G)$ for some $r\in\pi(G)\setminus\pi(A_5)$. For $3-s\in\pgc(G),s\in\pi(G)\setminus\pi(A_5)$, let $\Delta$ be a finite subgraph of $\pgc(G)$ containing $\pi(A_5), \{3,5,p\},3-s,$ and $2-r$. Then, there exists a finite, strictly $A_5$-solvable group $H$ with $\pgc(H) = \Delta$ by \thref{hallAnalog}, so by \cite[Lemma A.6]{2023REU}, $2-s,5-s\in\Delta\subseteq\pgc(G)$. Notice that $s$ was an arbitrary prime in $\pi(G)\setminus\pi(A_5)$, so if $3-s\in\pgc(G), 2-s,5-s\in\pgc(G)$. Thus, $N(3)\setminus\{5\}\subseteq N(2)\cap (5)$, so $\pgc(G)$ satisfies (2.2).

    Let $X = \{a, b, c\}$. We now turn to the backwards direction. If (1) is satisfied, then $\Xi$ is realized by a solvable group by \thref{solvableTheorem1}. 
    
    If (2.1) is satisfied, notice that there is a graph isomorphism from $\Xi[X]$ to $\pgc(A_5)$. For each $v\in N(X)$, our hypothesis states that $v$ is adjacent to $c$ but not $a$ or $b$. Note that there is a complex irreducible representation of $A_5$ in which order 2 and 3 elements have fixed points but order 5 elements do not. By \thref{generalized236}, $\Xi$ is thus realizable by an $A_5$-solvable group.

    If (2.2) is satisfied, notice that there is a graph isomorphism from $\Xi[X]$ to $\pgc(\operatorname{SL(2, 5))}$, which is $A_5$-solvable by \cite[Lemma A.4]{2023REU}. For each $v\in N(X)$, our hypothesis states that one of the following holds:
    \begin{itemize}
        \item[(1)] $v$ is adjacent to $a$ but not $b$ or $c$.
        \item[(2)] $v$ is adjacent to $c$ but not $a$ or $b$.
        \item[(3)] $v$ is adjacent to $a$ and $c$ but not $b$.
        \item[(4)] $v$ is adjacent to $a, b, $ and $c$.
    \end{itemize}
    As we can see in \thref{fpA5}, $\operatorname{SL}(2, 5)$ has complex irreducible representations with these conditions. Thus, $\Xi$ is realizable by an $A_5$-solvable group by \thref{generalized236}.

    Finally, suppose (2.3) holds. We have a graph isomorphism from $\Xi[X]$ to $\pgc(2.S_5)$. For each $v\in N(X)$, our hypothesis states that either $v$ is adjacent to $c$ but not $b$ or $v$ is adjacent to $c$ and $b$. By \thref{fpA5}, there are complex irreducible representations of $C_2.S_5$ (double cover of $S_5$) corresponding to these relations, so $\Xi$ is the prime graph of an $A_5$-solvable group.
\end{proof}
The following theorem generalizes \cite[Theorem 3.1.6]{2023REU}. As such, we attempt to keep the wording as similar as possible to that used in the proof of \cite[Theorem 3.1.6]{2023REU}. We include some of the tables given in \cite{2023REU} for convenience:
\begin{fact}\thlabel{fpA7}
    \begin{tabular}[h]{|c|c|}
    \hline
        Fixed Point Information &  \\
        \hline
        $A_7, 3.A_7, 16.A_7, 64.A_7$ & $[[2,3,5,7],[2,3,5]]$ \\
        \hline
        $2.A_7$ & $[[2,3,5,7],[2,3,5],[3,5,7],[3,7]]$ \\
        \hline
        $6.A_7$ & $[[2,3,5,7],[2,3,5],[3,5,7],[3,5],[3,7]]$ \\
        \hline
    \end{tabular}\\
\end{fact}
The lists in \thref{fpA7} correspond to the fixed point information of complex irreducible representations of the listed groups. Each sublist corresponds to a representation, and a prime $p\in\pi(A_7)$ appears in a sublist if and only if an element of order $p$ in the representation has nontrivial fixed points. All information can be computed in \cite{GAP} or found in \cite{2023REU}, and the data is organized as in \cite{psl}.

\begin{theorem}
    Let $\Xi$ be a countable locally finite graph.  There exists a (locally) torsion $A_7$-solvable group $G$ with the finite Frobenius property realizing $\Xi$ if and only if one of the following holds:
    \begin{enumerate}
        \item[(1)] $\Xi$ is triangle-free and 3-colorable.
        \item[(2)] There exist vertices $a, b, c, d$ in $\Xi$ and a 3-coloring of $\Xi$ such that all vertices adjacent to but not included in $\{a, b, c, d\}$ have the same color. Additionally, one of the following holds:
        \begin{itemize}
            \item[(2.1)] $\Xi$ contains exactly two triangles $\{a, b, c\}$ and $\{b,c,d\}$. All edges incident to $a, b, d$ are included in these two triangles.
            \item[(2.2)] $\pgc(G)$ contains exactly one triangle $\{a, b, c\}$ and a vertex $d$ adjacent only to vertex $c$. Vertices $a$ and $b$ are not adjacent to any other vertices.
            \item[(2.3)] $\Xi$ contains exactly one triangle $\{a, b, c\}$ and a vertex $d$ that is adjacent only to vertex $b$. Vertices $a$ and $b$ are not adjacent to any other vertices.
            \item[(2.4)] $\Xi$ contains exactly one triangle $\{a, b, c\}$ and a vertex $d$. Vertex $a$ is not adjacent to any other vertices. $d$ is not adjacent to $a,b,$ or $c$. Additionally, $N(b)\setminus\{a,c\}\subseteq N(d)$, and $N(c)\cap N(d) = \varnothing$.
            \item[(2.5)] $\Xi$ contains exactly one triangle $\{a,b,c\}$ and an isolated vertex $d$. The vertex $a$ is not adjacent to any other vertex.
        \end{itemize}
    \end{enumerate}
\end{theorem}
\begin{proof}
    We first prove the forwards direction.     By \thref{kColorings} and \cite[Theorem 3.1.6]{2023REU}, $\pgc(G)$ is 3-colorable. If $\pgc(G)$ is triangle-free, (1) is satisfied, so assume $\pgc(G)$ contains a triangle. By \thref{triangleJail} and \cite[Lemma 3.1.2]{2023REU}, all triangles in $\pgc(G)$ are contained in $\pi(A_7)$. By \thref{disconnect}, \thref{fpA7}, and  \cite[Corollary 2.2.9]{2023REU}, there are no edges from $3$ to vertices in $\pi(G)\setminus\pi(A_7)$. Applying \thref{oneColoringComplete} with \cite[Theorem 3.1.6]{2023REU}, there exists a 3-coloring in which all vertices adjacent to but not included in $\pi(A_7)$ have the same color. Notice that $\pgc(G)[\pi(A_7)] = \pgc(H)$ for some finite $A_7$-solvable group by \thref{hallAnalog}, so $\pgc(G)[\pi(A_7)]$, and thus $\pgc(G)$, contains at most two triangles. Because $\pgc(G)[\pi(A_7)] = \pgc(H)$ for a finite $A_7$-solvable group $H$, by \cite[Lemma 3.1.2]{2023REU}, those triangles are $\{3, 5, 7\}$ and $\{3, 7, 2\}$.  
    Now notice that at least one of $5-2,5-3$ must be in $\pgc(G)$.
 We proceed by cases on the edges in $\pgc(G)$:

    \hfill\break
    \textbf{Case 1:} $5-2, 5-3\in\pgc(G)$: If there existed an edge $5-p$ or $2-p$ in $\pgc(G)$ for $p\in\pi(G)\setminus\pi(A_7)$, then there would exist a  a finite $A_7$-solvable group $H$ such that $\pgc(H) = \pgc(G)[\pi(A_7)\cup\{p\}]$, a contradiction to the proof of \cite[Theorem 3.1.6]{2023REU}, so no such edges can exist. Notice that, as $\pgc(G)[\pi(A_7)]$ contains a triangle, $7-2$ or $7-3$ is in $\pgc(G)$. If both are in $\pgc(G)$, (2.1) holds. Otherwise, (2.3) is satisfied.

    \hfill\break
    \textbf{Case 2:} $5-2\in\pgc(G), 5-3\not\in\pgc(G)$. Let $\Delta$ be any finite subgraph of $\pgc(G)$ containing $\pi(T)$. Then, there exists a finite strictly $A_7$-solvable group $H$ with $\pgc(H) = \Delta$, so $2-p\not\in\pgc(G)$ for any $p\in\pi(G)\setminus\pi(A_7)$ or else we would contradict case 2 of the proof of \cite[Theorem 3.1.6]{2023REU}. Similarly, if $7-3\in\pgc(G)$, then $5-p\not\in\pgc(G)$ for $p\in\pi(G)\setminus\pi(A_7)$. Thus, (2.2) holds. If $7-3\in\pgc(G)$ and if $\pgc(G)[\pi(A_7)]$ contained an edge $5-3$ or $7-3$, there would be a finite $A_7$-solvable group $H$ such that $\pgc(H) = \pgc(G)[\pi(A_7)]$ by \thref{hallAnalog}, a contradiction to the proof of \cite[Theorem 3.1.6]{2023REU}. Then, (2.5) is satisfied.

    \hfill\break
    \textbf{Case 3:} $5-3\in\pgc(G), 5-2\not\in\pgc(G)$: In the case that $7-2\in\pgc(G)$, similarly to case 1, we have $5-p, 2-p\not\in\pgc(G)$ for all $p\in\pi(G)\setminus\pi(A_7)$, os $\pgc(G)$ satisfies (2.2). If $7-2\not\in\pgc(G)$ and $2-p\in\pgc(G)$ for some $p\in\pi(G)\setminus\pi(A_7)$ for a $p\in\pi(G)\setminus\pi(A_7)$, we consider two cases for contradiction:

    \hfill\break
    \textbf{Subcase 1:} $N(5)\setminus\{3, 7\}\not\subseteq N(2)$: Then there exists a $q\in\pi(G)$ such that $5-q\in\pgc(G), 2-q\not\in\pgc(G)$. There then exists a strictly $A_7$-solvable finite group $H$ with $\pgc(H) = \pgc(G)\cup[\pi(A_7)\cup\{q\}]$ by \thref{hallAnalog}. This contradicts the proof of \cite[Theorem 3.1.6]{2023REU}. 

    \hfill\break
    \textbf{Subcase 2:} $N(7)\cap N(2)\neq\varnothing$: If there existed $q\in N(7)\cap N(2)$, there exists strictly $A_7$-solvable finite group $H$ with $\pgc(H) = \pgc(G)\cup[\pi(A_7)\cup\{q\}]$ by \thref{hallAnalog}. This contradicts the proof of \cite[Theorem 3.1.6]{2023REU}.

    \hfill\break
    Thus, $\pgc(G)$ must satisfy (2.4). Now, if $7-2,2-p\not\in\pgc(G)$ for all $p\in\pi(G)\setminus\pi(A_7)$, $\pgc(G)$ has exactly one triangle $\{3,5,7\}$, and thus (2.5) is satisfied.

    We now turn to the backwards direction. If (1) is satisfied, $\Xi$ is realized by an $A_7$-solvable group by \thref{solvableTheorem1}. (2.1), (2.2), (2.3), and (2.4) can all be proved by applying \thref{generalized236} to $E$, where we define $F$ to be $A_7, 64.A_7, 16.A_7, $ and $2.A_7$ (the double cover of $A_7$) respectively after choosing the correct representations from the table \thref{fpA7}. Because these cases are similar, we will only prove (2.4) in detail. Make the assignments $a\mapsto 3, b\mapsto 5, c\mapsto7, $ and $d\mapsto 2$. Then, $F = 2.A_7$, the double cover of $A_7$, realizes $\Xi[\{a, b, c, d\}]$ as its prime graph is equal to $\Xi[\{a, b, c, d\}]$ under the labeling. For each $v\in N(X)$, our hypothesis states that one of the following holds:
    \begin{itemize}
        \item[(1)] $v$ is adjacent to $b$ and $d$ but not $a$ or $c$.
        \item[(2)] $v$ is adjacent to $d$ but not $a, b, $ or $c$.
        \item[(3)] $v$ is adjacent to $c$ but not $a, b, $ or $d$.
    \end{itemize}
    Consider the complex irreducible representations of $F$ described in \thref{fpA7}. If (1) holds, note that there is a representation in which elements of order 3 and 7 have fixed points where elements of order 2 and 5 do not. If (2) holds, take the representation in which elements of order 3,5, and 7 have fixed points and order 2 elements do not. If (3) holds, select the representation in which elements of order 2,3, and 5 have fixed points while order 7 elements do not. Applying \thref{generalized236}, $\Xi$ is the prime graph of a $A_7$-solvable group.
\end{proof}

The following theorem generalizes \cite[Theorem 3.4.1]{2023REU}. As such, we attempt to keep the wording as similar as possible to that used in the proof of \cite[Theorem 3.4.1]{2023REU}. We include some of the tables given in \cite{2023REU} for convenience:
\begin{fact}\thlabel{fpinfo1}
    \begin{tabular}[h]{|c|c|}
    \hline
        Fixed Point Information &  \\
        \hline
        $\operatorname{PSL(3, 5)}$ & $[[2,3,5,31],[2,3,5]]$ \\
        \hline
        $\operatorname{PSL(3, 8)}$ & $[[2,3,7,73],[2,3,7]]$ \\
        \hline
        $\operatorname{PSL(3, 17)}$ & $[[2,3,17,307],[2,3,17]]$ \\
        \hline
    $U_3(4)$ & $[[2,3,5,13],[2,3,5]]$ \\
    \hline
    $U_3(7)$ & $[[2,3,7,43],[2,3,7]]$ \\
    \hline
    $U_5(2)$ & $[[2,3,5,11],[2,3,5]]$\\
    \hline
    \end{tabular}\\
\end{fact}
The lists in \thref{fpinfo1} correspond to the fixed point information of complex irreducible representations of the listed groups. Each sublist corresponds to a representation, and a prime $p\in\pi(T)$ appears in a sublist if and only if an element of order $p$ in the representation has nontrivial fixed points. All information can be computed in \cite{GAP} or found in \cite{2023REU}, and the data is organized as in \cite{psl}.

 We include the table from the proof of \cite[Theorem 3.4.1]{2023REU} for convenience:
\begin{fact}\thlabel{table1}
    \begin{tabular}[h]{|c|c|c|c|c|}
    \hline
        $T$ &  $s_1$&$s_2$&$s_3$&$s_4$\\
        \hline
        $\operatorname{PSL(3, 5)}$ & 31&3&5&2 \\
        \hline
        $\operatorname{PSL(3, 8)}$ & 73& 3& 2& 7 \\
        \hline
        $\operatorname{PSL(3, 17)}$ & 307 & 3 & 17 &2 \\
        \hline
    $U_3(4)$ & 13 & 3&2 &5 \\
    \hline
    $U_3(7)$ & 43& 3& 7& 2 \\
    \hline
    $U_5(2)$ & 11 & 5 & 2 & 3\\
    \hline
    \end{tabular}\\
\end{fact}
\begin{theorem}
    Let $T$ be one of $\operatorname{PSL}(3, 5), \operatorname{PSL}(3, 8),\operatorname{PSL}(3, 17), U_3(4), U_3(7),$ or $U_5(2)$ and let $\Xi$ be a countable locally finite graph. $\Xi$ is realized by a (locally) torsion $T$-solvable group with the finite Frobenius property if and only if $\Xi$ is 3-colorable and one of the following holds:
    \begin{enumerate}
        \item $\Xi$ is triangle-free.
        \item $\Xi$ contains exactly one triangle $\{w, x, y\}$ and an isolated vertex $z$ such that $w, x$ are not adjacent to other vertices in $\Xi$. Additionally, $\Xi$ has a 3-coloring for which all neighbors of $y$ other than $w$ and $x$ have the same color.
        \item $\Xi$ contains exactly one triangle $\{w, x, y\}$ and the vertices $w, x$ are not adjacent to any other vertices in $\Xi$. $\Xi$ has a vertex $z$ adjacent only to vertex $y$. Additionally, $\Xi$ has a 3-coloring for which all neighbors of $y$ other than $w, x,$ and $z$ have the same color.
    \end{enumerate}
\end{theorem}
\begin{proof}
    We first consider the forward direction. If $\Xi$ is 3-colorable and triangle-free, (1) is satisfied, so we may assume otherwise. Thus, by \thref{solvableTheorem1}, $G$ is strictly $T$-solvable. Let $X = \pi(T)$. We first prove that $\pgc(G)$ is 3-colorable. Let $\Delta\subseteq\pgc(G)$ be a finite subgraph. By \thref{hallAnalog}, there exists a finite $T$-solvable group realizing $\Delta\cup\pgc(G)[\pi(T)]$, and by \cite[Theorem 3.4.1]{2023REU}, $\Delta\subseteq \Delta\cup\pgc(G)[\pi(T)]$ is 3-colorable. Thus, $\pgc(G)$ is 3-colorable by \thref{compact}. By our assumptions on $\pgc(G)$, $\pgc(G)$ must contain a triangle. We now show that the only triangles of $\pgc(G)$ occur in $\pgc(G)[\pi(T)]$. For contradiction, suppose otherwise. Then there exists a $\Delta\subseteq\pgc(G)$ containing the vertices of the triangle and $\pi(T)$. By \thref{hallAnalog}, there exists a finite $T$-solvable group $H$ such that $\pgc(H) = \Delta$, so by \cite[Theorem 3.4.1]{2023REU}, all triangles of $\Delta$ are contained in $\Delta[\pi(T)]$, a contradiction. We now show that there exists no edge  $m - p\in\pgc(G)$ for $m\in\{w, x, z\}$. For contradiction, suppose there existed an edge $m - p\in\pgc(G)$ for $m\in\{w, x, z\}$, where $w, x, y, z$ are labeled appropriately with $\pi(T)$ according to \cite[Theorem 3.4.1]{2023REU} (\thref{table1}). Let $\Delta$ be a finite subgraph of $\pgc(G)$ containing the edge and $\pgc(G)[\pi(T)]$. By \thref{hallAnalog}, there exists a finite $T$-solvable group $H$ such that $\pgc(H) = \Delta$, a contradiction to \cite[Theorem 3.4.1]{2023REU}.

    Now, let $\Delta\subseteq\pgc(G)$ be a finite subgraph such that $\pi(T)\subseteq V(\Delta)$ and $V(\Delta)\cap N(\pi(T))\neq\varnothing$. By \thref{hallAnalog}, there exists a finite strictly $T$-solvable group realizing $\Delta$, so because $\Delta$ contains a triangle, there exists a 3-coloring of $\Delta$ such that all neighbors of $\pi(T)$ other than those in $\pi(T)$ are monochromatic by \cite[Theorem 3.4.1]{2023REU}. Thus, by \thref{kOneColor}, all neighbors of $\pi(T)$ other than those in $\pi(T)$ are monochromatic.
    
    Notice that $\pgc(G)[\pi(T)]$ must contain a triangle, so because $\pgc(G)[\pi(T)]\subseteq\pgc(T)$, $\pgc(G)[\pi(T)]$ either is a triangle and an isolated vertex, or a triangle such that exactly one vertex of the triangle has an edge with the vertex outside the triangle. In the first case, we have shown that (2) is satisfied, and in the second case, we have shown that (3) is satisfied.

    We now turn to the reverse direction. If $\Xi$ satisfies (1), then by \thref{solvableTheorem1}, there exists a triangle-free and 3-colorable solvable (and thus $T$-solvable) group realizing $\Xi$.

    Suppose $\Xi$ satisfies (3). Label $w\mapsto s_4, x\mapsto s_2, y\mapsto s_1, z\mapsto s_3$ as in \thref{table1} (\cite[Theorem 3.4.1]{2023REU}). Notice that $\pgc(T) = \Xi[X]$ under this labeling. Now consider the complex irreducible representation of $T$ in \thref{fpinfo1}. There exists a representation for which elements of order $s_1,s_3$ and $s_4$ have fixed points but elements of order $s_2$ do not. By \thref{generalized236}, there exists a torsion $T$-solvable group $G$ with the finite Frobenius property such that $G$ realizes $\Xi$.

    We finally consider (2). Let $\Xi_0$ be the graph obtained from $\Xi$ by introducing the edge $s_3-s_1$. Then, $\Xi_0$ satisfies (3), so there exists a torsion $T$-solvable group $H$ with the finite Frobenius property such that $H$ realizes $\Xi_0$. Thus, $H\times C_{s_3}$ is a torsion $T$-solvable group with the finite Frobenius property realizing $\Xi$.
\end{proof}
The next result is a generalization of, and parts of it follow similarly to \cite[Theorem 3.5.3]{2023REU} (under the generalizations of certain lemmas). When appropriate, we use the wording and notation from the proof of \cite[Theorem 3.5.3]{2023REU}. We include a table of fixed point information, which was already calculated in \cite{2023REU}, for convenience:
\begin{fact}\thlabel{fpinfoM11}
    \begin{tabular}[h]{|c|c|}
    \hline
        Fixed Point Information &  \\
        \hline
        $M_{11}$ & $[[2,3,5,11],[2,3,5]]$ \\
        \hline
    \end{tabular}\\
\end{fact}
The lists in \thref{fpinfo1} correspond to the fixed point information of complex irreducible representations of the listed groups. Each sublist corresponds to a representation, and a prime $p\in\pi(T)$ appears in a sublist if and only if an element of order $p$ in the representation has nontrivial fixed points. All information can be computed in \cite{GAP} or found in \cite{2023REU}, and the data is organized as in \cite{psl}.
\begin{theorem}
    Let $\Xi$ be a countable locally finite graph. $\Xi$ is isomorphic to the prime graph of a (locally) torsion $M_{11}$-solvable group with the finite Frobenius property if and only if one of the following is satisfied:
    \begin{enumerate}
        \item $\Xi$ is triangle-free and 3-colorable.
        \item There exist vertices $a, b, c, d$ in $\Xi$ and a 3-coloring for which all vertices adjacent to but not included in $\{a, b, c, d\}$ have the same color. Additionally, one of the following holds:
        \begin{itemize}
            \item[2.1] $\Xi$ contains exactly two triangles $\{a, b, c\}$ and $\{b,c,d\}$. All edges incident to $a, b, d$ are included in these two triangles.
            \item[2.2] $\Xi$ contains exactly one triangle $\{a, b, c\}$ and a vertex $d$ adjacent only to vertex $c$. Vertices $a$ and $b$ are not adjacent to any other vertices.
            \item[2.3] $\Xi$ contains exactly one triangle $\{a, b, c\}$ and an isolated vertex $d$. All edges incident to $a, b$ are included in this triangle.
        \end{itemize}
    \end{enumerate}
\end{theorem}
\begin{proof}
    If $\pgc(G)$ is triangle-free and 3-colorable, (1) is satisfied, so assume otherwise. By \thref{kColorings}, $\pgc(G)$ is $3$-colorable. By \thref{kOneColor} and \cite[Theorem 3.5.3]{2023REU}, all vertices incident to $\pi(T)$ that are not in $\pi(T)$ is monochromatic in some 3-coloring of $\pgc(G)$. By \thref{disconnect} and \cite[Lemma 3.5.1]{2023REU}, the only vertex in $\pi(T)$ that can have edges to vertices in $\pi(G)\setminus\pi(T)$ is 11, so all edges incident to 2, 3, or 5 are contained in $\pi(T)$.  Because we have shown that $\pgc(G)$ is 3-colorable, it must contain a triangle. By \thref{triangleJail} and \cite[Theorem 3.5.3]{2023REU}, all triangles of $\pgc(G)$ are contained in $\pgc(G)[\pi(T)]$. Now, notice that the only thing remaining to prove is that $\pgc(G)[\pi(T)]$ corresponds to one of the configurations of $\{a, b, c, d\}$ described in (2.1)-(2.3). Notice that, by \thref{hallAnalog} there exists a $T$-solvable group $H$ such that $\pgc(G)[\pi(T)] = \pgc(H)$, so by \cite[Theorem 3.5.3]{2023REU}, there exists a labeling of $\pgc(G)[\pi(T)]$ with $a, b, c, d$ such that one of (2.1), (2.2), or (2.3) is satisfied.

    We now consider the backwards direction. If $\Xi$ satisfies (1), then there exists a $M_{11}$-solvable group realizing $\Xi$ by \thref{solvableTheorem1}, so suppose $\Xi$ satisfies (2).
    If (2.1) or (2.2) holds, there exists a $M_{11}$-solvable group realizing $\Xi$ by \thref{generalized236} with $F=M_{11}$ in (2.1) and $F = 3^5.M_{11}$ (this group is also ``PerfectGroup(1924560,1)" in \cite{GAP}, see the proof of \cite[Theorem 3.5.3]{2023REU}) in (2.2) (see \thref{fpinfoM11} for the representation theoretic justification, and the proof of \cite[Theorem 3.5.3]{2023REU} for the appropriate labeling of $\{a, b, c,d\}$ with $\pi(M_{11})$). Now, if (2.3) holds, then consider $\Xi_0$, the graph which is $\Xi$ with the additional edge $c-d$. By (2.2), there exists a torsion $T$-solvable group $G$ with the finite Frobenius property such that $G$ realizes $\Xi_0$, so $G\times C_d$ is a torsion $T$-solvable group with the finite Frobenius property realizing $\Xi$.
\end{proof}

The next result is a generalization of, and parts of it follow similarly to \cite[Theorem 3.5.4]{2023REU} (under the generalizations of certain lemmas). When appropriate, we use the wording and notation from the proof of \cite[Theorem 3.5.4]{2023REU}. We include a table of fixed point information, which was already calculated in \cite{2023REU}, for convenience:
\begin{fact}\thlabel{fpinfoM12}
    \begin{tabular}[h]{|c|c|}
    \hline
        Fixed Point Information &  \\
        \hline
        $M_{12}$ & $[[2,3,5,11]]$ \\
        \hline
        $2.M_{12}$& $[[2,3,5,11],[2,3,5]]$\\
        \hline
    \end{tabular}\\
\end{fact}
The lists in \thref{fpinfo1} correspond to the fixed point information of complex irreducible representations of the listed groups. Each sublist corresponds to a representation, and a prime $p\in\pi(T)$ appears in a sublist if and only if an element of order $p$ in the representation has nontrivial fixed points. All information can be computed in \cite{GAP} or found in \cite{2023REU}, and the data is organized as in \cite{psl}. $2.M_{12}$ denotes any perfect extension of $M_{12}$ by $C_2$.
\begin{theorem}
    Let $\Xi$ be a countable locally finite graph. $\Xi$ is realized by a (locally) torsion $M_{12}$-solvable group with the finite Frobenius property if and only if $\Xi$ is 3-colorable and one of the following holds:
   \begin{enumerate}
       \item $\Xi$ is triangle-free.
       \item $\Xi$ has exactly one triangle $\{a, b, c\}$, vertex $c$ is adjacent to exactly on other vertex $d$, and $d$ is not adjacent to vertices other than $c$. Vertices $a, b, c, d$ are not adjacent to any other vertices.
       \item $\Xi$ has exactly one triangle $\{a, b, c\}$ and an isolated vertex $d$ such that $a, b$ are not adjacent to other vertices in $\Xi$. Additionally, $\Xi$ has a 3-coloring for which all neighbors of $c$ other than $a$ and $b$ have the same color.
   \end{enumerate}
\end{theorem}
\begin{proof}
    We first prove the forwards direction. If $\pgc(G)$ is triangle-free and 3-colorable, (1) holds, so assume $\pgc(G)$ is not triangle-free and 3-colorable. By \thref{kColorings} and \cite[Theorem 3.5.4]{2023REU}, $\pgc(G)$ is 3-colorable. Thus, $\pgc(G)$ must contain a triangle. By \thref{triangleJail} and \cite[Theorem 3.5.4]{2023REU}, all triangles of $\pgc(G)$ are contained in $\pgc(G)[\pi(M_{12})]$. We may apply \thref{disconnect} and \cite[Lemma 3.5.1]{2023REU} to see that all edges incident to 2, 3, and 5 are in $\pgc(G)[\pi(M_{12})]$. By \thref{oneColoringComplete} and \cite[Lemma 3.5.1]{2023REU}, $N(\pi(M_{12}))\setminus\pi(M_{12})$ is monochromatic in some 3-coloring of $\pgc(G)$. Because $2-3, 2-5\not\in\pgc(M_{12})$, $\pgc(G)[\pi(T)]$ is a subgraph of the graph which contains the triangle $\{3, 5, 11\}$ and has the edge $11-2$. Because $\pgc(G)$ must contain at least one triangle and all triangles of $\pgc(G)$ are contained in $\pgc(G)[\pi(T)]$, the triangle $\{3,5,11\}$ must exist, so (2) or (3) is satisfied.

    We now turn to the backwards direction. If $\Xi$ satisfies (1), by \thref{solvableTheorem1}, there exists a $T$-solvable group realizing $\Xi$. If $\Xi$ satisfies (2), by using techniques from \thref{solvableTheorem1}, we can find a solvable group $H$ such that $H$ with $\pi(H)\cap\pi(M_{12})=\varnothing$ such that $H$ realizes $\Xi\setminus\{a, b, c, d\}$ as $\Xi\setminus\{a, b, c, d\}$ is triangle-free and 3-colorable. Thus, $H\times T$ realizes $\Xi$. Now suppose that $\Xi$ satisfies (3), and let $X = \{a, b, c, d\}$. Under the labeling $a\mapsto 3, b\mapsto 5, c\mapsto 11, d\mapsto 2$, we get that $\Xi[X]$ is isomorphic to $\pgc(2.M_{12})$, where $2.M_{12}$ is the double cover of $M_{12}$ (see the proof of \cite[Theorem 3.5.4]{2023REU}). For each $v\in N(X)$, $v$ is adjacent to $c$ but none of the other vertices in $X$. By \thref{fpinfoM12}, there is a complex irreducible representation in which order 2,3, and 5 elements have nontrivial fixed points but order 11 elements do not. Thus, by \thref{generalized236}, there exist a $M_{12}$-solvable group realizing $\Xi$.
\end{proof}

The next result is a generalization of, and parts of it follow similarly to \cite[Theorem 3.6]{psl} (under the generalizations of certain lemmas). When appropriate, we use the wording and notation from the proof of \cite[Theorem 3.6]{psl}. We also used the figure from \cite{psl}.

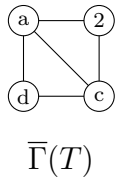
\begin{figure}[H]
    \centering
    \begin{minipage}{\textwidth}
        \centering
        \begin{tikzpicture}
            \coordinate (A) at (0, 1);   
            \coordinate (B) at (1, 1);   
            \coordinate (C) at (1, 0);   
            \coordinate (D) at (0, 0);   
            \draw (B) -- (C);
            \draw (C) -- (D);
            \draw (A) -- (D);
            \draw (A) -- (B);
            \draw (A) -- (C);
            \foreach \point in {A, B, C, D}
                \draw[draw, fill=white] (\point) circle [radius=0.2cm];
            \foreach \point/\name in {(A)/a, (B)/2, (C)/c, (D)/d}
                \node at \point {\scriptsize\name};
        \end{tikzpicture}
        \caption*{$\pgc(T)$}
    \end{minipage}%
\end{figure}

Let $T$ be a nonabelian simple group with $\pi(T) = \pi(\operatorname{Aut}(T)) = \{2, a, c, d\}$ such that $T$ has the prime graph listed above, and such that $T$ satisfies the following criteria:
\begin{enumerate}
    \item The Schur multiplier of $T$ is 2.
    \item The fixed point information for $T$ matches \thref{fpinfoPSL227}.
    \item There exists an $R$ with $\operatorname{Inn}\leq R\leq \operatorname{Aut}(T)$ such that $\pgc(R)$ realizes \SuzukiThirtyTwoElusiveGraph.
\end{enumerate}
We include a table of fixed point information, which was already determined in \cite{psl}, for convenience:
\begin{fact}\thlabel{fpinfoPSL227}
    \begin{tabular}[h]{|c|c|}
    \hline
        Fixed Point Information &  \\
        \hline
        $T$ & $[[2,a,c,d]]$ \\
        \hline
        $2.T$& $[[2,a,c,d],[a,c,d]]$\\
        \hline
    \end{tabular}\\
\end{fact}

\begin{theorem}\thlabel{generalThm}
Given a countable locally finite graph $\Xi$, we have that $\Xi$ is isomorphic to the prime graph complement of some (locally) torsion $T$-solvable group with the finite Frobenius property if and only if one of the following is true: 
\begin{enumerate}
    \item $\Xi$ is triangle-free and 3-colorable.
    \item $\Xi$ is 3-colorable and contains a subset $X = \{w, x, y, z\} \subseteq V(\Xi)$ with $\Xi[X] \cong \notIsolatedVertexWhiteBad$ where $w$ is the white vertex; for all $p\in \pi(G)\setminus \pi(T)$ such that $p-q\in\pgc(G)$ and $q\in\pi(T)$, we have that $q = w$; and that in some 3-coloring of $\Xi$ in which the closed neighborhood $N(X) \setminus X$ is monochromatic. Also, the only triangle of $\Xi$ is $\Xi[x, y, z]$. 
    \item $\Xi$ is 3-colorable and contains a subset $X = \{w, x, y, z\} \subseteq V(\Xi)$ such that $N(X) = X$, and $\Xi[X]\cong\trianglePlusIsolated$, $\Xi[X]\cong\SuzukiThirtyTwoElusiveGraph$, or $\Xi[X]\cong\hamSandwich$. Furthermore, all triangles are contained in $\Xi[X]$.
\end{enumerate}
\end{theorem}
\begin{proof}
    We begin with the forwards direction. Note that $\pgc(G)$ is 3-colorable by \cite[Theorem 3.6]{2023REU} and \thref{kColorings}. If $\pgc(G)$ is triangle-free, then (1) holds, so we may assume $\pgc(G)$ contains a triangle. Thus, by \thref{solvableTheorem1}, $G$ is strictly $T$-solvable. By \cite[Theorem 3.4]{2023REU}, there exist no edges $p-q\in\pgc(K)$ for $p\in\{a, c, d\},q\in\pi(K)\setminus\pi(T)$ for any strictly $T$-solvable $K$, and thus by \thref{disconnect}, there exist no edges $p-q\in\pgc(G)$ for $p\in\{a, c, d\},q\in\pi(G)\setminus\pi(T)$. If there existed a triangle not contained in $\pgc(G)[\pi(T)]$, let $\Delta$ be a finite subgraph of $\pgc(G)$ containing that triangle and $\pi(T)$. Then, there exists a finite strictly $T$-solvable group $H$ with $\pgc(H) = \Delta$, so by \cite[Lemma 2.3.5]{2023REU}, all triangles of $\pgc(H)$ are contained in $\pgc(H)[\pi(T)] = \pgc(G)[\pi(T)]$, a contradiction, so all triangles of $\pgc(G)$ are contained in $\pgc(G)[\pi(T)]$. We now consider two cases: The case where there is some $p\in\pi(G)\setminus\pi(T)$ such that $2-p\in\pgc(G)$ and the case where no such $p$ exists.

\hfill\break
\textbf{Case 1:} Suppose such a $p$ exists. Let $\Delta$ be a finite subgraph containing $2-p$ and $\pi(T)$. Then, by \thref{hallAnalog}, there exists a finite strictly $T$-solvable group $H$ with $\pgc(H) = \Delta$. By \cite[Lemma 1.1]{psl}, $\pgc(G)[\pi(T)]=\pgc(H)[\pi(T)]\subseteq\notIsolatedVertexWhiteBad$ where the white vertex is 2. 

If $\pgc(G)[\pi(T)]$ does not contain a triangle, (1) is satisfied. If it does contain a triangle, label vertices $d, c, a$ with $x,y, z$, label 2 with $w$, and let $X = \{w,x,y,z\}$. Because $\pgc(G)[\pi(T)]$ is isomorphic to a subgraph of \notIsolatedVertexWhiteBad and there exists a triangle in $\pgc(G)[\pi(T)], \pgc(G)[\pi(T)]\cong \notIsolatedVertexWhiteBad$. By \thref{oneColoringComplete} and \cite[Theorem 3.6]{psl}, $N(\pi(T))$ is monochromatic in some 3-coloring of $\pgc(G)$. Notice that (2) is satisfied.

\hfill\break
\textbf{Case 2:} We finally consider the case where there is no $\in\pi(G)\setminus\pi(T)$ such that $2-p\in\pgc(G)$. Then, $N[\pi(T)] = \pi(T)$. Let $X = \pi(T)$. Let $\Delta = \pgc(G)[\pi(T)]$. Then, by \thref{hallAnalog}, there exists a finite strictly $T$-solvable group $H$ with $\pgc(H) = \Delta = \pgc(G)[\pi(T)]$, so by \cite[Proposition 3.5]{psl}, $\pgc(G)[\pi(T)]$ is isomorphic to one of \trianglePlusIsolated, \SuzukiThirtyTwoElusiveGraph, or \hamSandwich. Thus, (3) is satisfied.

    We now turn to the backwards direction. If $\Xi$ satisfies (1), then there exists a solvable torsion group with the finite Frobenius property realizing (1) by \thref{labelingsolvable}.

    Now suppose that $\Xi$ satisfies (2). Color $\Xi$ with the colors $\mathcal{O},\mathcal{D},\mathcal{I}$ such that all of $N(X)$ is colored $\mathcal{I}$, $w$ and $x$ are colored $\mathcal{O}$, $y$ is colored $\mathcal{I}$, and $z$ is colored $\mathcal{D}$. Assign $w$ to 2, $x$ to $d$, $y$ to $a$, and $z$ to $c$ to obtain a graph isomorphism from $\pgc(C_2.T)$ to $\Xi[X]$  by \cite[Proposition 3.5]{psl}. We know that there are irreducible representations of $C_2.T$ such that elements of orders $a, c, $ and $d$ have fixed points while elements of order 2 have no fixed points in the representation by \thref{fpinfoPSL227}. Then, by \thref{generalized236}, $\Xi$ is the prime graph complement of a $T$-solvable group.

    Finally, if $\Xi$ satisfies (3), there exists a finite $T$-solvable group $G_0$ such that $\pgc(G_0)\cong\Xi[X]$ by \cite[Proposition 3.5]{psl}. By \thref{labelingsolvable}, there exists a solvable torsion group $N$ with the finite Frobenius property such that $\pi(N)\cap\pi(G_0) = \varnothing$ and $N$ realizes $\Xi\setminus X$. Then, $G_0\times N$ realizes $\Xi$.
\end{proof}
\begin{corollary}
    Let $T$ be one of $\operatorname{PSL}(2, 3^3)$ or $\operatorname{PSL}(2, 7^2)$. The prime graph complements of (locally) torsion $T$-solvable groups with the finite Frobenius property have been completely classified. 
\end{corollary}
\begin{proof}
    See \thref{generalThm} and \cite[Section 3]{psl}.
\end{proof}
The next result is a generalization of, and parts of it follow similarly to \cite[Theorem 3.6]{psl} (under the generalizations of certain lemmas). When appropriate, we use the wording and notation from the proof of \cite[Theorem 5.10]{psl}. We also used the figure from \cite{psl}.
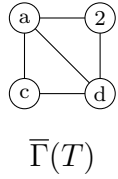
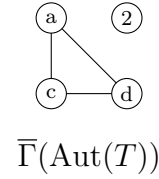
\begin{figure}[H]
    \centering
    \begin{minipage}{.5\textwidth}
        \centering
        \begin{tikzpicture}
            \coordinate (A) at (0, 1);   
            \coordinate (B) at (1, 1);   
            \coordinate (C) at (1, 0);   
            \coordinate (D) at (0, 0);   
            \draw (B) -- (C);
            \draw (C) -- (D);
            \draw (A) -- (D);
            \draw (A) -- (B);
            \draw (A) -- (C);
            \foreach \point in {A, B, C, D}
                \draw[draw, fill=white] (\point) circle [radius=0.2cm];
            \foreach \point/\name in {(A)/a, (B)/2, (C)/d, (D)/c}
                \node at \point {\scriptsize\name};
        \end{tikzpicture}
        \caption*{$\pgc(T)$}
    \end{minipage}%
    \begin{minipage}{.5\textwidth}
        \centering
        \begin{tikzpicture}
            \coordinate (A) at (0, 1);   
            \coordinate (B) at (1, 1);   
            \coordinate (C) at (1, 0);   
            \coordinate (D) at (0, 0);
            \draw (C) -- (D);
            \draw (A) -- (D);
            \draw (A) -- (C);
            \foreach \point in {A, B, C, D}
                \draw[draw, fill=white] (\point) circle [radius=0.2cm];
            \foreach \point/\name in {(A)/a, (B)/2, (C)/d, (D)/c}
                \node at \point {\scriptsize\name};
        \end{tikzpicture}
        \caption*{$\pgc(\Aut(T))$}
    \end{minipage}
\end{figure}

Let $T$ be a finite nonabelian simple group such that $\pgc(T)$ and $\pgc(\operatorname{Aut}(T))$ have the above prime graph complements. Suppose $T$ satisifes the below criteria:
\begin{enumerate}
    \item The Schur multiplier of $T$ is 2.
    \item The fixed point information of $T$ matches that of \thref{fpInfoPSL225}.
    \item The graph \szthirtytwoTrickyGraph is realized by $P\rtimes T$ for some $p$-group $P$ such that $p\in\pi(T)$.
    \item The Sylow-2 subgroups of $T$ are dihedral groups.
    \item The Sylow-$a$ subgroups of $T$ do not satisfy the Frobenius criterion. 
\end{enumerate}
\begin{fact}\thlabel{fpInfoPSL225}
    \begin{tabular}[h]{|c|c|}
    \hline
        Fixed Point Information &  \\
        \hline
        $2.T$ & $[ [ 2, a, c, d ], [ a, c ], [ a, c, d ] ]$ \\
        \hline
        $T$ & $[ [2, a, c, d]]$ \\
        \hline
    \end{tabular}
\end{fact}

\begin{theorem}\thlabel{psl225thm}
    Let $\Xi$ be a countable locally finite graph. Then $\Xi$ is isomorphic to the prime graph complements of a (locally) torsion $T$-solvable group with the finite Frobenius property if and only if one of the following is true:
    \begin{itemize}
        \item $\Xi$ is triangle-free and 3-colorable.
        \item $\Xi$ contains a subset $X = \{w, x, y, z\}\subseteq V(\Xi)$ such that $N[X] = X$ and $\Xi[X]\cong\twoIsolatedpsl$, $\Xi[X]\cong\hamSandwich$, or $\Xi[X] \cong \szthirtytwoTrickyGraph$. Furthermore, $\Xi\setminus X$ is triangle-free and three colorable.
        \item $\Xi$ contains a subset $X = \{w, x, y, z\} \subseteq V(\Xi)$ such that in some 3-coloring of $\Xi$ the closed neighborhood $N(X) \setminus X$ is monochromatic. Also, $\Xi$ contains exactly one triangle $\{x,y,z\}$ and $N(w) \cap \{x,y,z\} = \varnothing$. All edges incident to $x \text{ and } y$ are within the triangle and $N(z)\setminus\{x,y\} \subseteq N(w)$.
        \item $\Xi$ contains a subset $X = \{w, x, y, z\} \subseteq V(\Xi)$ such that in some 3-coloring of $\Xi$ the closed neighborhood $N(X) \setminus X$ is monochromatic. Also, $\Xi$ contains exactly one triangle $\{x,y,z\}$ and $N(w) \cap \{x,y,z\} = \varnothing$. All edges incident to $w,x,y$ are contained in $\Xi[X]$.
    \end{itemize}
\end{theorem}
\begin{proof}
    We first consider the forwards direction. Let $G$ be a $T$-solvable (locally) torsion  group with the finite Frobenius property. By \thref{solvableTheorem1} and \cite[Theorem 5.10]{psl}, $\pgc(G)$ is 3-colorable. If $\pgc(G)$ is also triangle free, then (1) holds, so we may assume that $\pgc(G)$ contains at least one triangle. By \thref{solvableTheorem1}, $G$ is strictly $T$-solvable. All triangles are contained in $\pgc(G)[\pi(T)]$ by \cite[Lemma 5.7]{psl} and \thref{triangleJail}. By the proof of \cite[Theorem 5.10]{2023REU} and \thref{oneColoringComplete}, if $N(\pi(T))$ is nonempty, it is monochromatic in some 3-coloring of $\Xi$. We consider three cases: The case were there exist no $r\in\pi(T)$ and $p\in\pgc(G)\setminus\pi(T)$ with $r-p\in\pgc(G)$, the case where there exist $p\in\pi(G)\setminus\pi(T)$ with $2-p\in\pgc(G)$, and the case where there exist $p\in\pi(G)\setminus\pi(T)$ with $d-p\in \pgc(G)$ but no $q\in\pi(G)\setminus\pi(T)$ with $2-q\in\pgc(G)$.

    \hfill\break
    \textbf{Case 1:} Let $X = \pi(T)$. Then, by the fact that $\pgc(T)\cong\hamSandwich$, (2) is satisfied.

    \hfill\break
    \textbf{Case 2:} For contradiction, suppose $2-q\in\pgc(G)$ for some $q\in\pi(T)\setminus\{2\}$. Let $\Delta$ be the finite subgraph containing $\pgc(G)[\pi(T)]$ and some edge $2-p$ for a $p\in\pi(G)\setminus\pi(T)$. By \thref{hallAnalog}, there exists some finite strictly $T$-solvable $H$ with $\pgc(H) = \Delta$, a contradiction to \cite[Lemma 1.1]{psl}. Thus, $\{a, c, d\}\cap N(2) = \varnothing$. Let $w = 2, x = a, y = c, d = z$. For contradiction, suppose that $N(z)\setminus\{x,y\}\not\subseteq N(w)$. Then, there exists an edge $w-q\in\pgc(G)$ such that $z-q\not\in\pgc(G)$ for some $q\in\pi(G)\setminus\pi(T)$. Let $\Delta\subseteq\pgc(G)$ be the finite subgraph of $\pgc(G)$ containing $\pgc(G)[\pi(T)]$ and $w-q$. Then, by \thref{hallAnalog}, $\Delta$ is realized by a finite strictly $T$-solvable group, a contradiction to \cite[Lemma 5.5]{psl}. Thus, (3) is satisfied.

    \hfill\break
    \textbf{Case 3:} Let $w = 2, x = a, y = c, d = z$. By assumption, all edges incident to $\{w, x, y\}$ are contained in $\pgc(G)[\pi(T)]$. By assumption, there exists some prime $q\in\pi(G)\setminus\pi(T)$ with $d-q\in\pgc(G)$. Let $\Delta = \pgc(G)[\pi(T)\cup\{q\}]$. By \thref{hallAnalog}, there exists a finite strictly $T$-solvable group $H$ with  $\pgc(H) = \Delta$. The argument given in case 3 of the proof of \cite[Theorem 5.10]{psl} shows that $N(2)\cap\{a, c, d\}=\varnothing$. We showed earlier that $\pgc(G)[\pi(T)]$ contains a triangle, so that triangle must be $\{a, c, d\}$. Thus, (4) is satisfied.

    We now turn to the backwards direction. If $\Xi$ satisfies (1), then by \thref{solvableTheorem1}, there exists a solvable torsion group with the finite Frobenius property realizing $\Xi$. Thus, we only consider the graphs satisfying (2)-(4).

    Suppose $\Xi$ satisfies (2). By \cite[Proposition 5.4]{2023REU}, there exists a $T$-solvable group $E$ such that $\pgc(E)\cong \Xi[X]$. By \thref{solvableTheorem1}, there is a solvable group $N$ such that $\pi(N)\cap\pi(E) = \varnothing$ and $\pgc(N)\cong\Xi\setminus X$. Thus, $\pgc(E\times N)\cong \Xi$ and $E\times N$ is a $T$-solvable group with the finite Frobenius property.

    Suppose $\Xi$ satisfies (3), let $X = \{w, x, y, z\}$. Let $E \cong2.T$ be the perfect central extension of $T$ by $C_2$. Notice that $\pgc(E)\cong \Xi[X]$ by the labeling $w\mapsto 2, x\mapsto c, y\mapsto a, z\mapsto d$. For each $v\in N(X)$, one of the following holds:
    \begin{enumerate}
        \item[(1)] $v$ is adjacent to $w$ but not $x, y, $ and $z$.
        \item[(2)] $v$ is adjacent to $w$ and $z$ but not $x$ and $y$.
    \end{enumerate}
    By \thref{fpInfoPSL225}, we may select the representation where $a, c, $ and $d$ have fixed points but 2 does not if (1) holds, and select the representation where $c$ and $a$ have fixed points but 2 and $d$ do not. We may apply \thref{generalized236} to get a torsion $T$-solvable group with the finite Frobenius property realizing $\Xi$.

    Finally, consider the case where $\Xi$ satisfies (4). Let $\Xi'$ be the graph with all edges of $\Xi$ with the additional edges $w-v\in \Xi'$ for all $v\in V(\Xi)\setminus X$ such that $z-v\in \Xi$. Notice that $\Xi'$ satisfies (3), so there exists a $T$-solvable group $H$ with $\pgc(H)\cong\Xi'$. Then, $H\times C_2$ realizes $\Xi$.
\end{proof}
\begin{corollary}
Let $T$ be one of $\operatorname{PSL}(2, 5^2)$ or $\operatorname{PSL}(2, 3^4)$. The prime graph complements of (locally) torsion $T$-solvable groups with the finite Frobenius property have been completely classified.     
\end{corollary}
\begin{proof}
    See \thref{psl225thm} and \cite[Section 5]{psl}.
\end{proof}
\begin{theorem}\thlabel{sz8}
    Let $\Xi$ be a countable locally finite graph. $\Xi$ is realized by a (locally) torsion $\operatorname{Sz}(8)$-solvable group with the finite Frobenius property if and only if one of the following holds:
    \begin{enumerate}
        \item $\Xi$ is triangle-free and 3-colorable.
        \item $\Xi$ is a disjoint union of two graphs $A$ and $B$, where $A$ is any graph on 4 vertices that contains a triangle except for \hamSandwich, and $B$ is triangle-free and 3-colorable.
        \item There exists a subset $X = \{a, b, c,d\}\subseteq V(\Xi)$ for which all outgoing edges are to a single vertex $e$. Furthermore, the subgraph induced by $X\cup\{e\}$ is isomorphic as a rooted graph (with $e$ acting as the root), to one of the graphs in $\mathcal{F}_4$, listed below:
        \begin{center}
        \house, \northStar, \codart, \happyFace, \balloon, \dart, \triangleWithShortTail.
    \end{center}
    Additionally, $\Xi$ must be triangle-free and 3-colorable.
    \end{enumerate}
\end{theorem}
\begin{proof}
    Let $G$ be an $\operatorname{Sz}(8)$-solvable group. If (1) holds, then the claim is satisfied, so assume it does not. Then, by \thref{solvableTheorem1}, $G$ is strictly $\operatorname{Sz}(8)$-solvable. Using \thref{kColoringsMinusSubgrap}, we can see that $\pgc(G)\setminus\pi(\operatorname{Sz}(8))$ is 3-colorable. By \thref{disconnect} and \cite[Theorem 6.6]{Suz}, $r-p\not\in\pgc(G)$ for all $r\in\pi(\operatorname{Sz}(8)), p\in\pi(G)\setminus(\pi(\operatorname{Sz}(8))\cup\{3\})$. We consider two cases:

    \hfill\break
    \textbf{Case 1:} If $3\not\in\pi(G)$:

     If $\pgc(G)$ contains a triangle, notice that there exists a finite subgraph $\Delta\subseteq \pgc(G)$ such that $\Delta$ contains that triangle and $\pi(\operatorname{Sz}(8))$. By \thref{hallAnalogLocallyTorsion}, there exists a finite $\operatorname{Sz}(8)$-solvable group $H$ with $\pgc(H) = \Delta$. By \cite[Theorem 6.20]{Suz}, the triangle is contained in $\pgc(G)[\pi(\operatorname{Sz}(8))]$ and $\pgc(G)[\pi(\operatorname{Sz}(8))]$ is not \hamSandwich. Thus, (2) holds.

    \hfill\break
    \textbf{Case 2:} If $3\in\pi(G)$:

    If $\pgc(G)$ contains a triangle, notice that there exists a finite subgraph $\Delta\subseteq \pgc(G)$ such that $\Delta$ contains that triangle and $\pi(\operatorname{Aut}(\operatorname{Sz}(8)))$. By \thref{hallAnalogLocallyTorsion}, there exists a finite $\operatorname{Sz}(8)$-solvable group $H$ with $\pgc(H) = \Delta$. By \cite[Theorem 6.20]{Suz}, the triangle is contained in $\pgc(H)[\pi(\operatorname{Aut}(\operatorname{Sz}(8)))] = \pgc(G)[\pi(\operatorname{Aut}(\operatorname{Sz}(8)))]$. By \cite[Lemma 6.4]{Suz}, if $3-p\not\in\pgc(G)$ for any $p\in\pi(\operatorname{Sz}(8))$, then (2) is satisfied. If $3-p\in\pgc(G)$ for some $p\in\pi(\operatorname{Sz}(8))$, then by \cite[Proposition 6.19]{Suz}, $\pgc(H)[\pi(\operatorname{Aut}(\operatorname{Sz}(8)))]$ realizes one of the graphs listed in (3), so (3) is satisfied.

    We now consider the backwards direction. If (1) is satisfied, then by \thref{solvableTheorem1}, there exists a solvable $G$ realizing $\Xi$. If (2) is satisfied, by \cite[Lemma 6.5]{Suz}, there exists a finite $\operatorname{Sz}(8)$-solvable group $H$ realizing $\Xi[\{a, b, c, d\}]$, so by  \thref{solvableTheorem1} we can construct a solvable torsion $N$ such that $\pi(N)\cap\pi(H) = \varnothing$ and $\pgc(N)$ realizes $\Xi\setminus X$, so $H\times N$ realizes $\Xi$. If (3) is satisfied, notice that there exists an $\operatorname{Sz}(8)$-solvable group $H$ realizing $\Xi[X\cup\{e\}]$ where $e$ is identified with 3 by \cite[Proposition 6.19]{Suz}, and identify $X$ with the appropriate primes in $\pgc(H)$. Then, we may modify the labeling in \thref{labelingsolvable} so that when we label $\Xi\setminus(X\cup\{e\})$, no vertex is labeled with a prime in $\pi(\operatorname{Aut}(\operatorname{Sz}(8)))$. Let $\Lambda_0\subset\Lambda_1\subset\cdots$ be a sequence of finite subgraphs of $\Xi$ such that $\Xi[V(\Lambda_i)] = \Lambda_i$ for all $i$, $X\cup\{e\}\subseteq\Lambda_i$ for all $i$, and $\Xi = \bigcup_{i\in\mathbb{N}}\Lambda_i$. A proof similar to that of \thref{generalized236} that uses the techniques of \cite[Theorem 6.20]{Suz} constructs a sequence of groups $\{H_n\}_{n\in\mathbb{N}}$ such that for each $H_n$, $H_n$ has a subnormal series $1\unlhd N_{1,n}\unlhd N_{2,n}\unlhd N_{3,n}\unlhd N_{4,n}\unlhd H_n$ such that $ N_{1,n}$ is abelian, $N_{2,n}/ N_{1,n}$ is abelian, $N_{3,n}/N_{2,n}$ is abelian, $N_{4,n}/N_{3,n}$ is isomorphic to an $\operatorname{Sz}(8)$-solvable group $F$ with $\pi(F) = \pi(\operatorname{Sz}(8))$, and $H_n/N_{4,n}$ is isomorphic to $C_3$; along with a sequence of embeddings $\{\varphi_n\}_{n\in\mathbb{N}}$ such that $\varphi_i:H_i\to H_{i+1}$ and $\varphi(x)\in N_{k,n+1}$ if and only if $x\in N_{k,n}$. Let $1=M_0\unlhd M_1\unlhd\cdots\unlhd M_r = F$ be a subnormal series such that $M_{i+1}/M_i$ is either abelian or isomorphic to $\operatorname{Sz}(8)$ for each $i$. Notice that, by choice of morphisms and construction of $F$, the subnormal series of $H_n$ can be refined further so that:
    \begin{align*}
        1 \unlhd N_{1,n}\unlhd N_{2,n}\unlhd N_{3,n}\unlhd \cdots\unlhd N_{r + 3,n}\unlhd H_n\\
        \text{ where } N_{i+1,n}/N_{i,n}\text{ is abelian for }1\leq i< 3,\\
        N_{i+1,n}/N_{i,n}\text{ is isomorphic to }M_{i-2,n}/M_{i-3,n}\text{ for }3\leq i\leq r+2\text{ and }H_n/N_{r+3,n}\cong C_3
    \end{align*}
    and such that $\varphi(x)\in N_{k,n+1}$ if and only if $x\in N_{k,n}$. Define:
    \begin{align*}
        f_{i,j}(x) = \varphi_{j-1}\circ\cdots\circ\varphi_i(x)\text{ for all }j>i\\
        f_{i,i}(x) = x\text{ for all }i\in\mathbb{N}.
    \end{align*}
    Notice that $\langle H_i, f_{i,j}\rangle_{\mathbb{N}}$ is a directed system of groups. By an argument similar to that of \thref{generalized236}, $\pgc(\varinjlim H_i)$ realizes $\Xi$, and $\varinjlim H_i$ is a torsion group with the finite Frobenius property.
\end{proof}
We include the statement for the generalization of the classification of the prime graphs of $\operatorname{PSL}(2, 2^5)$-solvable groups (\cite[Theorem 7.19]{Suz}) as its proof follows very similarly to that of \thref{sz8}.
\begin{theorem}
    Let $\Xi$ be a countable locally finite graph. $\Xi$ is realized by a (locally) torsion $\operatorname{PSL}(2, 2^5)$-solvable group with the finite Frobenius property if and only if one of the following holds:
    \begin{enumerate}
        \item $\Xi$ is triangle-free and 3-colorable.
        \item $\Xi$ is a disjoint union of graphs $A$ and $B$, where $A$ is any graph on four vertices that contains a triangle except for \fourVertexCompleteGraph, and $B$ is triangle-free and 3-colorable.
        \item There exists a subset $X = \{a, b, c, d\}\subseteq V(\Xi)$ for which all edges to $\Xi\setminus X$ are to $e$. Additionally, $\Xi[X\cup\{e\}]$ is isomorphic to one of the following rooted graphs (where $e$ corresponds to the root):
        \begin{center}
            \fan, \northStar, \triangleWithShortTail, \codart, \happyFace, \balloon, \dart, \cricket.
        \end{center}
        Additionally, $\Xi\setminus X$ is triangle-free and 3-colorable.
    \end{enumerate}
\end{theorem}

\section{Outlook}\label{outlook}
There are still many open questions pertaining to the study of prime graphs of infinite groups. One possible direction for future work would be to remove, in some capacity, the requirement of the finite Frobenius property from the classification of the prime graphs of solvable torsion groups with the finite Frobenius property. Another direction would be to consider other reasonably well-behaved classes of groups beyond that of torsion groups in the classifications of $T$-solvable groups. Finally, one could also continue the classification of prime graphs of $T$-solvable (locally) torsion  (and locally torsion) groups.
\section*{Acknowledgments}
Thanks to Thomas Keller and Anand Pillay for their helpful comments on this draft.

\end{document}